\documentclass[11pt,a4paper]{article}
\usepackage[utf8]{inputenc}
\usepackage{amsmath,amsfonts,amssymb,amsthm,dsfont,enumitem,leftidx,tikz-cd,bbm,subcaption}
\usepackage[scr=boondoxo]{mathalfa}
\usepackage{graphicx}
\providecommand{\keywords}[1]{\textbf{\textit{Keywords---}} #1}
\providecommand{\classification}[1]{\textbf{\textit{2010 Subject classification index
---}} #1}

\author{Mart\'{\i}n de Diego, David; Sato Mart\'in de Almagro, Rodrigo T.}
\newtheorem{theorem}{Theorem}[section]
\newtheorem{proposition}{Proposition}[section]
\newtheorem{corollary}{Corollary}[theorem]

\theoremstyle{remark}
\newtheorem*{remark}{Remark}

\theoremstyle{definition}
\newtheorem{definition}{Definition}[section]

\begin{document}
\title{High-order geometric methods for nonholonomic mechanical systems}
\author{ D. Mart\'{\i}n de Diego, R. T. Sato Mart{\'\i}n de Almagro
	\\[2mm]
	{\small  Instituto de Ciencias Matem\'aticas (CSIC-UAM-UC3M-UCM)} \\
	{\small C/Nicol\'as Cabrera 13-15, 28049 Madrid, Spain}}

\date{\today}

		\maketitle
			

\classification{ 
 49Mxx, 65L80, 70G45, 70Hxx.}

		\begin{abstract}
In the last two decades, significant effort has been put in understanding and designing so-called structure-preserving numerical methods for the simulation of mechanical systems. Geometric integrators attempt to preserve the geometry associated to the original system as much as possible, such as the structure of the configuration space, the energy behaviour, preservation of constants of the motion and of constraints or other structures associated to the continuous system (symplecticity, Poisson structure...). In this article, we develop high-order geometric (or pseudo-variational) integrators for nonholonomic systems, i.e., mechanical systems subjected to constraint functions which are, roughly speaking, functions on velocities that are not derivable from position constraints. These systems realize rolling or certain kinds of sliding contact and are important for describing different classes of vehicles.
		\end{abstract}
		
		\tableofcontents
\title{}
\section{Introduction}\label{sec:Intro}
During the last twenty-five years, there has been a continuous interest in numerical methods for differential equations preserving some of its relevant geometric features. These methods constitute a new discipline called  Geometric integration \cite{sanz-serna} and now it is possible to find hundred of papers devoted to this topic and in multiple applications in applied sciences (see as main monogrographs \cite{feng,serna,Leim-reich,hairer,blanes} among others).

These methods arise in contrast to classical integrators which are mainly focused on solving numerically generic ordinary differential equations on an euclidean space. But in many cases of interest in mathematics, physics and engineering, we know some important qualitative and geometric features of the continuous system to be simulated, as for instance,  preservation of the energy, linear or angular momenta, symplectic form, volume form, Poisson structures, etc. and as we will see in this paper, also the preservation of constraints on configuration or phase spaces.

A very relevant subclass of structure-preserving integrators are variational integrators (see \cite{marsden-west,leoshin}). This type of geometric integrators is based on a discretization of the corresponding variational principles and admits very interesting interpretation and can be extended to forced systems, higher order Lagrangian theories, field theories... The idea behind these is to replace the action integral associated to a continuous Lagrangian by a discrete action sum and extremize it over all the sequences of points with fixed-end point conditions. Assuming a regularity condition, the derived numerical method is guaranteed to be symplectic or Poisson preserving and also, provided symmetry invariance, it is possible to obtain preservation of the associated constants of the motion. These methods can be extended naturally to systems subjected to holonomic constraints \cite{marsden-west,MR1862825}.

Most important is that for these methods we have the variational error analysis framework to rigorously analyse their order (see \cite{marsden-west,patrick-cuell} and also \cite{rodrigo} for forced systems). This tells us that for a regular Lagrangian system $L: TQ\rightarrow {\mathbb R}$ we can relate the error of the derived numerical method with that obtained by measuring the difference between its corresponding discrete Lagrangian and the exact discrete Lagrangian, derived from the continuous solutions of the Euler-Lagrange equations.

To construct high-order variational integrators it is natural to consider the class of Runge-Kutta (RK) methods that arises when we discretize a Hamilton-Pontryagin action \cite{YoshimuraMarsden} using a RK scheme and take variations. The resulting methods are known as variational partitioned Runge-Kutta (VPRK) methods (see \cite{hairer, marsden-west,campos01}).

An important class of problems not covered by the above developments are related with problems subjected to nonholonomic constraints which are, roughly speaking, constraints depending on the velocities describing the mechanical system. Great interest has been shown recently in this class of systems due to its applications to  control theory and robotics \cite{bloch}. Traditionally, the equations of motion for nonholonomic mechanics are derived from the Lagrange-d’Alembert principle, which is non-variational since it restricts the set of infinitesimal variations (or constrained forces) and these variations are not derivable from constrained curves in the space of admissible curves satisfying the boundary conditions. Therefore, we do not have at our disposal any variational error analysis result to construct high-order geometric integrators for nonholonomic systems. Thus it has been necessary to prove new results about local and global convergence for the methods that we propose in this paper. 

The paper is structured as follows. In section \ref{sec:Holonomic systems} we give a short introduction to Lagrangian and Hamiltonian mechanics to keep the paper mostly self-contained and to fix the notation throughout. Section \ref{sec:noholonomo} is devoted to the case of nonholonomic systems describing both the Lagrangian and Hamiltonian formalisms for constrained systems.
In section \ref{sec:var_integrators} we introduce discrete mechanics, variational integrators and, in particular, partioned Runge-Kutta methods derived using a discrete Hamilton-Pontryagin action. There it is shown how these methods can be extended to holonomic constrained mechanical systems.
The most important results of the paper are given in section \ref{sec:discrete-nonholonomic} where a new class of  nonholonomic pseudo-variational partitioned Runge-Kutta integrators are proposed of arbitrary order. In section 
\ref{sec:liegroups} we extend the previous results to the case of Lie group integrators using a retraction map to derive our methods.
Finally, in section \ref{sec:examples} we study several nonholonomic examples using our methods. In particular, a nonholonomic particle in an harmonic potential, a pendulum-driven continuous variable transmission (CVT), a fully chaotic nonholonomic system, the nonholonomic vertical disc (unicycle) with an elastic spring and the  nonholonomic ball on a turntable. Several appendices are added at the end of the paper to introduce the order conditions in numerical methods and to summarize the numerical methods derived on the paper to facilitate their  use in other examples.

\section{Lagrangian and Hamiltonian systems}\label{sec:Holonomic systems}
\subsection{Lagrangian description}
Let $Q$ be a smooth manifold of $\dim Q = n$ and let $TQ$ denote its tangent bundle with canonical projection $\tau_Q: TQ\rightarrow Q$. Let $L: TQ \rightarrow \mathbb{R}$ be a $C^2$ function. We say that the pair $(Q,L)$ form a \emph{Lagrangian system} where $Q$ is denoted as the configuration manifold, $TQ$ is the velocity phase space and $L$ the Lagrangian function of the system which determines its dynamics \cite{foundation}.
Consider local coordinates $(q^i)$ on $Q$, $i = 1, ..., n$. The corresponding fibered coordinates on $TQ$ will be denoted by $(q^i, v^i)$ and consequently $\tau_Q(q^i, v^i)=(q^i)$. 

Consider curves $c: [a, b] \subseteq \mathbb{R} \rightarrow Q$ of class $C^2$ connecting two fixed points $q_0, q_1 \in Q$ and the collection of all these curves
\begin{equation*}
C^2 (q_0, q_1, [a,b])= \left\lbrace c: [a, b] \rightarrow Q \,|\, c \in C^2([a, b]), c(a) = q_0, c(b) = q_1\right\rbrace.
\end{equation*}
whose tangent space is
\begin{align*}
T_c C^2(q_0, q_1, [a,b]) &=\\
&\hspace{-1cm}\left\lbrace X: [a, b]\rightarrow TQ \,|\, X \in C^1([a, b]), \tau_Q \circ X = c \text{ and } X(a) = X(b) = 0\right\rbrace.
\end{align*}
Given a Lagrangian $L$ we can define the \emph{action functional}: 
\begin{equation*}
\begin{array}{rrcl}
\mathcal{J}: & C^2(q_0, q_1, [a,b]) & \longrightarrow &                               \mathbb{R}\\
             &                    c &     \longmapsto & \int_a^b L(c(t), \dot{c}(t)) \mathrm{d}t
\end{array}
\end{equation*}
where $\dot{c} = \frac{\mathrm{d}c}{\mathrm{d}t}$. Together with the following variational principle, this gives us the equations of motion for our system.

\begin{definition}{\textbf{(Hamilton's principle)}}\label{def:hamiltons_principle}{ A curve $c \in C^2(q_0, q_1, [a, b])$ is a solution of the Lagrangian system defined by $L: TQ\rightarrow \mathbb{R}$ if and only if $c$ is a critical point of the functional $\mathcal{J}$, i.e. $\mathrm{d}\mathcal{J}(c)(X) = 0$, for all $X \in T_c C^2(q_0, q_1, [a,b])$.}
\end{definition} 

Using standard techniques from variational calculus, it is easy to show that the curves $c(t) = (q^i(t))$ solutions of the Lagrangian system defined by $L$ are the solutions of the following system of second order implicit  differential equations:
\begin{equation*}
D_{\mathrm{EL}} L({c})=\frac{\mathrm{d}}{\mathrm{d}t} \left( \frac{\partial L}{\partial {\dot{q}}^i}\right) - \frac{\partial L}{\partial {q}^i} = 0
\end{equation*}
which are the well-known Euler-Lagrange equations.

Let us denote by $S = \mathrm{d}q^i \otimes \frac{\partial}{\partial v^i}$ and $\Delta= v^i \, \frac{\partial}{\partial v^i}$ the vertical endomorphism and the Liouville vector field on $TQ$ (see \cite{ManuelMecanica} for intrinsic definitions).

The Poincar\'e-Cartan 2-form is defined by $\omega_L = - \mathrm{d}\theta_L,  \theta_L = S^*(\mathrm{d}L)$
and the energy function $E_L = \Delta(L)-L$, which in local coordinates read as
\begin{eqnarray*}
\theta_L & = & \frac{\partial L}{\partial v^i} \mathrm{d} q^i\\
\omega_L & = & \mathrm{d}q^i \wedge \mathrm{d}\left( \frac{\partial L}{\partial v^i}\right)\\
E_L & = & v^i \frac{\partial L}{\partial v^i} - L (q, v) .
\end{eqnarray*}
Here $S^*$ denotes the adjoint operator of $S$. We may construct the transformation $\mathbb{F}L: TQ \rightarrow T^*Q$, called \emph{Legendre transform} or \emph{fibre derivative}, where $\langle \mathbb{F}L(v_q), w_q \rangle={\frac{\mathrm{d}}{\mathrm{d}t}\big |_{ t=0}L(v_q + t w_q)}$.
In coordinates, $\mathbb{F}L(q^i, v^i)=(q^i, \frac{\partial L}{\partial v^i}(q,v))$. 
We say that the Lagrangian is regular if $\mathbb{F}L$ is a local diffeomorphism, which in local coordinates is equivalent to the regularity of the Hessian matrix whose entries are:
\begin{equation*}
\left(g_{L}\right)_{ij} = \frac{\partial^2 L}{\partial v^i \partial v^j}.
\end{equation*}
In this case, $\omega_L$ is a symplectic form on $TQ$. Observe that in this case, the Euler-Lagrange equations are written as a system of explicit second order differential equations We will assume for the rest of the paper that $L$ is regular.

The Euler-Lagrange equations are geometrically encoded as the equations for the flow of the vector field $X_{E_L}$:
\begin{equation*}
\imath_{X_{E_L}}\omega_L = \mathrm{d}E_L.
\end{equation*}

\subsection{Hamiltonian description}
The cotangent bundle $T^*Q$ of a differentiable manifold $Q$ is equipped with a canonical exact symplectic structure $\omega_Q = -\mathrm{d}\theta_Q$, where $\theta_Q$ is the canonical 1-form on $T^*Q$ defined point-wise by
\begin{equation*}
(\theta_Q)_{\alpha_q}(X_{\alpha_q}) = \langle \alpha_q, T_{\alpha_q}\pi_Q(X_{\alpha_q})\rangle
\end{equation*}
where $X_{\alpha_q}\in T_{\alpha_q}T^*Q$, $\alpha_q\in T_q^*Q$.

In canonical bundle coordinates these become
\begin{eqnarray*}
\theta_Q &= p_i\, \mathrm{d}q^i\; ,\ 
\omega_Q &= \mathrm{d}q^i\wedge \mathrm{d}p_i\; .
\end{eqnarray*}

Given a Hamiltonian function $H: T^*Q \rightarrow \mathbb{R}$ we define the Hamiltonian vector field
\begin{equation*}
\imath_{X_H}\omega_Q = \mathrm{d}H
\end{equation*}
which can be written locally as
\begin{equation*}
X_H = \frac{\partial H}{\partial p_i}\frac{\partial}{\partial q^i} - \frac{\partial H}{\partial q^i}\frac{\partial}{\partial p_i}.
\end{equation*}
Its integral curves are determined by the Hamilton's equations: 
\begin{eqnarray*}
\frac{\mathrm{d}q^i}{\mathrm{d}t}&=&\frac{\partial H}{\partial p_i}\; ,\\
\frac{\mathrm{d}p_i}{\mathrm{d}t}&=&-\frac{\partial H}{\partial q^i}\; .
\end{eqnarray*}
using canonical coordinates $(q^i, p_i)\in T^*Q$.

Given a function $H: T^*Q\rightarrow {\mathbb R}$, a Hamiltonian function, we may construct the transformation $\mathbb{F}H: T^*Q\rightarrow TQ$ where $\langle \beta_q, \mathbb{F}H(\alpha_q)\rangle={\frac{\mathrm{d}}{\mathrm{d}t}\big |_{ t=0}H(\alpha_q+t\beta_q)}$.
In coordinates, $\mathbb{F}H(q^i, p_i)=(q^i, \frac{\partial H}{\partial p_i}(q,p))$. 
We say that the Hamiltonian is regular if $\mathbb{F}H$ is a local diffeomorphism, which in local coordinates is equivalent to the regularity of the Hessian matrix whose entries are:
\begin{equation}
g_H^{i j} = \frac{\partial^2 H}{\partial p_i\partial p_j}. \label{eq:metric_on_cotangent_bundle}
\end{equation}

If we have a regular Lagrangian problem $L$, we can define an associated Hamiltonian problem as $H = E_L \circ \left(\mathbb{F}L\right)^{-1}$. Additionally one gets that $\theta_Q = \left(\mathbb{F}L^{-1}\right)^*\theta_L$ and $\omega_Q = \left(\mathbb{F}L^{-1}\right)^*\omega_L$ which proves that in this case $\omega_L$ provides $TQ$ with a symplectic structure. In this particular case, $\mathbb{F}H=\mathbb{F}L^{-1}$ and also that $\mathbb{F}L_* X_{E_L}=X_H$. . 

\section{Nonholonomic mechanical systems}\label{sec:noholonomo}
\subsection{Lagrangian description}
Nonholonomic mechanics is the study of mechanical systems whose evolution is constrained depending on both its current position and velocity, more rigorously, the nonholonomic constraints are specified by a submanifold  $N \subset TQ$. This is in contrast with the holonomic case where $N \subset Q$. In most applications $N$ is a vector subbundle completely described by a non-integrable distribution $\mathcal{D}$ and so one identifies $N = \mathcal{D}$, although that need not be the case for us. Thus let us state the following:

\begin{definition}
A \textbf{nonholonomic mechanical system} is a triple $(L,Q,N)$ where $L: TQ \to \mathbb{R}$ is a $C^k$ regular Lagrangian, with $k \geq 2$, and $N \subset TQ$ such that $N \neq TX$ for some $X \subset Q$.
\end{definition}

In what follows we will assume for simplicity that $\tau_Q(N) = Q$, where $\tau_Q: TQ \to Q$. See more details in \cite{bloch,cortes02,neimarkfufaev}. 

One is commonly given a function $\Phi: TQ \to \mathbb{R}^m$, with $m = \mathrm{codim}_{TQ}(N)$ such that its null-set is $N$ (i.e. $\Phi^{-1}(0) = N$). Clearly if $i_N: N \to TQ$ then we could define a restricted Lagrangian, $L^N = L \circ i_N$, and \emph{a priori} this latter description would be the most natural for the system but that is not necessarily true. In fact, for a nonholonomic  Lagrangian system the equations of motion still rely on the complete Lagrangian $L$. Only a subset, albeit an important one, of these systems admit a complete description in terms of $L^N$, the constrained variational system or vakonomic systems (see \cite{cortes02}).

An important space that will appear later is the Chetaev bundle, $S^*(TM^0)$, where $\left(TM\right)^0 \subseteq T^*Q$ denotes the annihilator of $TM$. This is locally spanned by $S^*(\mathrm{d}\Phi)$ which can be understood a set of separated semibasic 1-forms
\begin{equation*}
\frac{\partial \Phi^a}{\partial v^i} \mathrm{d}q^i, \quad \forall a = 1,...,m.
\end{equation*}
It will always be assumed that $\Phi$ is such that $\mathrm{rank}\, S^*(\mathrm{d}\Phi) = m$ (\emph{admissibility condition}). Additionally we will assume that $L$ and $\Phi$ satisfy that the matrix whose elements are:
\begin{equation*}
C^{a b} = g_{L}^{i j} \frac{\partial \Phi^a}{\partial v^i} \frac{\partial \Phi^b}{\partial v^j},
\end{equation*}
where $g_{L}^{i j}$ are the elements of the inverse of $\left(g_{L}\right)_{i j}$, is regular (\emph{compatibility condition}) (see \cite{dLdD96b}).

Once we are given a nonholonomic mechanical system, the next thing to do is to obtain its corresponding equations of motion. Mathematically it is easier to formulate the equations of motion for the $(L, \Phi)$ system in an augmented setting through the method of Lagrange multipliers. An intrinsic derivation of the equations in $N$ (restricted setting) exist which sidesteps the use of these multipliers, but we will use the former. It is well known that the equation of motion of a nonholonomic system are not described using constrained variational calculus for $L^N$ (see \cite{lewismurray95}). 

The main departure point of nonholonomic mechanics from its holonomic counterpart is that its evolution equations are non-variational, i.e. they cannot be derived from a purely variational principle like Hamilton's principle. As we will see in a moment we will need to use \emph{Chetaev's principle} instead, which can be understood as an instance of the Lagrange-D'Alembert principle for a particular kind of constraints (linear or affine) \cite{dLdD96b,cortes02}.

Consider the submanifold $\widetilde{C}^2 (q_0, q_1, [a,b])$ of $C^2 (q_0, q_1, [a,b])$ consisting of those curves compatible with the constraint:
\begin{equation*}
\widetilde{C}^2 (q_0, q_1, [a,b])= \left\lbrace \tilde{c} \in C^2 (q_0, q_1, [a,b]) \,|\, \left(\tilde{c}(t),\dot{\tilde{c}}(t)\right) \in N\right\rbrace.
\end{equation*}

For each $\tilde{c}$ we can consider the vector subspace of $T_{\tilde{c}} C^2(q_0, q_1, [a,b])$,
\begin{align*}
\mathcal{V}^{\Phi}_{\tilde{c}} (q_0, q_1, [a,b]) &=&
\hspace{-2cm}\left\lbrace X \in T_{\tilde{c}} C^2(q_0, q_1, [a,b]) \,|\, \forall \bar{X}(t) \in T_{(\tilde{c}(t),\dot{\tilde{c}}(t))} TQ \right.\\
&&\left.\text{ s.t. } T\tau_Q\left(\bar{X}\right) = X, S^*\left(\mathrm{d}\Phi\right)(X) = 0\right\rbrace.
\end{align*}

Given a vectir field along a solution $\tilde{c}$,  $X = X^i \frac{\partial}{\partial q^i}$, then $X \in \mathcal{V}^{\Phi}_{\tilde{c}} (q_0, q_1, [a,b])$ if and only if:
\begin{equation}
\tag{variational constraint}
X^i \frac{\partial \Phi^a}{\partial \dot{q}^i}\Big|_{\tilde{c}} = 0, \quad \forall a = 1,...,m.
\label{eq:kinematic_condition}
\end{equation}

\begin{definition}{\textbf{(Chetaev's principle)}}\label{def:chetaev_principle}{ A curve $\tilde{c} \in  \tilde{C}^2(q_0, q_1, [a, b])$ is a solution of the nonholonomic Lagrangian system defined by $L: TQ\rightarrow \mathbb{R}$ and $\Phi: TQ \to \mathbb{R}^m$ if and only if $\tilde{c}$ satisfies $\mathrm{d}\mathcal{J}(\tilde{c})(X) = 0$, for all $X \in \mathcal{V}^{\Phi}_{\tilde{c}} (q_0, q_1, [a,b])$}
\end{definition}

This means that $\tilde{c}$ is a solution of the nonholonomic problem if and only if:
\begin{equation*}
D_{\mathrm{EL}} L(\tilde{c}) (X) = \left(\frac{\mathrm{d}}{\mathrm{d}t}\left(\frac{\partial L}{\partial \dot{q}^i}\right) - \frac{\partial L}{\partial q^i}\right) X^i = 0,
\end{equation*}
for all $X^i$ satisfying the \ref{eq:kinematic_condition}, i.e. $D_{\mathrm{EL}} L(\tilde{c})$ is in the Chetaev bundle. This implies that the equations of motion can be written as:
\begin{equation}
\frac{\mathrm{d}}{\mathrm{d}t}\left(\frac{\partial L}{\partial \dot{q}^i}\right) - \frac{\partial L}{\partial q^i} = \lambda_a \frac{\partial \Phi^a}{\partial \dot{q}^i}, \quad \forall i = 1,...,n;
\end{equation}
where $\lambda_1, ..., \lambda_m$ are Lagrange multipliers. These multipliers are determined by ensuring that the curve belongs to $\widetilde{C}^2(q_0, q_1, [a, b])$, i.e. imposing the constraint equations:
\begin{equation}
\Phi^a(q,\dot{q}) = 0, \quad \forall a = 1, ..., m.
\end{equation}
To ensure that the resulting system of equations for the multipliers has a unique solution it is necessary to invoke the \emph{compatibility condition} ($(C^{ab})$ is a regular matrix or see next section).

\subsection{Geometric description}
Geometrically the equations of motion derived from Chetaev's principle can be reformulated as:
\begin{equation*}
\imath_X \omega_L - \mathrm{d}E_L \in F^0_{\text{nh}},
\end{equation*}
where $X \in TN$ and $F^0_{\text{nh}} = S^*((TN)^0)$ is the Chetaev bundle. Define also the space $F^{\perp}_{\text{nh}}$ by $\flat_L\left(F^{\perp}_{\text{nh}}\right) = F^0_{\text{nh}}$ where $\flat_L: T(TQ) \to T^*(TQ)$ and $\sharp_L: T^*(TQ) \to T(TQ)$ are the musical endomorphisms defined by $\omega_L$, that is, $\flat_L (X)=i_X \omega_L$ and $\sharp_L=(\flat_L)^{-1}$. The regularity assumption about  a nonholonomic system stated above can be recast as:
\begin{itemize}
\item $\mathrm{codim}\, N = \mathrm{rank}\, F^0_{\text{nh}}$ (\emph{admissibility condition}),
\item $TN \cap F^{\perp}_{\text{nh}} = 0$ (\emph{compatibility condition}) .
\end{itemize}

A solution $X$ will be of the form $X = \xi_L + \lambda_b \zeta^b$, where $\xi_L$ is the Hamiltonian vector field of the unconstrained problem and $\zeta^a = \sharp_L (S^*(\mathrm{d}\Phi^a))$. To determine the Lagrange multipliers we need to use the tangency condition $X(\Phi) = 0$, where we get:
\begin{equation}
X(\Phi^a) = \xi_L(\Phi^a) + \lambda_b \zeta^b(\Phi^a) = 0
\label{eq:Lagrange_multipliers_determination}
\end{equation}
where $\zeta^b(\Phi^a) = C^{b a}$. In  \cite{dLdD96b,dLMdD97} it is shown that the regularity of this matrix implies the geometric compatibility condition just stated.

\subsection{Hamiltonian description}
As in the Lagrangian setting, the nonholonomic problem from the Hamiltonian point of view would be described by a function $H: T^*Q \to \mathbb{R}$ and a constraint submanifold $M \subset T^*Q$ with $\mathrm{codim}_{T^*Q}M = m$ \cite{marle98}. This manifold can be locally described by a function $\Psi: T^*Q \to \mathbb{R}^m$.

For deriving the Hamiltonian description we   need to rely on the Lagrangian description of the Chetaev bundle and pull it back using the Legendre transform $\mathbb{F}H$ induced by our Hamiltonian. This forces us to admit that $H$ must be regular.  The matrix in eq.\eqref{eq:metric_on_cotangent_bundle} allows us to define a definite quadratic form   $g_H: T^*Q \times T^*Q \to \mathbb{R}$ with corresponding  isomorphisms $\sharp_H: T^*Q \to TQ$ and $\flat_{H}: TQ \to T^*Q$. Thus the Hamiltonian Chetaev bundle $\left(F^0_{\mathrm{nh}}\right)^*$ can be locally described by $\flat_{H}\left(T\tau_{Q} \left(\sharp_{\omega} \mathrm{d}\Psi\right)\right)$, where $\sharp_{\omega}: T^*(T^*Q) \to T(T^*Q)$ is the musical isomorphism induced by the canonical structure and $T\tau_{Q}: T(T^*Q) \to TQ$.

As in the Lagrangian case, the resulting equations of motion are:
\begin{equation*}
\imath_X \omega_Q - \mathrm{d}H \in \left(F^0_{\mathrm{nh}}\right)^*
\end{equation*}
Using the notation $\left(g_{H}\right)_{i j}$ for the elements of the inverse of $g_{H}^{i j}=\frac{\partial^2 H}{\partial p_i
	\partial p_j}$, we can write these equations in local coordinates as:
\begin{eqnarray*}
\dot{q}^i &=& \frac{\partial H}{\partial p_i},\\
\dot{p}_i &=& -\frac{\partial H}{\partial q^i} + \lambda_a \left(g_{H}\right)_{i j} \frac{\partial \Psi^a}{\partial p_j}, \quad \forall i = 1, ..., n
\end{eqnarray*}
together with the constraint equations
\begin{equation*}
\Psi^a(q,p) = 0, \quad \forall a = 1, ...,m.
\end{equation*}

\section{Discrete mechanics and variational integrators}\label{sec:var_integrators}
Let $\bar{Q} = \mathbb{R} \times Q$ denote the time extended configuration space. The discrete counterpart of Lagrangian mechanics can be seen to arise from the \emph{complete solution} $\bar{S}: \bar{Q} \times \bar{Q} \to \mathbb{R}$ of the Hamilton-Jacobi PDE \cite{marsden-west,hairer} :
\begin{equation*}
H\left(t_1, q_1, \frac{\partial \bar{S}}{\partial q_1} (t_0, q_0, t_1, q_1)\right) = -\frac{\partial \bar{S}}{\partial t_1} (t_0, q_0, t_1, q_1)
\end{equation*}

A solution $\bar{S}(t_0, q_0, t_1, q_1)$ can be understood as a \emph{type 1 generating function} when seen in the framework of canonical transformations \cite{metodosmatem}. Its partial derivatives with respect to $q_1$ and $q_0$ are  $\frac{\partial \bar{S}}{\partial q_1} (t_0, q_0, t_1, q_1) \in T_{q_1}^*Q$ and $\frac{\partial \bar{S}}{\partial q_0} (t_0, q_0, t_1, q_1) \in T_{q_0}^*Q$, respectively.

As Jacobi himself showed, this function corresponds to the action integral:
\begin{equation}
\bar{S}(t_0, q_0, t_1, q_1) = \int_{t_0}^{t_1} L(q(\tau),\dot{q}(\tau),\tau) \mathrm{d}\tau = \mathcal{J}(c)
\label{eq:Jacobis_solution}
\end{equation}
where $c \in C^2(q_0,q_1,[t_0,t_1])$ satisfies Hamilton's principle, i.e. it is a solution of the Euler-Lagrange equations, satisfying the boundary conditions $q(0) = q_0$ and $q(t) = q_1$.

Assuming that our Lagrangian is time-independent, we can then work with $S(q_0, q_1, t_1 - t_0) := \bar{S}(0, q_0, t_1 - t_0, q_1)$, which is perhaps a better-known presentation of a complete solution of the Hamilton-Jacobi equation. If $L$ is regular then for sufficiently small $h \in \mathbb{R}$ we know that $q(h \tau)$, with $\tau \in [0, 1]$, is unique and so $S(q_0, q_1, h)$ is well-defined. This  function is what we call the \emph{exact discrete Lagrangian} of the system and in the jargon of variational integrators it is usually denoted as $S \equiv L_d^{e}$.

Given an extremal trajectory $c(t) \in C^2(q_a,q_b,[t_a,t_b])$ for the Hamilton's principle with $t_b > t_a$ and a set of values $t_i \in (t_a,t_b)$, for $i = 1, ..., N-1$ we may divide the trajectory into smaller segments, $c_{i}(t) \in C^2(q_{i},q_{i+1},[t_i,t_{i+1}])$, where $q_{i} = c(t_i)$. With $t_a \equiv t_0$, $t_b \equiv t_N$, $q_a \equiv q_{0}$ and $q_b \equiv q_{N}$, we say that $c_d: \left\lbrace t_i \right\rbrace_{i = 0}^{N} \to Q, t_i \mapsto q_i$ is a discrete exact  trajectory. From the definition, we have 
\begin{equation*}
L_d^{e}(q_a, q_b, t_b - t_a) = \sum_{k = 0}^{N-1} L_d^{e}(q_k, q_{k+1}, t_{k+1} - t_k)\; .
\end{equation*}

Let us restrict to constant step-size discretizations so that $t_{i+1} - t_i = h, \forall i = 0, ..., N-1$, and consider the space of discrete trajectories connecting $q_a, q_b \in Q$.
%\begin{equation*}
%C_d \left(q_a, q_b, \left\lbrace t_i \right\rbrace_{i = 0}^{N}\right) = \left\lbrace c_d : \left\lbrace t_i \right\rbrace_{i = 0}^{N} \to Q\,|\, c_d(t_a) = q_a, c_d(t_b) = q_b\right\rbrace,
%\end{equation*}
\begin{equation*}
C_d \left(q_a, q_b\right) = \left\lbrace c_d : \left\lbrace t_i \right\rbrace_{i = 0}^{N} \to Q\,\vert\, c_d(t_0) = q_a, c_d(t_N) = q_b\right\rbrace.
\end{equation*}
whose tangent space is:
\begin{equation*}
T_{c_d} C_d \left(q_a, q_b\right) = \left\lbrace X_d : \left\lbrace t_i \right\rbrace_{i = 0}^{N} \to TQ\,\vert\, \pi_Q X_d = c_d, X_d(t_0) = X_d(t_N) = 0\right\rbrace.
\end{equation*}

If we define the functional
%\begin{equation*}
%\begin{array}{rrcl}
%\mathcal{J}_d: & C_d \left(q_a, q_b, \left\lbrace t_i \right\rbrace_{i = 0}^{N}\right) & \longrightarrow &                               \mathbb{R}\\
%               &                    c_d &     \longmapsto & \sum_{k = 0}^{N-1} L_d^{e}(q_k, q_{k+1}, t_{k+1} - t_k)
%\end{array}
%\end{equation*}
\begin{equation*}
\begin{array}{rrcl}
\mathcal{J}^e_d: & C_d \left(q_a, q_b\right) & \longrightarrow &                               \mathbb{R}\\
               &                    c_d &     \longmapsto & \sum_{k = 0}^{N-1} L_d^{e}(q_k, q_{k+1}),
\end{array}
\end{equation*}
where $L_d^{e}(q_k, q_{k+1}) := L_d^{e}(q_k, q_{k+1}, h)$, we can check that if $c_d$ discretizes $c(t)$, the solution of the Euler-Lagrange equations, then $\mathrm{d} \mathcal{J}^e_d (c_d)(X_d) = 0$, which leads us to stablish the following
%\begin{definition}{\textbf{(Discrete Hamilton's principle)}}\label{def:discrete_hamiltons_principle}{ $c_d \in C_d \left(q_a, q_b, \left\lbrace t_i \right\rbrace_{i = 0}^{N}\right)$ is a solution of the discrete Lagrangian system defined by a discrete Lagrangian $L_d: Q \times Q \times \mathbb{R} \to \mathbb{R}$ if and only if $c_d$ is a critical point of the functional $\mathcal{J}_d$, i.e. $\mathrm{d}\mathcal{J}_d(c_d) = 0$}
%\end{definition}
\begin{definition}{\textbf{(Discrete Hamilton's principle)}}\label{def:discrete_hamiltons_principle}{ $c_d \in C_d \left(q_a, q_b\right)$ is a solution of the discrete Lagrangian system defined by a discrete Lagrangian $L_d: Q \times Q \to \mathbb{R}$ if and only if $c_d$ is a critical point of the functional $\mathcal{J}_d$, i.e. $\mathrm{d}\mathcal{J}_d(c_d)(X_d) = 0$, for all $X_d \in T_{c_d} C_d(q_a, q_b)$}.
\end{definition}
Here $L_d: Q\times Q\rightarrow {\mathbb R}$ is an arbitrary $C^2$-function not necessarily related with a continuous lagrangian although typically we will choose an suitable apporoximation of the exact discrete Lagrangian $L_d^e$ (see \cite{marsden-west}).

It is easy to show that discrete curves $c_d$ satisfying the discrete Hamilton's principle are the solutions of the following system of implicit difference equations
%\begin{equation}
%D_2 L_d(q_{i-1}, q_{i}, t_{i} - t_{i-1}) + D_1 L_d(q_{i}, q_{i+1}, t_{i+1} - t_i) = 0, \quad i = 1, ..., N - 1;
%\label{eq:discrete_euler_lagrange}
%\end{equation}
\begin{equation}
D_2 L_d(q_{i-1}, q_{i}) + D_1 L_d(q_{i}, q_{i+1}) = 0, \quad i = 1, ..., N - 1;
\label{eq:discrete_euler_lagrange}
\end{equation}
where $D_i$ denotes partial derivation with respect to the $i$-th variable. These are the \emph{discrete Euler-Lagrange} equations \cite{marsden-west}. Each of these equations gives us a map, called \emph{discrete Lagrangian map}
\begin{equation*}
\begin{array}{rccc}
F_{L_d}:& Q \times Q & \to & Q \times Q\\
        &(q_{i-1},q_{i}) & \mapsto & (q_{i},q_{i+1})
\end{array}
\end{equation*}
which evolves our discrete system. 

Following  \cite{marsden-west}, we have that if $c_d$ is a critical point of $\mathcal{J}_d$, then $\mathrm{d}\mathcal{J}_d(c_d) = \Theta_{L_d}^{+}(q_{N-1},q_{N}) - \Theta_{L_d}^{-}(q_{0},q_{1}) = \left[\left(F_{L_d}^{N-1}\right)^* \Theta_{L_d}^{+} - \Theta_{L_d}^{-}\right](q_0,q_1)$, where
\begin{eqnarray*}
\Theta_{L_d}^{-}(q_0,q_1) &=& - D_1 L_d(q_0,q_1) \mathrm{d}q_0\\
\Theta_{L_d}^{+}(q_0,q_1) &=& D_2 L_d(q_0,q_1) \mathrm{d}q_1
\end{eqnarray*}
are called the discrete Poincar\'e-Cartan forms. Applying a second exterior derivative to these we find that:
\begin{equation*}
\Omega_{L_d} = \mathrm{d}\Theta_{L_d}^{-} = \mathrm{d}\Theta_{L_d}^{+} = \frac{\partial^2 L_d}{\partial q_0^i \partial q_0^j} \mathrm{d}q_0^i \wedge \mathrm{d}q_1^j.
\end{equation*}
As in the continuous case, this form is nondegenerate if and only if the matrix $(\partial^2 L_d/\partial q_0\partial q_1)$ is nonsingular (a symplectic 2-form), and in this case we say that the discrete Lagrangian is \emph{regular}.

From $\mathrm{d}^2 \mathcal{J}_d(c_d) = 0$ we get that $\left(F_{L_d}^{N-1}\right)^* \Omega_{L_d} = \Omega_{L_d}$, which is true for any number of steps. Therefore the  discrete Lagrangian map preserve  
2-form $\Omega_{L_d}$ and therefore its is  a \emph{symplectic trasnformation} in the regular case. 

Now, if we consider the maps $\mathbb{F}L_d^{\pm}: Q \times Q \to T^*Q$ defined by:
\begin{eqnarray*}
\mathbb{F}L_d^{-}(q_0,q_1) &=& (q_0, p_0 = - D_1 L_d(q_0,q_1) )\\
\mathbb{F}L_d^{+}(q_0,q_1) &=& (q_1, p_1 = D_2 L_d(q_0,q_1) )
\end{eqnarray*}
called \emph{discrete fibre derivatives} or \emph{discrete Legendre transforms}, eqs.\eqref{eq:discrete_euler_lagrange} can be reinterpreted as a matching condition for the momenta:
\begin{equation*}
- D_2 L_d(q_{i-1}, q_{i}) = p_i^{-} = p_i^{+} = D_1 L_d(q_{i}, q_{i+1}), \quad i = 1, ..., N - 1.
\end{equation*}

This in turn leads us to the following commutative diagram:
\begin{center}
\begin{tikzcd}[column sep=tiny, row sep=huge]
		Q \times Q :& & (q_0,q_1) \arrow[dl, mapsto, "\mathbb{F}L_d^{-}"'] \arrow[dr, mapsto, "\mathbb{F}L_d^{+}"] \arrow[rr, mapsto, "F_{L_d}"] & & (q_1,q_2) \arrow[dl, mapsto, "\mathbb{F}L_d^{-}"'] \arrow[dr, mapsto, "\mathbb{F}L_d^{+}"] &\\
		T^*Q :& (q_0,p_0) \arrow[rr, mapsto, "\widetilde{F}_{L_d}"] & & (q_1,p_1) \arrow[rr, mapsto, "\widetilde{F}_{L_d}"] & & (q_2,p_2)
\end{tikzcd}
\end{center}
Here $\widetilde{F}_{L_d}$ receives the name of \emph{discrete Hamiltonian map} and is also  symplectic for the canonical symplectic form $\omega_Q$, i.e., $\widetilde{F}_{L_d}^*\omega_Q=\omega_Q$. It is also worthwhile noting that together with the continuous fibre derivative we can relate $(q_0, v_0)$, the initial values of our IVP, with $(q_0, q_1)$ via $(q_0, p_0)$:
\begin{center}
\begin{tikzcd}[column sep=tiny, row sep=huge]
		Q \times Q : && (q_0,q_1)\\
		T^*Q : &(q_0,p_0) \arrow[ur, mapsto, "\left(\mathbb{F}L_d^{-}\right)^{-1}"'] &\\
		T Q : &(q_0,v_0) \arrow[u, mapsto, "\mathbb{F}L"'] &
\end{tikzcd}
\end{center}
which will be useful for the initialization of a numerical integration procedure.

\subsection{Variational integrators}
Although warranted to exist, it is generally impossible to obtain an analytic expression for the exact discrete Lagrangian. What we can do is try to approximate the action integral by a finite sum and hope that the extremum of this approximation converges to the extremum of the continuous problem as the quality of our quadrature augments and as the length of our intervals, $h$, diminishes. This leads to the following definition:
\begin{definition}
Let $L:TQ\rightarrow {\mathbb R}$ a regular lagrangian, $L_d^e$ the corresponding exact discrete Lagrangian and $L_{d}: Q\times Q\rightarrow \mathbb{R}$ be a discrete Lagrangian. We say that $L_{d}$ is a discretization of order $r$ if there exist an open subset $U_{1}\subset TQ$ with compact closure and constants $C_1 > 0$, $h_1 > 0$ so that:
\begin{equation*}
\left\Vert L_{d}(q(0),q(h)) - L_{d}^{e}(q(0),q(h))\right\Vert \leq C_1 h^{r+1}
\end{equation*}
for all solutions $q(t)$ of the second-order Euler-Lagrange equations with initial conditions $(q_0,\dot{q}_0)\in U_1$ and for all $h \leq h_1$.
\end{definition}

Using these approximations $L_{d} \approx L_{d}^{e}$ we can apply the discrete Hamilton's principle which leads us again to eq.\eqref{eq:discrete_euler_lagrange}. The resulting discrete flows, be it the discrete Lagrangian map $F_{L_d}$ our the discrete Hamiltonian map $\widetilde{F}_{L_d}$, become our \emph{variational integrator}. By construction, variational integrators automatically preserve symplecticity and momentum and exhibit quasi-energy conservation for exponentially long times (\cite{marsden-west} and references therein).

Following \cite{marsden-west,patrick-cuell}, we have the next result about the order of a variational integrator.

\begin{theorem}\label{thm:variational_error}
If $\widetilde{F}_{L_d}$ is the Hamiltonian map of an order $r$ discretization $L_d: Q\times Q \to \mathbb{R}$ of the exact discrete Lagrangian $L_d^{e}: Q \times Q \to \mathbb{R}$, then
\begin{equation*}
\widetilde{F}_{L_d} = \widetilde{F}_{L_{d}^{e}} + \mathcal{O}(h^{r+1}).
\end{equation*}
In other words, $\widetilde{F}_{L_d}$ gives an integrator of order $r$ for $\widetilde{F}_{L_{d}^{e}} = F_{H}^{h}$.
\end{theorem}

Note that given a discrete Lagrangian $L_{d}$ its order can be calculated by expanding the expressions for $L_d(q(0),q(h))$ in a Taylor series in $h$ and comparing this to the same expansions for the exact discrete Lagrangian. If the series agree up to $r$ terms, then the discrete Lagrangian is of order $r$ \cite{marsden-west}.

This result is key as it essentially tells us that the order of the quadrature rule we use to approximate our discrete Lagrangian will be the order of the resulting variational integrator.

\subsection{Discrete Hamilton-Pontryagin action and partitioned Runge-Kutta methods}
In order to construct variational integrators of arbitrary order it will be important for us to consider the Hamilton-Pontryagin action and its corresponding discretization \cite{MR2496560,KobiMarsSukh}. First, consider the extended space of curves
\begin{align*}
&C^1 ((q_a,v_a,p_a),(q_b,v_b,p_b), [a,b])\\
&= \left\lbrace (q,v,p): [a, b] \rightarrow \mathbb{T}Q \,|\, q \in C^2([a, b]), v, p \in C^1([a, b]),\right.\\
&\left.\hspace{0.62cm}(q,v,p)(a) = (q_a,v_a,p_a), (q,v,p)(b) = (q_b,v_b,p_b) \right\rbrace
\end{align*}
where $\mathbb{T}Q := TQ \oplus T^*Q=\{ (v_q, \alpha_q), v_q\in T_qQ, \alpha_q\in T^*_q, \forall q\in Q\}$ denotes the Whitney sum of the tangent and cotangent bundles of $Q$.

Let us denote by $\pi_{\oplus}: \mathbb{T}Q \to Q$ the corresponding bundle projection, $\pi_{\oplus}(v_q, \alpha_q)=q$. Then if $c \in C^1 ((q_a,v_a,p_a),(q_b,v_b,p_b), [a,b])$ and $\pi_{\oplus} c = q \in C^2 (q_0, q_1, [a,b])$ we say that $c$ is over the curve $q$.

Let $\mathcal{J_{HP}}: C^1 ((q_a,v_a,p_a), (q_b,v_b,p_b), [a,b]) \to \mathbb{R}$ denote the functional defined by:
\begin{equation}
\mathcal{J_{HP}}(q,v,p) = \int_0^{h} \left[ L(q(t),v(t)) + \left\langle p(t), \dot{q}(t) - v(t)\right\rangle \right] \mathrm{d}t\; .
\label{eq:continuous_Hamilton-Pontriagyn_action}
\end{equation}
This is the so-called \emph{Hamilton-Pontryagin action} functional and it can be interpreted as a constrained action functional where the $p$'s act as Lagrange multipliers and we are imposing the kinematic constraint $\dot{q}(t) = v(t)$. It can be checked that for the extended curve $(q,v,p)$ to be a compatible critical point of the action, $\mathrm{d}\mathcal{J_{HP}}(q,v,p) = 0$, then the following equations must be satisfied:
\begin{align*}
\frac{\mathrm{d}p(t)}{\mathrm{d}t} &= D_1 L(q(t),v(t)),\\
p(t) &= D_2 L(q(t),v(t)),\\
\frac{\mathrm{d}q(t)}{\mathrm{d}t} &= v(t), \quad \forall t \in [0, h].
\end{align*}
These imply that $q(t)$ must be a critical point of $\mathcal{J}(q)$, and that the Lagrange multipliers must coincide with the canonical momenta, hence the notation.

Incidentally, if variations are taken without imposing fixed end-point conditions we obtain the boundary term:
\begin{equation*}
\left.\left\langle p(t), \delta q(t)\right\rangle\right\vert_0^h = \left\langle p(h), \delta q(h)\right\rangle - \left\langle p(0), \delta q(0)\right\rangle = \theta_{L}(\delta q)(h) - \theta_{L}(\delta q)(0)
\end{equation*}

This variational principle does not add much in the continuous realm for regular Lagrangian systems, but higher order variational integrators can be generated by discretising such an action. To proceed we must choose how to discretise the constraint $\dot{q} = v$. Assuming we are working on a vector space $Q$ (later we will explore the case of a Lie group), the usual way to do so is by using an $s$-stage Runge-Kutta (RK) scheme for the integration of such an ODE:
\[
Q_0^i = q_0 + h \sum_{j=1}^{s} a_{i j} V_0^i\; , \qquad
q_1 = q_0 + h \sum_{j=1}^{s} b_{j} V_0^i
\]

From here on it will be assumed that the RK coefficients satisfy the \emph{consistency condition}
\[
\sum_{j = 1}^s a_{i j} = c_i,
\]
which often appears as part of the definition of RK method, and at least the order 1 condition:
\begin{equation}
\sum_{j = 1}^s b_{j} = 1.\tag{order 1}\label{eq:order_1_cond}
\end{equation}

Following \cite{MR2496560}, given an $s$-stage RK scheme, let us consider the space of \emph{$s$-stage variationally partitioned RK ( s-stage VPRK) sequences}:
\begin{align*}
&C_d^s (q_a, q_b)\\
&= \left\lbrace \left(q,\tilde{p},\left\lbrace Q^i, V^i, \tilde{P}^i\right\rbrace_{i = 1}^s\right): \left\lbrace t_k\right\rbrace_{k=0}^N \rightarrow T^*Q \times \left(\mathbb{T}Q\right)^s \,|\, q(a) = q_a, q(b) = q_b\right\rbrace.
\end{align*}

Then we can define the following discrete Hamilton-Pontryagin functional, $\left(\mathcal{J_{HP}}\right)_d: C_d^s (q_a, q_b) \to \mathbb{R}$, by:
\begin{align}
\left(\mathcal{J_{HP}}\right)_d = &\sum_{k = 0}^{N-1} \sum_{i = 1}^s h b_i \left[ L\left(Q_k^i, V_k^i\right) + \left\langle P_k^i, \frac{Q_k^i - q_k}{h} - \sum_{j=1}^s a_{i j} V_k^j\right\rangle\right.\label{eq:discrete_Hamilton-Pontriagyn_action}\\
&+ \left. \left\langle \tilde{p}_{k+1}, \frac{q_{k+1} - q_k}{h} - \sum_{j=1}^s b_j V_k^j\right\rangle \right]\nonumber
\end{align}

\begin{theorem}\label{thm:VPRK_methods}
Let $L : TQ \to \mathbb{R}$ be a $C^l$ function with $l \geq 2$ and an $s$-stage VPRK sequence $c_d \in C_d^s (q_0, q_N)$. Then $c_d$ is a critical point of the discrete Hamilton-Pontryagin functional, $\left(\mathcal{J_{HP}}\right)_d$, if and only if for all $k = 0, ..., N-1$ and $i = 1, ..., s$ it satisfies
\begin{equation}
\label{eq:symplectic_partitioned_integrator}
\begin{array}{ll}
\, q_{k+1}  = q_k + h \sum_{j = 1}^{s} b_{j} V_k^j,\qquad  & p_{k+1}  = p_k + h \sum_{i = 1}^{s} \hat{b}_{j} W_k^j,\\
Q_k^i  = q_k + h \sum_{j = 1}^{s} a_{i j} V_k^j, & P_k^i  = p_k + h \sum_{j = 1}^{s} \hat{a}_{i j} W_k^j,\\
W_k^i  = D_1 L(Q_k^i, V_k^i), & P_k^i  = D_2 L(Q_k^i, V_k^i),
\end{array}
\end{equation}
where the RK coefficients satisfy $b_i \hat{a}_{i j} + \hat{b}_j a_{j i} = b_i \hat{b}_j$ and $\hat{b}_{i} = b_i$.
\end{theorem}

The condition on the RK coefficients is called the \emph{symplecticity condition} of the partitioned method.

\begin{proof}
Computing the variations of this action we get:
\begin{align*}
\hspace{1cm}&\hspace{-1cm}\left\langle\mathrm{d}\left(\mathcal{J_{HP}}\right)_d, \delta c_d \right\rangle\\
&=\sum_{k = 0}^{N-1} \sum_{i = 1}^s h b_i \left[ \vphantom{\sum_{j=1}^s} \left\langle D_1 L\left(Q_k^i, V_k^i\right), \delta Q_k^i\right\rangle + \left\langle D_2 L\left(Q_k^i, V_k^i\right), \delta V_k^i\right\rangle\right.\\
&+ \left\langle \delta\tilde{P}_k^i, \frac{Q_k^i - q_k}{h} - \sum_{j=1}^s a_{i j} V_k^j\right\rangle + \left\langle \tilde{P}_k^i, \frac{\delta Q_k^i - \delta q_k}{h} - \sum_{j=1}^s a_{i j} \delta V_k^j\right\rangle\\
&+ \left. \left\langle \delta \tilde{p}_{k+1}, \frac{q_{k+1} - q_k}{h} - \sum_{j=1}^s b_j V_k^j\right\rangle + \left\langle \tilde{p}_{k+1}, \frac{\delta q_{k+1} - \delta q_k}{h} - \sum_{j=1}^s b_j \delta V_k^j\right\rangle \right]
\end{align*}

Let us collect all the different terms separately:
\begin{align*}
\delta q &: \sum_{k = 0}^{N-1} \sum_{i = 1}^s b_i \left[ - \left\langle \tilde{P}_k^i, \delta q_k \right\rangle + \left\langle \tilde{p}_{k+1}, \delta q_{k+1} - \delta q_k \right\rangle \right]\\
&\hspace{0.25cm} = \sum_{k = 0}^{N-1} \sum_{i = 1}^s b_i \left[ - \left\langle \tilde{P}_k^i + \tilde{p}_{k+1}, \delta q_k \right\rangle + \left\langle \tilde{p}_{k+1}, \delta q_{k+1} \right\rangle \right]\\
&\hspace{0.25cm} = \sum_{k = 1}^{N-1} \sum_{i = 1}^s b_i \left[ \left\langle -\tilde{P}_k^i + \tilde{p}_{k} -  \tilde{p}_{k+1} , \delta q_k \right\rangle \right]\\
&\hspace{0.25cm}\hspace{0.5cm} + \sum_{i = 1}^s b_i \left\langle \tilde{p}_{N}, \delta q_{N} \right\rangle  - \sum_{i = 1}^s b_i \left\langle \tilde{P}_0^i + \tilde{p}_{1}, \delta q_0 \right\rangle\\
&\hspace{0.25cm} = \sum_{k = 1}^{N-1} \left\langle \tilde{p}_{k} - \tilde{p}_{k+1} - \sum_{i = 1}^s b_i \tilde{P}_k^i , \delta q_k \right\rangle\\
&\hspace{0.25cm}\hspace{0.5cm} + \left\langle \tilde{p}_{N}, \delta q_{N} \right\rangle  - \left\langle \sum_{i = 1}^s b_i \tilde{P}_0^i + \tilde{p}_{1}, \delta q_0 \right\rangle
\end{align*}
where we have used the \ref{eq:order_1_cond} condition.

\begin{align*}
\delta Q &: \sum_{k = 0}^{N-1} \sum_{i = 1}^s b_i \left[ \left\langle h D_1 L\left(Q_k^i, V_k^i\right), \delta Q_k^i\right\rangle + \left\langle \tilde{P}_k^i, \delta Q_k^i\right\rangle \right]\\
&\hspace{0.25cm} = \sum_{k = 0}^{N-1} \sum_{i = 1}^s b_i \left\langle h D_1 L\left(Q_k^i, V_k^i\right) +  \tilde{P}_k^i, \delta Q_k^i\right\rangle
\end{align*}

\begin{align*}
\delta V &: \sum_{k = 0}^{N-1} \sum_{i = 1}^s h b_i \left[ \left\langle D_2 L\left(Q_k^i, V_k^i\right), \delta V_k^i\right\rangle - \left\langle \tilde{P}_k^i, \sum_{j=1}^s a_{i j} \delta V_k^j\right\rangle - \left\langle \tilde{p}_{k+1}, \sum_{j=1}^s b_j \delta V_k^j\right\rangle \right]\\
%&\hspace{0.25cm} = \sum_{k = 0}^{N-1} h \left[ \left\langle \sum_{i = 1}^s b_i D_2 L\left(Q_k^i, V_k^i\right), %\delta V_k^i\right\rangle - \sum_{i = 1}^s \sum_{j=1}^s \left\langle b_i \tilde{P}_k^i, a_{i j} \delta %V_k^j\right\rangle \right.\\
&\hspace{0.25cm}\hspace{0.5cm} \left.- \left\langle \sum_{i = 1}^s b_i \tilde{p}_{k+1}, \sum_{j=1}^s b_j \delta V_k^j\right\rangle\right]\\
&\hspace{0.25cm} = \sum_{k = 0}^{N-1} h \left[ \left\langle \sum_{i = 1}^s b_i D_2 L\left(Q_k^i, V_k^i\right), \delta V_k^i\right\rangle - \sum_{i = 1}^s \sum_{j=1}^s \left\langle b_j \tilde{P}_k^j, a_{j i} \delta V_k^i\right\rangle \right.\\
&\hspace{0.25cm}\hspace{0.5cm} \left.- \left\langle \tilde{p}_{k+1}, \sum_{i=1}^s b_i \delta V_k^i\right\rangle\right]\\
&\hspace{0.25cm} = \sum_{k = 0}^{N-1} \sum_{i=1}^s h \left\langle b_i D_2 L\left(Q_k^i, V_k^i\right) - \sum_{j=1}^s b_j a_{j i} \tilde{P}_k^j - b_i \tilde{p}_{k+1}, \delta V_k^i\right\rangle
\end{align*}

\begin{align*}
\delta \tilde{p} &: \sum_{k = 0}^{N-1} \sum_{i = 1}^s h b_i \left\langle \delta \tilde{p}_{k+1}, \frac{q_{k+1} - q_k}{h} - \sum_{j=1}^s b_j V_k^j\right\rangle
\end{align*}

\begin{align*}
\delta \tilde{P} &: \sum_{k = 0}^{N-1} \sum_{i = 1}^s h b_i \left\langle \delta\tilde{P}_k^i, \frac{Q_k^i - q_k}{h} - \sum_{j=1}^s a_{i j} V_k^j\right\rangle
\end{align*}

From these last two variations we recuperate the original discrete kinematic constraints, as expected.

From $\delta Q$ we get that:
\begin{equation*}
\tilde{P}_k^i = - h D_1 L\left(Q_k^i, V_k^i\right)
\end{equation*}

Inserting this in $\delta q$ we get:
\begin{equation*}
\tilde{p}_{k+1} = \tilde{p}_{k} + h \sum_{i = 1}^s b_i D_1 L\left(Q_k^i, V_k^i\right)
\end{equation*}

Comparing the boundary terms from the continuous and discrete cases we see that these $\tilde{p}_k$ variables are approximations to the continuous $p(t_k)$. Thus we will drop the tildes, making this identification explicit.

From $\delta V$ we find that:
\begin{equation*}
D_2 L\left(Q_k^i, V_k^i\right) =  p_{k+1} - h \sum_{j=1}^s \frac{b_j a_{j i}}{b_i} D_1 L\left(Q_k^j, V_k^j\right)
\end{equation*}

Inserting what we found from $\delta q$ here, we get:
\begin{equation*}
D_2 L\left(Q_k^i, V_k^i\right) =  p_{k} + h \sum_{j = 1}^s b_j \left( 1 - \frac{a_{j i}}{b_i}\right) D_1 L\left(Q_k^j, V_k^j\right)
\end{equation*}

Rewriting the equations using $\hat{b}_i$, $\hat{a}_{i j}$, $P_k^i$ and $W_k^i$ we get the result we were after.
\end{proof}

The resulting system of equations from theorem \ref{thm:VPRK_methods} defines a discrete Hamiltonian map, i.e. a mapping $(q_k, p_k) \mapsto (q_{k+1}, p_{k+1})$ and in order to determine these we will also need to determine the set $\left\lbrace Q_k^i, V_k^i\right\rbrace_{i=1}^{s}$. By theorem \ref{thm:variational_error}, the order of the discrete Hamiltonian map will coincide with the order of the RK method applied \cite{hairer,marsden-west}. 

\subsection{Discrete holonomically constrained mechanical systems}
We will not go into much detail here as it would take much space to delve into the details and the caveats involved in the study of holonomically constrained mechanical systems \cite{LeRe}. Conceptually and geometrically there is not a big departure from the unconstrained case save for the particularities of the augmented picture, where we include Lagrange multipliers which make the augmented Lagrangian and its corresponding discrete counterpart singular. The reader is referred to \cite{marsden-west,BouRabeeOwhadi,Jay1996}  for more information.

As in the continuous case, let $i_N: N \hookrightarrow Q$ denote the inclusion map of the submanifold $N$ in $Q$. We may naturally extend this inclusion to $Q \times Q$ where discrete Lagrangians live as $i_{N \times N}: N \times N \hookrightarrow Q \times Q$, $i_{N\times N}(q_0,q_1) = (i_N(q_0),i_N(q_1))$. Given a discrete Lagrangian  $L_d$ we may define the restricted discrete Lagrangian $L_d^N = L_d \circ i_{N\times N}$.

Let us now define the augmented space, $Q \times \Lambda$, where $\Lambda \cong \mathbb{R}^{m}$ is the space of Lagrange multipliers and $m = \mathrm{codim}_Q N$, the number of independent holonomic constraints. With this we may also define an augmented discrete Lagrangian $\tilde{L}_d: Q \times \Lambda \times Q \times \Lambda \to \mathbb{R}$ as an approximation to the exact discrete Lagrangian for $\tilde{L} = L + \left\langle\lambda, \Phi\right\rangle$. We will be more interested in the augmented approach, as we will use that for nonholonomic systems in the next section (see \cite{mamast} for a more intrinsic approach using Lagrangian submanifolds).

Variation of the augmented discrete action leads to:
\begin{align}
0 &= \delta \sum_{k = 0}^{N-1} L_d (q_k, q_{k+1}) + \sum_{k = 0}^{N-1} \left[ \left\langle f_d^{-}(q_k, \lambda_{k}, q_{k+1}, \lambda_{k+1}), \delta q_k\right\rangle \right. +\nonumber\\
&+ \left. \left\langle f_d^{+}(q_k, \lambda_{k}, q_{k+1}, \lambda_{k+1}), \delta q_{k+1}\right\rangle \right] + h \sum_{k = 0}^{N-1} \sum_{i = 1}^{s} b_i \left\langle\delta \Lambda_k^i, \Phi\left(Q_k^i\right)\right\rangle
\label{eq:discrete_constrained_variation}
\end{align}
where:
\begin{subequations}
	\label{eq:holonomic_forcing}
	\begin{alignat}{2}
		f_d^{-}(q_k, \lambda_{k}, q_{k+1}, \lambda_{k+1}) &= h \sum_{i = 1}^{s} b_i \left\langle\Lambda_k^i, D\Phi\left(Q_k^i\right)\frac{\partial Q_k^i}{\partial q_{k}},\right\rangle\\
		f_d^{+}(q_k, \lambda_{k}, q_{k+1}, \lambda_{k+1}) &= h \sum_{i = 1}^{s} b_i \left\langle\Lambda_k^i, D\Phi\left(Q_k^i\right) \frac{\partial Q_k^i}{\partial q_{k+1}},\right\rangle
	\end{alignat}
\end{subequations}
are the forcing terms arising from the constraints. Let us assume that we discretize our Lagrangian applying the trapezoidal rule, which simplifies these to:
\begin{align*}
f_d^{-}(q_k, \lambda_{k}, q_{k+1}, \lambda_{k+1}) &= \frac{h}{2} \left\langle \lambda_{k}, D\Phi\left(q_k\right)\right\rangle,\nonumber\\
f_d^{+}(q_k, \lambda_{k}, q_{k+1}, \lambda_{k+1}) &= \frac{h}{2} \left\langle \lambda_{k+1}, D\Phi\left(q_{k+1}\right)\right\rangle.
\end{align*}

Then the constrained discrete Euler-Lagrange equations take the form:
\begin{align*}
D_2 L_d(q_{k-1}, q_{k}) + D_1 L_d(q_{k}, q_{k+1}) & = - h\left\langle \lambda_{k}, D\Phi(q_k)\right\rangle\\
\Phi(q_k) & = 0, \quad \forall k = 0, ..., N
\end{align*}
which means the solution must satisfy:
\begin{equation*}
\left(T^*i_{N}\right)_{q_k} \left[ D_2 L_d \circ i_{N \times N}(\tilde{q}_{k-1}, \tilde{q}_{k}) + D_1 L_d \circ i_{N \times N} (\tilde{q}_{k}, \tilde{q}_{k+1}) \right] = 0.
\end{equation*}
where $\tilde{q}_k \in N$.

Using the restricted discrete Lagrangian the equations simplify to the expected:
\begin{equation*}
D_2 L_d^N(\tilde{q}_{k-1}, \tilde{q}_{k}) + D_1 L_d^N(\tilde{q}_{k}, \tilde{q}_{k+1}) = 0
\end{equation*}

The equivalent symplectic integrator written in Hamiltonian form becomes:
\begin{align*}
p_0 &= - D_1 L_d(q_0, q_1) - \frac{h}{2}\left\langle\lambda_0, D\Phi(q_0)\right\rangle\\
p_1 &= D_2 L_d(q_0, q_1) + \frac{h}{2}\left\langle\lambda_1, D\Phi(q_1)\right\rangle\\
0 &= \Phi(q_1)\\
0 &= \left\langle D\Phi(q_1), \frac{\partial H}{\partial p}(q_1,p_1)\right\rangle
\end{align*}
where the last equation must be enforced so that $p_1 \in T^*_{q_1} N$. This latter method is known as RATTLE and its augmented Lagrangian version is the SHAKE method
\cite{rycibe, andersen, leimkhulerskeel, LeRe, marsden-west}. 

Working with the augmented discrete Hamilton-Pontryagin action:
\begin{align}
\left(\widetilde{\mathcal{J}}_\mathcal{{HP}}\right)_d = &\sum_{k = 0}^{N-1} \sum_{i = 1}^s h b_i \left[ L\left(Q_k^i, V_k^i\right) + \left\langle \tilde{P}_k^i, \frac{Q_k^i - q_k}{h} - \sum_{j=1}^s a_{i j} V_k^j\right\rangle\right.\label{eq:augmented_discrete_Hamilton-Pontriagyn_action}\\
&+ \left. \left\langle \tilde{p}_{k+1}, \frac{q_{k+1} - q_k}{h} - \sum_{j=1}^s b_j V_k^j\right\rangle + \left\langle \Lambda_k^i, \Phi(Q_k^i)\right\rangle\right]\nonumber
\end{align}
we can obtain a constrained version of the symplectic partitioned Runge-Kutta method:
\begin{subequations}
\label{eq:holonomic_integrator}
\begin{alignat}{2}
q_{k+1} & = q_k + h \sum_{j = 1}^{s} b_{j} V_k^j, & p_{k+1} & = p_k + h \sum_{j = 1}^{s} \hat{b}_{j} W_k^j,\\
Q_k^i & = q_k + h \sum_{j = 1}^{s} a_{i j} V_k^j, & P_k^i & = p_k + h \sum_{j = 1}^{s} \hat{a}_{i j} W_k^j,\\
W_k^i & = D_1 L(Q_k^i, V_k^i) + \left\langle\Lambda_k^i, D\Phi(Q_k^i)\right\rangle \vphantom{\sum_{j = 1}^{s}} & P_k^i & = D_2 L(Q_k^i, V_k^i) \vphantom{\sum_{j = 1}^{s}}\\
0 & = \Phi(Q_k^i) \vphantom{\sum_{j = 1}^{s}} & 0 & = \left\langle D\Phi(q_{k+1}), D_2 H(q_{k+1},p_{k+1})\right\rangle \vphantom{\sum_{j = 1}^{s}}\label{eq:holonomic_integrator_subeq5}
\end{alignat}
\end{subequations}
where $(a_{i j}, b_i)$ and $(\hat{a}_{i j}, \hat{b}_i)$ are a pair of symplectically conjugated methods. The tangency condition on $q_{k+1}$ (eq.\eqref{eq:holonomic_integrator_subeq5}, right) must be judiciously added to close the system, allowing us to obtain a Hamiltonian map $(q_k,p_k) \mapsto (q_{k+1},p_{k+1})$.

Unfortunately not every choice of RK method will give us the expected variational order for the Hamiltonian flow on $N$ (see \cite{Jay1993, Jay1996, Jay2003, marsden-west, JohnsonMurphey}. Indeed, numerical tests already tell us that for a $2$-stage Gau{\ss} method whose corresponding variational order should be $4$ actually has order $2$ when projected to $N$.

Forcing  the order to coincide, the corresponding augmented discrete Hamilton-Pontryagin action must restrict to $N$ as well. To warrant that the method restricts to $N$ we impose that $(a_{i j}, b_i)$ satisfy the hypotheses:
\begin{enumerate}[label=H\arabic*]
\item \label{itm:H1} $a_{1 j} = 0$ for $j = 1, ..., s$;
\item \label{itm:H2} the submatrix $\tilde{A} := (a_{i j})_{i,j \geq 2}$ is invertible;
\item \label{itm:H3} $a_{s j} = b_j$ for $j = 1, ..., s$ (the method is \emph{stiffly accurate}).
\end{enumerate}

If $(\hat{a}_{i j}, \hat{b}_i)$ are the coefficients of the symplectic conjugated of the former method, then they must satisfy:
\begin{enumerate}[label=H\arabic*']
\item \label{itm:H1'} $\hat{a}_{i s} = 0$ for $i = 1, ..., s$;
\item \label{itm:H2'} $\hat{a}_{i 1} = \hat{b}_1$ for $i = 1, ..., s$.
\end{enumerate}

\ref{itm:H1} and \ref{itm:H3} imply that $c_1 = 0$ and $c_s = 1$, and thus for any given $k = 0,..., N-1$ the first and last internal stages must coincide with the nodal values, i.e. $Q_k^1 = q_k$ and $Q_k^s = q_{k+1}$, which ensures that we can impose the constraints on purely variational grounds and so the associated augmented discrete Hamilton-Pontryagin must restrict to $N$.

A particular member of the family is the Lobatto IIIA-B pair, whose $2$-stage version is the well-known trapezoidal rule that we applied first.

\section{Discrete nonholonomic mechanics}\label{sec:discrete-nonholonomic}
It is natural to wonder whether we can construct integrators for mechanical nonholonomic systems in a similar manner as we have done with our variational integrators \cite{cortes02,perlmutter06,dLdDSM04,MR2164739,GNI}

As we already know, nonholonomic mechanics is not variational, yet we know that Chetaev's principle is not a radical departure from Hamilton's principle and the resulting equations of motion are fairly similar to those of a holonomic system \cite{neimarkfufaev,bloch,borisovhistoria,cortes02}. It is also true that nonholonomic mechanics is not symplectic either, so the value of applying the philosophy of symplectic integrators seems at least questionable, since there is not, in general, preservation of any symplectic or Poisson structure \cite{SM1994, CdLdD99}. Still, given that the departure from holonomic mechanics is not that dramatic and the generally good behaviour of variational integrators we still believe it is worth trying to extend our approach to nonholonomic systems.

The fact that these systems are not variational implies that the important result of theorem \ref{thm:variational_error} does not apply anymore, which strips us from one of our main tools to prove the order of the resulting methods. This leaves us with standard numerical analysis techniques and results to try and prove the order on a \emph{per family} basis. Again we will focus on symplectic RK pairs satisfying all the hypotheses stated in the holonomically constrained case.

Without further ado, we present the following {\sl nonholonomic partitioned Runge-Kutta integrator}: Let $(L, \Phi)$ be a regular nonholonomic Lagrangian system, with $\Psi = \Phi \circ \mathbb{F}L^{-1}$, then the equations for the integrator are:
\begin{subequations}
	\label{eq:nonholonomic_integrator}
	\begin{alignat}{2}
		q_{k+1} & = q_k + h \sum_{i = 1}^{s} b_{i} V_k^i, & p_{k+1} & = p_k + h \sum_{i = 1}^{s} \hat{b}_{i} W_k^i,\\
		Q_k^i & = q_k + h \sum_{j = 1}^{s} a_{i j} V_k^j, & P_k^i & = p_k + h \sum_{j = 1}^{s} \hat{a}_{i j} W_k^j,\\
		W_k^i & = D_1 L(Q_k^i, V_k^i) + \left\langle\Lambda_k^i, D_2 \Phi(Q_k^i, V_k^i)\right\rangle \vphantom{\sum_{j = 1}^{s}}, & P_k^i & = D_2 L(Q_k^i, V_k^i) \vphantom{\sum_{j = 1}^{s}},\\
		q_k^i & = Q_k^i \vphantom{\sum_{j = 1}^{s}}, & p_k^i & = p_k + h \sum_{j = 1}^{s} a_{i j} W_k^j \vphantom{\sum_{j = 1}^{s}},
	\end{alignat}
	\begin{equation}
		\Psi(q_k^i, p_k^i) = 0 \vphantom{\sum_{j = 1}^{s}}.\label{eq:nonholonomic_integrator_sub5}
	\end{equation}
\end{subequations}
where $(q_k, p_k, \lambda_k)$ are the initial data that must be supplied. This generates a flow
\begin{equation*}
\begin{array}{rccc}
\widetilde{F}_{d,nh}:& \left.T^*Q\right\vert_M \times \Lambda &\to    & \left.T^*Q\right\vert_M \times \Lambda\\
                 &(q_k,p_k,\lambda_k=\Lambda_k^1)         &\mapsto& (q_{k+1},p_{k+1},\lambda_{k+1}=\Lambda_k^s).
\end{array}
\end{equation*}

Of course, it is possible to apply the continuous fibre derivative $\mathbb{F}L$, and work only with $\Phi$ and forget about $\Psi$. We need only to introduce the variables $v_k, v_k^i \in TQ$ implicitly defined using the continuous Lagrangian by $p_k = D_2 L(q_k,v_k)$ and $p_k^i = D_2 L(q_k^i,v_k^i)$ respectively and change the constraint equations to $\Phi(q_k^i, v_k^i)$, thus generating the flow
\begin{equation*}
\begin{array}{rccc}
F_{d,nh}:& \left.TQ\right\vert_N \times \Lambda &\to    & \left.TQ\right\vert_N \times \Lambda\\
         &(q_k,v_k,\lambda_k=\Lambda_k^1)       &\mapsto& (q_{k+1},v_{k+1},\lambda_{k+1}=\Lambda_k^s).
\end{array}
\end{equation*}

A purely Hamiltonian version of this method would be:
\begin{subequations}
	\label{eq:Hamiltonian_nonholonomic_integrator}
	\begin{alignat}{2}
		q_{k+1} & = q_k + h \sum_{i = 1}^{s} b_{i} V_k^i, & \quad p_{k+1} & = p_k - h \sum_{i = 1}^{s} \hat{b}_{i} W_k^i,\\
		Q_k^i & = q_k + h \sum_{j = 1}^{s} a_{i j} V_k^j, & P_k^i & = p_k - h \sum_{j = 1}^{s} \hat{a}_{i j} W_k^j,\\
		V_k^i & = D_2 H(Q_k^i, P_k^i), & W_k^i & = D_1 H(Q_k^i, P_k^i) - \left\langle \Lambda_k^i, \flat_H\left( D_2 \Psi\right)(Q_k^i, P_k^i)\right\rangle , \vphantom{\sum_{j = 1}^{s}}\\
		q_k^i & = Q_k^i \vphantom{\sum_{j = 1}^{s}}, & p_k^i & = p_k - h \sum_{j = 1}^{s} a_{i j} W_k^j \vphantom{\sum_{j = 1}^{s}},
	\end{alignat}
	\begin{equation}
		\Psi(q_k^i, p_k^i) = 0 \vphantom{\sum_{j = 1}^{s}}.
	\end{equation}
\end{subequations}

The method admits a similar interpretation as its holonomic counterpart:
\begin{align}
0 &= \delta \sum_{k = 0}^{N-1} L_d (q_k, q_{k+1}) + \sum_{k = 0}^{N-1} \left[ \left\langle f_{d,nh}^{-}(q_k, \lambda_{k}, q_{k+1}, \lambda_{k+1}), \delta q_k\right\rangle \right. +\nonumber\\
&+ \left. \left\langle f_{d,nh}^{+}(q_k, \lambda_{k}, q_{k+1}, \lambda_{k+1}), \delta q_{k+1}\right\rangle \right] + h \sum_{k = 0}^{N-1} \sum_{i = 1}^{s} b_i \left\langle\delta \Lambda_k^i, \Phi\left(q_k^i, v_k^i\right)\right\rangle
\label{eq:discrete_constrained_variation_nh}
\end{align}
where now eqs.\eqref{eq:holonomic_forcing} become:
\begin{subequations}
	\label{eq:nonholonomic_forcing}
	\begin{alignat}{2}
		f_{d,nh}^{-}(q_k, \lambda_k, q_{k+1}, \lambda_{k+1}) &= h \sum_{i = 1}^{s} b_i \left\langle\Lambda_k^i, D_2\Phi\left(Q_k^i,V_k^i\right)\frac{\partial Q_k^i}{\partial q_{k}}\right\rangle\\
		f_{d,nh}^{+}(q_k, \lambda_k, q_{k+1}, \lambda_{k+1}) &= h \sum_{i = 1}^{s} b_i \left\langle\Lambda_k^i, D_2\Phi\left(Q_k^i,V_k^i\right) \frac{\partial Q_k^i}{\partial q_{k+1}}\right\rangle
	\end{alignat}
\end{subequations}
with $\lambda_k = \Lambda_k^1$ and $\lambda_{k+1} = \Lambda_k^s$. Note that once $q_k, \lambda_k$ and $q_{k+1}$ are set, $\lambda_{k+1}$ is fixed by the equations of the integrator. Still, these allow us to define discrete nonholonomic Legendre transforms over the solutions of the integrator:
\begin{eqnarray*}
\mathbb{F}L_{d,nh}^{-}(q_0,\lambda_0,q_1,\lambda_1) &=& (q_0, p_0 = - D_1 L_d(q_0,q_1) - f_{d,nh}^{-}(q_k, \lambda_{k}, q_{k+1}, \lambda_{k+1}),\lambda_0)\\
\mathbb{F}L_{d,nh}^{+}(q_0,\lambda_0,q_1,\lambda_1) &=& (q_1, p_1 = D_2 L_d(q_0,q_1) + f_{d,nh}^{+}(q_k, \lambda_{k}, q_{k+1}, \lambda_{k+1}),\lambda_1)
\end{eqnarray*}
This interpretation lends itself to a nonholonomic Hamilton-Jacobi viewpoint which we hope we will be able to explore in the future. It also leads to the possibility of diagrammatically representing our algorithm as:
\begin{center}
\small
\begin{tikzcd}[column sep=tiny, row sep=huge]
		Q \times \Lambda \times Q \times \Lambda :& & (q_0,\lambda_0,q_1,\lambda_1) \arrow[dl, mapsto, "\mathbb{F}L_{d,nh}^{-}"'] \arrow[dr, mapsto, "\mathbb{F}L_{d,nh}^{+}"] \arrow[rr, mapsto, "F_{L_d^{\Lambda}}"] & & (q_1,\lambda_1,q_2,\lambda_2) \arrow[dl, mapsto, "\mathbb{F}L_{d,nh}^{-}"'] \arrow[dr, mapsto, "\mathbb{F}L_{d,nh}^{+}"] &\\
		T^*Q \times \Lambda :& (q_0,p_0,\lambda_0) \arrow[rr, mapsto, "\widetilde{F}_{d,nh}"] \arrow[d, mapsto, "\mathbb{F}L^{-1}"] & & (q_1,p_1,\lambda_1) \arrow[rr, mapsto, "\widetilde{F}_{d,nh}"] \arrow[d, mapsto, "\mathbb{F}L^{-1}"] & & (q_2,p_2,\lambda_2) \arrow[d, mapsto, "\mathbb{F}L^{-1}"]\\
		TQ \times \Lambda :& (q_0,v_0,\lambda_0) \arrow[rr, mapsto, "F_{d,nh}"] & & (q_1,v_1,\lambda_1) \arrow[rr, mapsto, "F_{d,nh}"] & & (q_2,v_2,\lambda_2)
\end{tikzcd}
\end{center}
where $F_{L_d^{\Lambda}}$ is implicitly defined to close the diagram.

We can prove the following theorems concerning the convergence of the method (please, refer to appendix \ref{sec:app_num_sol} for the definition of the symplifying assumptions and the function $\mathcal{R}(\infty)$):
\begin{theorem}\label{thm:local_convergence}
\textbf{Local convergence.}
Assume an $s$-stage Runge-Kutta method $(a_{i j}, b_{j})$ whose coefficients satisfy hypotheses \ref{itm:H1}, \ref{itm:H2} and \ref{itm:H3} together with the symplifying assumptions $B(p), C(q)$ and $D(r)$. Assume also that $(\hat{a}_{i j}, \hat{b}_{j})$ represents the symplectic conjugate of the former (consequently satisfying \ref{itm:H1'} and \ref{itm:H2'} together with the symplifying assumptions $\widehat{B}(p), \widehat{C}(r)$ and $\widehat{D}(q)$). Then the the integrator defined by eqs.\eqref{eq:nonholonomic_integrator} for a regular nonholonomic Lagrangian system $(L, \Phi)$ with $\Psi = \Phi \circ \mathbb{F}L^{-1}$ with consistent initial condition $(q_0,p_0,\lambda_0)$, i.e. such that $\Psi(q_0,p_0) = 0$ and with $\lambda_0$ determined from the continuous problem, satisfies the local convergence estimates:
\begin{subequations}
	\label{eq:local_error_components}
	\begin{alignat}{2}
		q_1 - q(t_0 + h) &= \mathcal{O}(h^{\min(p, q + r + 1) + 1})\\
		p_1 - p(t_0 + h) &= \mathcal{O}(h^{\min(p, 2 q, q + r) + 1})\\
		\lambda_1 - \lambda(t_0 + h) &= \mathcal{O}(h^{q})
	\end{alignat}
\end{subequations}
where $(q(t), p(t), \lambda(t))$ is the exact solution of the nonholonomic problem.
\end{theorem}

\begin{theorem}\label{thm:global_convergence}
\textbf{Global convergence.}
In addition to the hypotheses of theorem \ref{thm:local_convergence}, suppose that $\left\Vert \mathcal{R}_A(\infty)\right\Vert \leq 1$, where $\mathcal{R}_A$ denotes the stability function of the method $(a_{i j}, b_j)$, and $q \geq 1$ if $\mathcal{R}_{A}(\infty)$. Then for $t_N - t_0 = N h \leq C$, where $C$ is some constant, the global error satisfies:
\begin{subequations}
	\label{eq:global_error_components}
	\begin{alignat}{2}
		q_N - q(t_N) &= \mathcal{O}(h^{\min(p, q + r + 1)})\label{eq:global_error_q}\\
		p_N - p(t_N) &= \mathcal{O}(h^{\min(p, 2 q, q + r)})\label{eq:global_error_p}\\
		\lambda_N - \lambda(t_N) &= \left\lbrace
			\begin{array}{rl}
			\mathcal{O}(h^{q}) & \text{if } -1 \leq \mathcal{R}_A(\infty) < 1,\\
			\mathcal{O}(h^{q - 1}) & \text{if } \mathcal{R}_A(\infty) = 1.
			\end{array} \right.\label{eq:global_error_lambda}
	\end{alignat}
\end{subequations}
\end{theorem}

\begin{proof}
Both result from the application of \cite{SatoMartinDeAlmagro}, Theorems 3.6-7 to the particular case of nonholonomic mechanical equations.
\end{proof}

\begin{corollary}
\label{cor:global_error_Lobatto}
The global error for the Lobatto IIIA-B method applied to eqs.\eqref{eq:nonholonomic_integrator} is:
\begin{subequations}
	\label{eq:Lobatto_error_components}
	\begin{alignat}{2}
		q_N - q(t_N) &= \mathcal{O}(h^{\min(2 s - 2)})\label{eq:Lobatto_error_q}\\
		p_N - p(t_N) &= \mathcal{O}(h^{\min(2 s - 2)})\label{eq:Lobatto_error_p}\\
		\lambda_N - \lambda(t_N) &= \left\lbrace
			\begin{array}{rl}
			\mathcal{O}(h^{s}) & \text{if $s$ even},\\
			\mathcal{O}(h^{s - 1}) & \text{if $s$ odd}.
			\end{array} \right.\label{eq:Lobatto_error_lambda}
	\end{alignat}
\end{subequations}
\end{corollary}

\begin{proof}
To prove this it suffices to substitute $p = 2 s - 2$, $q = s$, $r = s - 2$ and $\mathcal{R}_A(\infty) = (-1)^{s - 1}$ in the former theorem.
\end{proof}

These prove that the order of our method on the submanifold $M$ corresponds to the expected order.

The equations of the integrator should remind the reader of eqs. \ref{eq:holonomic_integrator}, and it is so by construction. The inspiration and intuition behind the proposed method comes precisely from the Lobatto IIIA-B pair mentioned before and we will now proceed to explain how we arrived at these equations.

\subsection{Origin and idea behind the algorithm}
How would the nonholonomic SHAKE would look like? Let us look back at the holonomic case and in particular at eq.\eqref{eq:discrete_constrained_variation}. We see that we need to determine an adequate forcing and impose the constraints.

The natural way to extend the forcing found in the holonomic case, eqs.\eqref{eq:holonomic_forcing}, would be as was already shown in eqs.\eqref{eq:nonholonomic_forcing}. If we apply the trapezoidal rule, these simplify to:
\begin{align*}
f_d^{-}(q_k, q_{k+1}, \lambda_{k}, \lambda_{k+1}) &= \frac{h}{2} \left\langle \lambda_{k}, D_2\Phi\left(q_k, \frac{q_{k+1}-q_{k}}{h}\right)\right\rangle,\\
f_d^{+}(q_k, q_{k+1}, \lambda_{k}, \lambda_{k+1}) &= \frac{h}{2} \left\langle \lambda_{k+1}, D_2\Phi\left(q_{k+1},\frac{q_{k+1}-q_{k}}{h}\right)\right\rangle.
\end{align*}

As for the constraint, we could enforce $\Phi(Q_k^i,V_k^i) = 0$, which in the trapezoidal case would lead to:
\begin{equation*}
\Phi\left(q_{k},\frac{q_{k+1}-q_{k}}{h}\right) = \Phi\left(q_{k+1},\frac{q_{k+1}-q_{k}}{h}\right) = 0, \quad \forall k
\end{equation*}

Note that this would not warrant that our integrator would preserve the continuous constraint, $\Phi(q_{k},v_k) = \Phi(q_{k+1},v_{k+1}) = 0$, where $v_k = \mathbb{F}L^{-1}(p_k)$. Also this is a direct discretization of the constraint manifold, which feels both arbitrary and rough. The more sensible option is indeed to impose the preservation of the continuous constraint, which we prefer to impose as $\Psi(q_k,p_k) = \Psi(q_{k+1},p_{k+1}) = 0$. Thus the integrator becomes:
\begin{align*}
D_2 L_d(q_{k-1}, q_{k}) + D_1 L_d(q_{k}, q_{k+1}) & = - h\left\langle \lambda_{k}, D_2\Phi\left(q_k, \frac{q_{k+1}-q_{k}}{h}\right)\right\rangle\\
\Psi(q_k, p_k) & = 0, \quad \forall k = 0, ..., N
\end{align*}
The discrete Euler-Lagrange equations are still a matching of momenta,
\begin{equation*}
p_k^{-}(q_{k-1},q_k,\lambda_k) = p_k^{+}(q_{k},q_{k+1},\lambda_k)
\end{equation*}
and we can chose either $p_k^{-}$ or $p_k^{+}$ to impose the constraint without difference ( see \cite{dLdDSM04}).

The question now is, how does this method generalize to higher order? When we use the Lobatto IIIA method in eqs.\eqref{eq:discrete_Hamilton-Pontriagyn_action} and apply the discrete Hamilton-Pontryagin principle we automatically obtain the Lobatto IIIB method in eqs.\eqref{eq:symplectic_partitioned_integrator} for the $P_k^i$ momenta. The first method is a continuous collocation method (fig. \ref{fig:continuous_collocation}) and the second is a discontinuous collocation method (fig. \ref{fig:discontinuous_collocation}).

\begin{figure}
\centering
\begin{subfigure}{.5\textwidth}
  \centering
  \includegraphics[width=5.2cm, clip=true, trim=20mm 105mm 40mm 6mm]{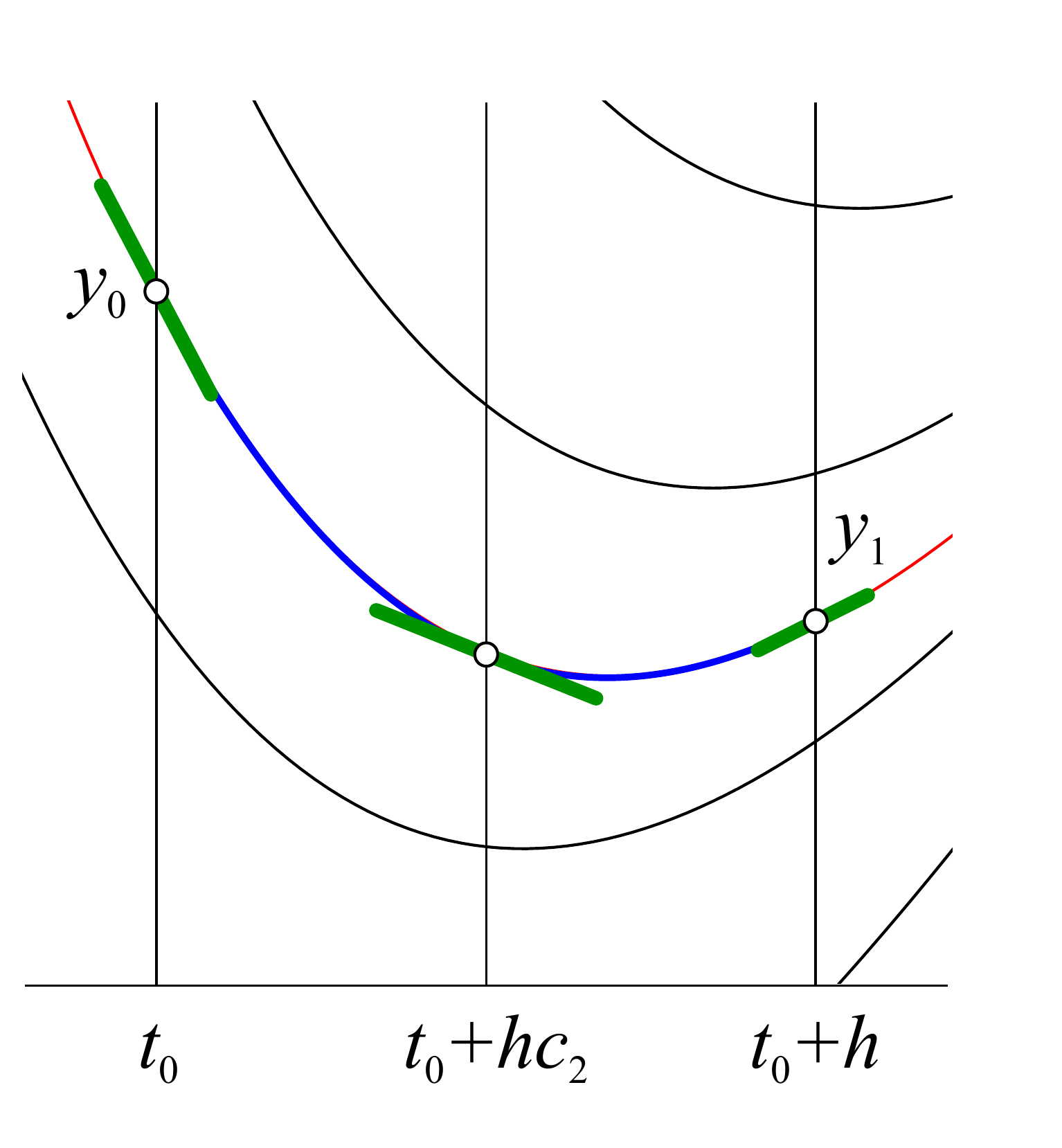}
  \caption{\scriptsize{Continuous collocation}}
  \label{fig:continuous_collocation}
\end{subfigure}%
\begin{subfigure}{.5\textwidth}
  \centering
  \includegraphics[width=5.2cm, clip=true, trim=20mm 105mm 40mm 6mm]{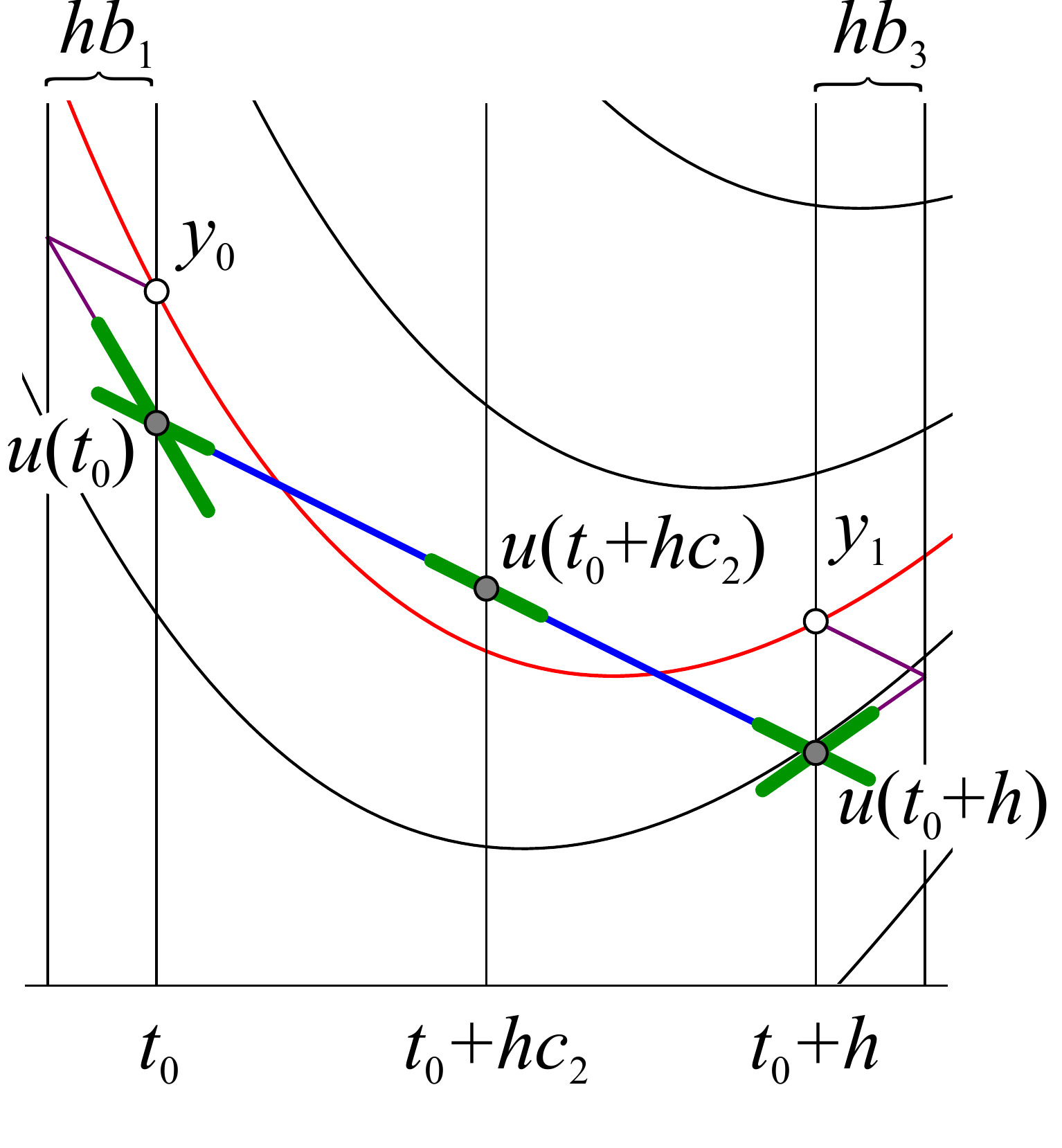}
  \caption{\scriptsize{Discontinuous collocation}}
  \label{fig:discontinuous_collocation}
\end{subfigure}
\caption{\scriptsize{On the left we have a continuous collocation method where the collocation polynomial (in blue) tries to give a good continuous approximation of the solution (in red). On the right we have a discontinuos collocation method applied to the same problem. The collocation polynomial $u(t)$ is a poor approximation of the solution, particularly at collocation points, but it allows us to compute $y_1$ to the same order as the continuous method.}}
\label{fig:collocation_comparison}
\end{figure}

Continuous collocation provides a relatively good (cf. eq.\eqref{eq:continuous_degree_of_approximation}) continuous approximation of the solution whereas discontinuous (cf. eq.\eqref{eq:discontinuous_degree_of_approximation}) offers a poorer one, forming something akin to a scaffolding to obtain the actual nodal values $p_k$ instead of trying to provide a good approximation on the interval (fig. \ref{fig:collocation_comparison}). For a better exposition of the matter we recommend the reader to check \cite{hairer} although in the appendices \ref{sec:app_num_sol} and \ref{sec:app_ord_conds} the reader can find a brief summary.

From the study of order conditions we know that at inner stages the convergence is related to the $C(q)$ simplifying assumption that each method satisfies, which are $C(s)$ and $\widehat{C}(s-2)$, thus coinciding in this case with the former estimates. At nodal points, the convergence is related to the $B(p)$ simplifying assumption (superconvergence) which both satisfy for $p = 2s - 2$.

This means that, in our case, $Q_k^i$ is an approximation of order $s$ to $q(t_k^i)$, whereas $P_k^i$ is an approximation of order $s-2$ to $p(t_k^i)$ and if we try to enforce a nonholonomic constraint using these values we will be asking our solution to lie far away from the corresponding point that the real trajectory would pass through. Thus, the question would be whether we can generate better approximations of $p(t_k^i)$ and whether we can do this cheaply. As it turns out, the answer is yes. We need only to apply the same Lobatto IIIA quadrature rule to the momenta and, better yet, we can reuse the $W_k^i$ values used for the determination of $P_k^i$. As a side effect this provides us with a better continuous approximation of $p$ (order $s-1$, cf. proposition \ref{prop:continuous_approximation_mixed}), although not as good as the continuous approximation of $q$ (see fig.\ref{fig:continuous_from_discontinuous}).

\begin{figure}
\centering
  \includegraphics[width=5.2cm, clip=true, trim=20mm 105mm 40mm 6mm]{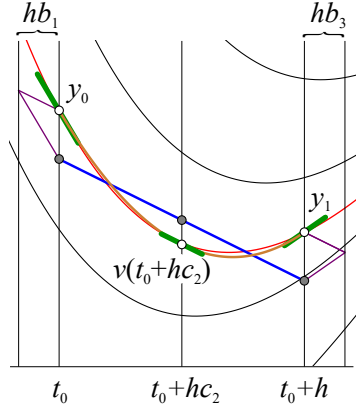}
  \caption{\scriptsize{Continuous collocation polynomial $v(t)$ (in orange) obtained from discontinuous collocation data.}}
  \label{fig:continuous_from_discontinuous}
\end{figure}

Surprisingly the new values we obtain using this method, which we call $p_k^i$, are an approximation of order $s$ to our desired $p(t_k^i)$. This is so because the Lobatto IIIA-B pair satisfies the mixed symplifying assumption $C\widehat{C}(s)$ \cite{Jay1996}.

Intuitively speaking, enforcing the constraints with these, should give better results. These better estimates of $p(t_k^i)$ can be obtained for any other symplectic integrator in holonomic systems, but there they become completely decoupled and can be obtained \emph{a posteriori}.

\section{Lie group integrators}\label{sec:liegroups}
\subsection{Hamilton-Pontryagin formulation and numerical methods in Lie groups}
In section \ref{sec:var_integrators} we used the Hamilton-Pontryagin action as our entry point for high-order variational numerical methods. This will also be the case for Lie groups, where we will follow the approach of \cite{MR2496560}. 

In the continuous case we can work exactly as we did in section \ref{sec:var_integrators} by simply prescribing that our configuration manifold is a Lie group $G$ \cite{CeHaOw,ismu}. The continuous problem becomes more interesting when we consider the possibility of left or right trivialization (we will only consider the first, but computations are analogous). This is a recurring theme in Lie group systems, where the Lagrangian usually is $G$-invariant, i.e.  for $\forall h \in G$ and $(g,v_g) \in TG$, $L(g,v_g) = L(hg, T_g L_{h} v_g)$. App,lying standar reduction theory we  canderive the corresponding Euler-Poincar\'e equations from a reduced Lagrangian defined on ${\mathfrak g}$  \cite{marsden3}.

In any  case, we can define a new trivialized Lagrangian $\ell : G \times {\mathfrak{g}} \to \mathbb{R}$, $\ell(g, \eta) = L(g, T_e L_{g} \eta)$ where $e$ is the neutral element of the group and ${\mathfrak g}=T_eG$. Manipulating the integrand of the Hamilton-Pontryagin action in eq.\eqref{eq:continuous_Hamilton-Pontriagyn_action} we obtain:
\begin{equation}
\ell(g(t),\eta(t)) + \left\langle \mu(t), T_{g(t)} L_{g^{-1}(t)} \dot{g}(t) - \eta(t)\right\rangle
\end{equation}

Therefore, it is natural to consider   the space of curves
\begin{align*}
&C^1 ((g_a,\eta_a,\mu_a),(g_b,\eta_b,\mu_b), [a,b])\\
&= \left\lbrace (g,\eta,\mu): [a, b] \rightarrow G \times \mathfrak{g} \times \mathfrak{g}^* \,|\, g \in C^2([a, b]), \eta, \mu \in C^1([a, b]),\right.\\
&\left.\hspace{0.62cm}(g,\eta,\mu)(a) = (g_a,\eta_a,\mu_a), (g,\eta,\mu)(b) = (g_b,\eta_b,\mu_b) \right\rbrace,
\end{align*}
and define the trivialized functional $\mathcal{J_{HP}}: C^1 ((g_a,\eta_a,\mu_a),(g_b,\eta_b,\mu_b), [a,b]) \to \mathbb{R}$:
\begin{equation}
\mathcal{J_{HP}}(g,\eta,\mu) = \int_0^{h} \left[ \ell(g(t),\eta(t)) + \left\langle \mu(t), T_{g(t)} L_{g^{-1}(t)}\dot{g}(t) - \eta(t)\right\rangle \right] \mathrm{d}t.
\label{eq:trivialized_Hamilton-Pontriagyn_action}
\end{equation}

It can be checked that for the extended curve $(g,\eta,\mu)$ to be a compatible critical point of the action, $\mathrm{d}\mathcal{J_{HP}}(g,\eta,\mu) = 0$, then the following equations must be satisfied:
	\begin{align*}
		\frac{\mathrm{d}\mu(t)}{\mathrm{d}t} &= \mathrm{ad}_{\eta(t)}^* \mu(t) + \left(L_{g(t)}\right)^{*} D_1 \ell(g(t),\eta(t)),\\
		\mu(t) &= D_2 L(g(t),\eta(t)),\\
		\frac{\mathrm{d}g(t)}{\mathrm{d}t} &= T_eL_{g(t)} \eta(t), \quad \forall t \in [0, h].
	\end{align*}
These are the corresponding \emph{Euler-Poincar{\'e}  and Lie-Poisson equations} of the problem, and again the $\mu$ coincide with the trivialized canonical momenta on ${\mathfrak g}^*$. Inserting the second equation into the first we may cast these equations in the more familiar form:
\begin{equation*}
\frac{\mathrm{d}}{\mathrm{d}t} \left(\frac{\partial \ell}{\partial \eta}\right) - \mathrm{ad}_{\eta(t)}^* \left(\frac{\partial \ell}{\partial \eta}\right) = L_{g(t)}^{*} \frac{\partial \ell}{\partial g}
\end{equation*}

If variations are taken without imposing fixed end-point conditions we obtain the boundary terms:
\begin{equation*}
\left.\left\langle \mu(t), \zeta(t) \right\rangle\right\vert_0^h = \left\langle \mu(h), \zeta(h)\right\rangle - \left\langle \mu(0), \zeta(0)\right\rangle
\end{equation*}
where $\zeta(t) = T_{g(t)} L_{g^{-1}(t)} \delta g(t)$.

In order to obtain numerical integrators we need to discretise this action, and in particular we need to know how to properly discretise the kinematic constraint $\dot{g}(t) = v(t) = T_{e} L_{g(t)} \eta(t)$. Again, we will consider RK methods for this but we must be mindful of the manifold structure of the configuration manifold.

The work of \cite{kaas} was one of the firsts to tackle the problem of generalizing RK-type algorithms to more general manifolds and specifically to the Lie group case. In this work the author acknowledges the fact that one needs a vector space structure in order to apply a RK method, as they rely heavily on its  linear structure. In order to address this issue, one can exploit the structure of Lie group of $G$ and its relation with its Lie algebra $\mathfrak{g}$, which is a vector space. In summary, the resulting methods, commonly called Runge-Kutta-Munthe-Kaas (RKMK) methods, move back and forth from the group to the algebra in order to respect the geometry of the group.

To perform these back-and-forth operations it is convenient to introduce the following concept  \cite{Absil}.
\begin{definition} 
A \textbf{retraction} on a manifold $Q$ is a smooth mapping $R: TQ \to Q$ satisfying the following properties. If we denote by $R_q = \left.R\right\vert_{T_q Q} $ the restriction of $R$ to $T_q Q$ then: 
\begin{itemize}
\item $R_q(0_q) = q$, 
\item Identifying $T_{0_q} T_q Q \equiv  T_q Q$ then 
$T_{0_q} R_q = \mathrm{id}_{T_q Q}$.
\end{itemize}
\end{definition}

In the case of Lie groups, we will consider the local diffeomorphism $\tau: \mathfrak{g} \to U_e \subset G$, where $U_e$ is a neighbourhood of the identity element, as retraction maps. The most common instances of these are the exponential map, $\exp$, and the Cayley map, $\mathrm{cay}$ (in the case of quadratic Lie groups).

Assuming that $G$ is connected, we will be able to translate a neighbourhood of any point to  $U_e$ and from $U_e$ to $\mathfrak{g}$ and back thanks to $\tau^{-1}$ and $\tau$. Not only that, but this will be also possible in $\mathbb{T}G=TG\oplus T^*G$, which is what we need for our mechanical problems.

For some $h \in G$, the complete geometric scheme is as follows:
\begin{center}
	\begin{tikzcd}[column sep=huge,row sep=huge]
		\mathbb{T}\mathfrak{g} \arrow[d, "\pi_\mathfrak{g}"]  \arrow[r, shift left=1mm, "\mathbb{T}\tau"] \arrow[r, shift right=1mm, leftarrow, "\mathbb{T}\tau^{-1}"'] & \mathbb{T}U_e \arrow[d, "\pi_U"] \arrow[r, shift left=1mm, "\mathbb{T}L_{h}"] \arrow[r, shift right=1mm, leftarrow, "\mathbb{T}L_{h}^{-1} = \mathbb{T}L_{h^{-1}}"'] & \mathbb{T}G \arrow[d, "\pi_G"]\\
		\mathfrak{g} \arrow[r, shift left=1mm, "\tau"] \arrow[r, shift right=1mm, leftarrow, "\tau^{-1}"'] & U_e \arrow[r, shift left=1mm, "L_{h}"] \arrow[r, shift right=1mm, leftarrow, "L_{h}^{-1} = L_{h^{-1}}"'] & G
	\end{tikzcd}
\end{center}
with 
\begin{eqnarray*}
	\mathbb{T}\tau(\xi, \eta_\xi, \mu_{\xi})&=&(\tau(\xi), T_{\xi}\tau(\xi),(T_{\tau(\xi)}\tau^{-1})^* \mu_{\xi})
	\end{eqnarray*}
where $\xi \in  \mathfrak{g}$, $\eta_\xi \in T_\xi \mathfrak{g} \cong \mathfrak{g}$ and $\mu_{\xi}\in T^*_\xi \mathfrak{g} \cong \mathfrak{g}^*$  and similar definitions for the other maps.

Assume we work with adapted coordinates $(g, v_g, p_g) \in \mathbb{T}G$ and $(\xi, \eta_{\xi}, \mu_{\xi}) \in \mathbb{T} \mathfrak{g}$. According to this diagram, if $h$ is such that $L_{h^{-1}} g \in U_e$, we find the following correspondences:
\begin{align*}
&\mathbb{T}_{L_{h^{-1}} g} \tau^{-1} \mathbb{T}_{g}L_{h^{-1}}(g,v_g,p_g)\\
&\quad = \left( \tau^{-1}\left(L_{h^{-1}} g\right), \mathrm{d}^L \tau^{-1}_{\tau^{-1} (L_{h^{-1}} g)} T_{g} L_{g^{-1}} v_g,  \left(\mathrm{d}^L \tau_{\tau^{-1}(L_{h^{-1}} g)}\right)^* \left(T_{e} L_{g}\right)^* p_g\right)\\
&\quad = \left( \xi, \eta_{\xi}, \mu_{\xi} \right)\\
&\mathbb{T}_{\tau(\xi)}L_{h} \mathbb{T}_{\xi} \tau (\xi,\eta_{\xi}, \mu_{\xi})\\
&\quad = \left( L_h \tau(\xi), T_e L_{L_h \tau(\xi)} \mathrm{d}^L \tau_{\xi} \eta_{\xi}, \left(T_{L_h \tau(\xi)} L_{\left(L_h \tau(\xi)\right)^{-1}}\right)^* \left(\mathrm{d}^L \tau^{-1}_{\xi}\right)^* \mu_{\xi} \right)\\
&\quad = \left( g, v_g, p_g \right)
\end{align*}
where $\mathrm{d}^{L}\tau: \mathfrak{g} \times \mathfrak{g} \to \mathfrak{g}$ is the left-trivialised tangent to $\tau$ defined by the relation $(T_\xi \tau) v_{\xi} = T_e L_{\tau(\xi)} \left(\mathrm{d}^L \tau_{\xi} v_{\xi}\right)$, for $(\xi, v_{\xi}) \in T\mathfrak{g}$, and $\mathrm{d}^{L}\tau^{-1}$ is its inverse, defined by $(T_g \tau^{-1}) v_{g} = \mathrm{d}^L \tau_{\tau^{-1}({g})}^{-1} \left(T_g L_{g^{-1}} v_{g}\right)$ \cite{MR2496560,BoMa}.

Let us also take this opportunity to define $\mathrm{dd}^{L}\tau: \mathfrak{g} \times \mathfrak{g} \times \mathfrak{g} \to \mathfrak{g}$ the second left-trivialised tangent, which will be necessary for later derivations. This is a linear map in the second and third variables such that $\partial_{\xi_{k}^i}\left(\mathrm{d}^{L}\tau_{\xi_k^i}\eta_k^i\right) \delta{\xi_k^i} = \mathrm{d}^{L}\tau_{\xi_k^i} \mathrm{dd}^{L}\tau_{\xi_k^i}(\eta_k^i, \delta\xi_k^i)$. It appears naturally when representing elements $(g, v_g, a_{v_g}) \in T^{(2)}G$, the second order tangent bundle of $G$ (see \cite{ManuelCampos}), with elements of $(\xi, \eta_  {\xi}, \zeta_{\eta_{\xi}}) \in T^{(2)}\mathfrak{g}$, which using matrix notation becomes:
\begin{equation*}
\left(e, g^{-1} v_g, g^{-1} a_{v_g} - g^{-1} v_g g^{-1} v_g \right) \mapsto \left(0, \mathrm{d}^L\tau_{\xi}\eta_{\xi}, \mathrm{d}^L\tau_{\xi}\left[\zeta_{\eta_{\xi}} +  \mathrm{dd}^L\tau_{\xi}\left(\eta_{\xi},\eta_{\xi}\right)\right] \right)
\end{equation*}

With this, RKMK methods for our kinematic constraint can be obtained (see appendix \ref{sec:app_lie_group_collocation}). Using matrix notation, our continuous constraint becomes $\dot{g}(t) = g(t) \eta(t)$, and its discrete version can be written as:
\begin{align*}
\tau^{-1}((g_k)^{-1} G_k^i) &= h \sum_{j=1}^s a_{i j} \mathrm{d}^{L}\tau^{-1}_{\tau^{-1}((g_k)^{-1} G_k^j)} (G_k^j)^{-1} V_k^j,\\
\tau^{-1}((g_k)^{-1} g_{k+1}) &= h \sum_{j=1}^s b_{j} \mathrm{d}^{L}\tau^{-1}_{\tau^{-1}((g_k)^{-1} G_k^j)} (G_k^j)^{-1} V_k^j.
\end{align*}
which, if $V_k^i = G_k^i \mathrm{d}^{L}\tau_{\tau^{-1}((g_k)^{-1} G_k^i)} \mathrm{H}_k^i$ (where $\mathrm{H}$ corresponds to the variable $\eta$), reduces to
\begin{align*}
\tau^{-1}((g_k)^{-1} G_k^i) &= h \sum_{j=1}^s a_{i j} \mathrm{H}_k^j,\\
\tau^{-1}((g_k)^{-1} g_{k+1}) &= h \sum_{j=1}^s b_{j} \mathrm{H}_k^j.
\end{align*}

\subsection{Variational Lie group integrators}
With the above expressions, the discrete Hamilton-Pontryagin equation \eqref{eq:discrete_Hamilton-Pontriagyn_action} on a Lie group transforms into:
\begin{align*}
\left(\mathcal{J_{HP}}\right)_d &= \sum_{k = 0}^{N-1} \sum_{i = 1}^s h b_i \left[ \vphantom{\sum_{i = 1}^s} L\left(G_k^i, V_k^i\right)\right.\\
&+ \left\langle \tilde{P}_k^i, \frac{1}{h} G_k^i \tau^{-1}((g_k)^{-1} G_k^i) - G_k^i \sum_{j=1}^s a_{i j} \mathrm{d}^{L}\tau^{-1}_{\tau^{-1}((g_k)^{-1} G_k^j)} (G_k^j)^{-1} V_k^j\right\rangle\\
&+ \left.\left\langle \tilde{p}_{k+1}, \frac{1}{h} g_{k+1} \tau^{-1}((g_k)^{-1} g_{k+1}) - g_{k+1} \sum_{j=1}^s b_{j} \mathrm{d}^{L}\tau^{-1}_{\tau^{-1}((g_k)^{-1} G_k^j)} (G_k^j)^{-1} V_k^j\right\rangle \right]
\end{align*}
where $\tilde{P}_k^i \in T_{G_k^i}^{*} G$ and $\tilde{p}_{k+1} \in T_{g_{k+1}}^{*} G$. We can shorten this expression to:
\begin{align*}
\left(\mathcal{J_{HP}}\right)_d &= \sum_{k = 0}^{N-1} \sum_{i = 1}^s h b_i \left[ \vphantom{\sum_{i = 1}^s} L\left(G_k^i, V_k^i\right)\right.\\
&+ \left\langle \widetilde{\mathrm{M}}_k^i, \frac{1}{h} \tau^{-1}((g_k)^{-1} G_k^i) - \sum_{j=1}^s a_{i j} \mathrm{d}^{L}\tau^{-1}_{\tau^{-1}((g_k)^{-1} G_k^j)} (G_k^j)^{-1} V_k^j\right\rangle\\
&+ \left.\left\langle \tilde{\mu}_{k+1}, \frac{1}{h} \tau^{-1}((g_k)^{-1} g_{k+1}) - \sum_{j=1}^s b_{j} \mathrm{d}^{L}\tau^{-1}_{\tau^{-1}((g_k)^{-1} G_k^j)} (G_k^j)^{-1} V_k^j\right\rangle \right]
\end{align*}
where $\tilde{\mu}_k, \widetilde{\mathrm{M}}_k^i \in \mathfrak{g}^*$.

Using elements $(\Xi_{k}^i, \mathrm{H}_{k}^i) \in T\mathfrak{g}$ as representatives of $(G_k^i, V_k^j) \in TG$, we find:
\begin{align*}
\left(\mathcal{J_{HP}}\right)_d &= \sum_{k = 0}^{N-1} \sum_{i = 1}^s h b_i \left[ \vphantom{\sum_{i = 1}^s} L\left(g_k \tau(\Xi_k^i), g_k \tau(\Xi_k^i) \mathrm{d}^{L}\tau_{\Xi_k^i} \mathrm{H}_k^i\right)\right.\\
&+ \left.\left\langle \widetilde{\mathrm{M}}_k^i, \frac{1}{h} \Xi_{k}^i - \sum_{j=1}^s a_{i j} \mathrm{H}_k^j\right\rangle + \left\langle \tilde{\mu}_{k+1}, \frac{1}{h} \tau^{-1}((g_k)^{-1} g_{k+1}) - \sum_{j=1}^s b_{j} \mathrm{H}_k^j\right\rangle \right]
\end{align*}

Expressing the Lagrangian $L: TG\rightarrow \mathbb{R}$ as $\ell: G \times \mathfrak{g}\rightarrow \mathbb{R}$ using left translation, this expression simplifies to
\begin{align}
\left(\mathcal{J_{HP}}\right)_d &= \sum_{k = 0}^{N-1} \sum_{i = 1}^s h b_i \left[ \vphantom{\sum_{i = 1}^s} \ell\left(g_k \tau(\Xi_k^i), \mathrm{d}^{L}\tau_{\Xi_k^i} \mathrm{H}_k^i\right)\right.\label{eq:Lie_discrete_Hamilton-Pontriagyn_action}\\
&+ \left.\left\langle \widetilde{\mathrm{M}}_k^i, \frac{1}{h} \Xi_{k}^i - \sum_{j=1}^s a_{i j} \mathrm{H}_k^j\right\rangle + \left\langle \tilde{\mu}_{k+1}, \frac{1}{h} \tau^{-1}((g_k)^{-1} g_{k+1}) - \sum_{j=1}^s b_{j} \mathrm{H}_k^j\right\rangle \right].\nonumber
\end{align}
which is the discrete equivalent of \eqref{eq:trivialized_Hamilton-Pontriagyn_action} and the one we will work with from here on.

Similarly as we did in the vector space case, let us consider the space of 
$s$-stage \emph{ variationally partitioned Runge-Kutta-Munthe-Kaas} sequences (
\emph{$s$-stage VPRKMK sequences}):
\begin{align*}
&C_d^s (g_a, g_b)\\
&= \left\lbrace \left(g,\tilde{\mu},\left\lbrace \Xi^i, \mathrm{H}^i, \widetilde{\mathrm{M}}^i\right\rbrace_{i = 1}^s\right): \left\lbrace t_k\right\rbrace_{k=0}^N \rightarrow G \times \mathfrak{g}^* \times \left(\mathbb{T}\mathfrak{g}\right)^s \,|\, g(a) = g_a, g(b) = g_b\right\rbrace.
\end{align*}
Note that $\mathbb{T}\mathfrak{g} \cong \mathfrak{g} \times \mathfrak{g} \times \mathfrak{g}^*$.

Now we are in a position to state the Lie group analogue of theorem \ref{thm:VPRK_methods}. This will be a more general version of \cite{MR2496560}, theorem 4.9, which is order 2-bound. In fact ours is essentially equivalent to what is stated in \cite{BoMa}.

\begin{theorem}\label{thm:VPRKMK_methods}
Let $\ell : G \times \mathfrak{g} \to \mathbb{R}$ be a $C^l$ function with $l \geq 2$ and an $s$-stage VPRKMK sequence $c_d \in C_d^s (g_0, g_N)$. Then $c_d$ is a critical point of the discrete Hamilton-Pontryagin functional, $\left(\mathcal{J_{HP}}\right)_d$, if and only if for all $k = 0, ..., N-1$ and $i = 1, ..., s$ it satisfies
%\begin{subequations}
%\label{eq:Lie_symplectic_partitioned_integrator}
%\begin{alignat}{2}
%\end{alignat}
%\end{subequations}
\begin{align}
\Xi_k^i &= \tau^{-1}\left(g_k^{-1} G_k^i\right) = h \sum_{j = 1}^s a_{i j} \mathrm{H}_k^j,
\label{eq:inner_xi_equations}\\
\xi_{k,k+1} &= \tau^{-1}\left(g_k^{-1} g_{k+1}\right) = h \sum_{j = 1}^s b_{j} \mathrm{H}_k^j,
\label{eq:outer_xi_equations}\\
\mathrm{M}_k^i &= \mathrm{Ad}_{\tau(\xi_{k,k+1})}^* \left[ \mu_{k} + h \sum_{j = 1}^s b_j \left( \mathrm{d}^L\tau^{-1}_{-\Xi_k^j} - \frac{a_{j i}}{b_i} \mathrm{d}^{L}\tau^{-1}_{-\xi_{k,k+1}}\right)^* \mathrm{N}_k^j\right],
\label{eq:inner_mu_equations}\\
\mu_{k+1} &= \mathrm{Ad}_{\tau(\xi_{k,k+1})}^* \left[ \mu_{k} + h \sum_{j = 1}^s b_j \left(\mathrm{d}^L\tau^{-1}_{-\Xi_k^j}\right)^* \mathrm{N}_k^j\right];
\label{eq:outer_mu_equations}
\end{align}
where
\begin{align*}
\mathrm{N}_k^i &= \left(\mathrm{d}^L\tau_{\Xi_k^i}\right)^* L_{g_k \tau(\Xi_k^i)}^* D_1 \ell\left(g_k \tau(\Xi_k^i), \mathrm{d}^{L}\tau_{\Xi_k^i} \mathrm{H}_k^i\right),\\
\mathrm{M}_k^i &= \left(\mathrm{d}^{L}\tau^{-1}_{\xi_{k,k+1}}\right)^*\left[\mathrm{\Pi}_k^i + h \sum_{j = 1}^s \frac{b_j a_{j i}}{b_i} \left( \mathrm{dd}^{L}\tau_{\Xi_k^j} \right)^* \left( \mathrm{H}_k^j, \mathrm{\Pi}_k^j\right)\right],\\
\mathrm{\Pi}_k^i &= \left(\mathrm{d}^{L}\tau_{\Xi_k^i}\right)^* D_2 \ell\left(g_k \tau(\Xi_k^i), \mathrm{d}^{L}\tau_{\Xi_k^i} \mathrm{H}_k^i\right),\\
\mu_k &= \left(\mathrm{d}^{L}\tau^{-1}_{\xi_{k-1,k}}\right)^* \tilde{\mu}_k.
\end{align*}
\end{theorem}

\begin{proof}
We know that $c_d$ is a critical point of the discrete Hamilton-Pontryagin functional if and only if $\mathrm{d}\left(\mathcal{J_{HP}}\right)_d(c_d) (\delta c_d) = 0$, $\forall \delta c_d \in T_{c_d} C_d^s (g_0, g_N)$. Let us write $\delta c_d = \left(\delta g, \delta \tilde{\mu}, \left\lbrace \delta \Xi^i, \delta \mathrm{H}^i, \delta \widetilde{\mathrm{M}}^i\right\rbrace_{i = 1}^s\right)$ and compute each of the individual variations separated from each other:
\begin{align*}
\delta g &: \sum_{k = 0}^{N-1} \sum_{i = 1}^s h b_i \left[ \vphantom{\frac{1}{h}} \left\langle D_1 \ell\left(g_k \tau(\Xi_k^i), \mathrm{d}^{L}\tau_{\Xi_k^i} \mathrm{H}_k^i\right), \delta g_k \tau(\Xi_k^i) \right\rangle \right.\\
&+ \left.\left\langle \tilde{\mu}_{k+1}, \frac{1}{h} D\tau^{-1}((g_k)^{-1} g_{k+1}) (- (g_k)^{-1} \delta g_k (g_k)^{-1} g_{k+1} + (g_k)^{-1} \delta g_{k+1}) \right\rangle \right]
\end{align*}

As it is customary we will define new variations $\zeta_k = (g_k)^{-1} \delta g_{k} \in \mathfrak{g}$ and use these to rewrite this equation, together with the trivialised tangents and the short-hand $\xi_{k,k+1} = \tau^{-1}((g_k)^{-1} g_{k+1})$:
\begin{align*}
\delta g &: \sum_{k = 0}^{N-1} \sum_{i = 1}^s h b_i \left[ \vphantom{\frac{1}{h}} \left\langle D_1 \ell\left(g_k \tau(\Xi_k^i), \mathrm{d}^{L}\tau_{\Xi_k^i} \mathrm{H}_k^i\right), g_k \tau(\Xi_k^i) \tau(-\Xi_k^i)\zeta_k \tau(\Xi_k^i)\right\rangle \right.\nonumber\\
&+ \left.\left\langle \tilde{\mu}_{k+1}, \frac{1}{h} \mathrm{d}^{L}\tau_{\xi_{k,k+1}}^{-1} (- \mathrm{Ad}^{-1}_{\tau\left(\xi_{k,k+1}\right)}\zeta_k + \zeta_{k+1}) \right\rangle \right]\\
&= \sum_{k = 0}^{N-1} \left[ \left\langle h \sum_{i = 1}^s b_i \left(\mathrm{Ad}^{-1}_{\tau(\Xi_k^i)}\right)^* L_{g_k \tau(\Xi_k^i)}^* D_1 \ell\left(g_k \tau(\Xi_k^i), \mathrm{d}^{L}\tau_{\Xi_k^i} \mathrm{H}_k^i\right), \zeta_k \right\rangle\right.\nonumber\\
&- \left\langle \left(\mathrm{Ad}^{-1}_{\tau\left(\xi_{k,k+1}\right)}\right)^*\left(\mathrm{d}^{L}\tau_{\xi_{k,k+1}}^{-1}\right)^*\tilde{\mu}_{k+1}, \zeta_k \right\rangle\nonumber\\
&+ \left.\vphantom{\sum_{i = 1}^s}\left\langle \left(\mathrm{d}^{L}\tau_{\xi_{k,k+1}}^{-1}\right)^*\tilde{\mu}_{k+1}, \zeta_{k+1} \right\rangle \right]
\end{align*}
where we have used the order one condition. If we rearrange the sum so that terms with the same $\zeta_k$ appear together (the discrete analogue of integration by parts) we are left with:
\begin{align*}
\delta g &: \sum_{k = 1}^{N-1} \left[ \left\langle h \sum_{i = 1}^s b_i \left(\mathrm{Ad}^{-1}_{\tau(\Xi_k^i)}\right)^* L_{g_k \tau(\Xi_k^i)}^* D_1 \ell\left(g_k \tau(\Xi_k^i), \mathrm{d}^{L}\tau_{\Xi_k^i} \mathrm{H}_k^i\right) \right.\right.\\
&- \left.\left. \vphantom{\sum_{i = 1}^s} \left(\mathrm{Ad}^{-1}_{\tau\left(\xi_{k,k+1}\right)}\right)^*\left(\mathrm{d}^{L}\tau_{\xi_{k,k+1}}^{-1}\right)^* \tilde{\mu}_{k+1} + \left(\mathrm{d}^{L}\tau_{\xi_{k-1,k}}^{-1}\right)^*\tilde{\mu}_{k}, \zeta_{k} \right\rangle \right]\\
&+ \left\langle h \sum_{i = 1}^s b_i \left(\mathrm{Ad}^{-1}_{\tau(\Xi_k^i)}\right)^* L_{g_k \tau(\Xi_k^i)}^* D_1 \ell\left(g_0 \tau(\Xi_0^i), \mathrm{d}^{L}\tau_{\Xi_0^i} \mathrm{H}_0^i\right) \right.\\
&- \left.\vphantom{\sum_{i = 1}^s} \left(\mathrm{Ad}^{-1}_{\tau\left(\Xi_{0,1}\right)}\right)^*\left(\mathrm{d}^{L}\tau_{\Xi_{0,1}}^{-1}\right)^*\tilde{\mu}_{1}, \zeta_0\right\rangle + \left\langle \left(\mathrm{d}^{L}\tau_{\Xi_{N-1,N}}^{-1}\right)^*\tilde{\mu}_{N}, \zeta_N\right\rangle
\end{align*}

Identification of the boundary terms with their counterparts on the continuous realm suggests the change $\left(\mathrm{d}^{L}\tau_{\xi_{k-1,k}}^{-1}\right)^* \tilde{\mu}_{k} = \mu_k$.

%where we have also used the fact that $\mathrm{d}^{L}\tau_{\Xi}^{-1} \mathrm{Ad}^{-1}_{\tau(\Xi)} = \mathrm{d}^{L}\tau_{-\Xi}^{-1}$.

Moving on to a different variation:
\begin{align*}
\delta \Xi &: \sum_{k = 0}^{N-1} \sum_{i = 1}^s h b_i \left[ \vphantom{\frac{1}{h}} \left\langle D_1 \ell\left(g_k \tau(\Xi_k^i), \mathrm{d}^{L}\tau_{\Xi_k^i} \mathrm{H}_k^i\right), g_k D\tau(\Xi_k^i) \delta \Xi_{k}^i \right\rangle \right.\\
&+ \left\langle D_2 \ell\left(g_k \tau(\Xi_k^i), \mathrm{d}^{L}\tau_{\Xi_k^i} \mathrm{H}_k^i\right), \partial_{\Xi_{k}^i}\left(\mathrm{d}^{L}\tau_{\Xi_k^i}\mathrm{H}_k^i\right) \delta{\Xi_k^i} \right\rangle + \left.\left\langle \frac{1}{h} \widetilde{\mathrm{M}}_k^i, \delta \Xi_{k}^i\right\rangle \right]
\end{align*}

Using the definition of $\mathrm{dd}^{L}\tau$ we can rewrite this as:
\begin{align*}
\delta \Xi &: \sum_{k = 0}^{N-1} \sum_{i = 1}^s h b_i \left[ \vphantom{\frac{1}{h}} \left\langle D_1 \ell\left(g_k \tau(\Xi_k^i), \mathrm{d}^{L}\tau_{\Xi_k^i} \mathrm{H}_k^i\right), g_k \tau(\Xi_k^i) \mathrm{d}^L\tau_{\Xi_k^i} \delta \Xi_{k}^i \right\rangle\right.\\
&+ \left\langle D_2 \ell\left(g_k \tau(\Xi_k^i), \mathrm{d}^{L}\tau_{\Xi_k^i} \mathrm{H}_k^i\right), \mathrm{d}^{L}\tau_{\Xi_k^i} \mathrm{dd}^{L}\tau_{\Xi_k^i}( \mathrm{H}_k^i, \delta{\Xi_k^i})\right\rangle + \left.\left\langle \widetilde{\mathrm{M}}_k^i, \frac{1}{h} \delta \Xi_{k}^i\right\rangle \right]\\
&= \sum_{k = 0}^{N-1} \sum_{i = 1}^s h b_i \left[ \vphantom{\frac{1}{h}} \left\langle \left(\mathrm{d}^L\tau_{\Xi_k^i}\right)^* L_{g_k \tau(\Xi_k^i)}^* D_1 \ell\left(g_k \tau(\Xi_k^i), \mathrm{d}^{L}\tau_{\Xi_k^i} \mathrm{H}_k^i\right), \delta \Xi_{k}^i \right\rangle\right.\\
&+ \left.\left\langle \left( \mathrm{dd}^{L}\tau_{\Xi_k^i}\right)^* \left( \mathrm{H}_k^i, \left(\mathrm{d}^{L}\tau_{\Xi_k^i}\right)^* D_2 \ell\left(g_k \tau(\Xi_k^i), \mathrm{d}^{L}\tau_{\Xi_k^i} \mathrm{H}_k^i\right)\right) + \frac{1}{h} \widetilde{\mathrm{M}}_k^i, \delta \Xi_{k}^i\right\rangle \right].
\end{align*}

It is now the turn of variations with respect to $\delta \mathrm{H}$:
\begin{align*}
\delta \mathrm{H} &: \sum_{k = 0}^{N-1} \sum_{i = 1}^s h b_i \left[ \vphantom{\sum_{i = 1}^s} \left\langle D_2 \ell\left(g_k \tau(\Xi_k^i), \mathrm{d}^{L}\tau_{\Xi_k^i} \mathrm{H}_k^i\right), \mathrm{d}^{L}\tau_{\Xi_k^i} \delta \mathrm{H}_k^i \right\rangle \right.\\
&- \left.\left\langle \widetilde{\mathrm{M}}_k^i, \sum_{j=1}^s a_{i j} \delta \mathrm{H}_k^j\right\rangle - \left\langle \tilde{\mu}_{k+1}, \sum_{j=1}^s b_{j} \delta \mathrm{H}_k^j\right\rangle \right]\\
&= \sum_{k = 0}^{N-1} \sum_{i = 1}^s h \left[ \vphantom{\sum_{i = 1}^s} \left\langle b_i \left(\mathrm{d}^{L}\tau_{\Xi_k^i}\right)^* D_2 \ell\left(g_k \tau(\Xi_k^i), \mathrm{d}^{L}\tau_{\Xi_k^i} \mathrm{H}_k^i\right), \delta \mathrm{H}_k^i \right\rangle \right.\\
&- \left.\left\langle \sum_{j=1}^s b_{j} a_{j i} \widetilde{\mathrm{M}}_k^j, \delta \mathrm{H}_k^i\right\rangle - \left\langle b_i \tilde{\mu}_{k+1}, \delta \mathrm{H}_k^i\right\rangle \right].
\end{align*}
where we have used again the \ref{eq:order_1_cond} condition and we have also rearranged summation indices.

The last two variations are the easiest ones, as they are nothing but the RKMK constraints:
\begin{align*}
\delta \widetilde{\mathrm{M}}, \delta \tilde{\mu} &: \sum_{k = 0}^{N-1} \sum_{i = 1}^s h b_i \left[ \left\langle \delta \widetilde{\mathrm{M}}_k^i, \frac{1}{h} \Xi_{k}^i - \sum_{j=1}^s a_{i j} \mathrm{H}_k^j\right\rangle \right.\\
&+ \left.\left\langle \delta \tilde{\mu}_{k+1}, \frac{1}{h} \tau^{-1}((g_k)^{-1} g_{k+1}) - \sum_{j=1}^s b_{j} \mathrm{H}_k^j\right\rangle \right].
\end{align*}

After imposing fixed-end variations we are left with the following set of equations for $k = 1, ..., N-1$ and $i = 1, ..., s$:
\begin{equation}
\mu_{k+1} = \mathrm{Ad}^*_{\tau(\xi_{k,k+1})} \left[\mu_{k} + h \sum_{i = 1}^s b_i \left(\mathrm{Ad}^{-1}_{\tau(\Xi_k^i)}\right)^* L_{g_k \tau(\Xi_k^i)}^* D_1 \ell\left(g_k \tau(\Xi_k^i), \mathrm{d}^{L}\tau_{\Xi_k^i} \mathrm{H}_k^i\right)\right],\label{eq:delta_g_var}
\end{equation}
\begin{align}
\widetilde{\mathrm{M}}_k^i &= - h \left[ \left(\mathrm{d}^L\tau_{\Xi_k^i}\right)^* L_{g_k \tau(\Xi_k^i)}^* D_1 \ell\left(g_k \tau(\Xi_k^i), \mathrm{d}^{L}\tau_{\Xi_k^i} \mathrm{H}_k^i\right) \right.
\label{eq:delta_xi_var}\\
&+ \left.\left( \mathrm{dd}^{L}\tau_{\Xi_k^i} \right)^* \left( \mathrm{H}_k^i, \left(\mathrm{d}^{L}\tau_{\Xi_k^i}\right)^* D_2 \ell\left(g_k \tau(\Xi_k^i), \mathrm{d}^{L}\tau_{\Xi_k^i} \mathrm{H}_k^i \right)\right)\right],
\nonumber
\end{align}
\begin{align}
\left(\mathrm{d}^{L}\tau_{\Xi_k^i}\right)^* D_2 \ell\left(g_k \tau(\Xi_k^i), \mathrm{d}^{L}\tau_{\Xi_k^i} \mathrm{H}_k^i\right) - \sum_{j=1}^s \frac{b_j a_{j i}}{b_i} \widetilde{\mathrm{M}}_k^j - \left(\mathrm{d}^{L}\tau_{\xi_{k,k+1}}\right)^* \mu_{k+1} = 0,
\label{eq:delta_eta_var}
\end{align}
\begin{equation}
\Xi_k^i = h \sum_{j = 1}^s a_{i j} \mathrm{H}_k^j,
\label{eq:delta_mu_1_var}
\end{equation}
\begin{equation}
\xi_{k,k+1} = h \sum_{j = 1}^s b_{j} \mathrm{H}_k^j.
\label{eq:delta_mu_2_var}
\end{equation}

Using some of the shorthand variables defined in the statement of the theorem we may rewrite eqs. \eqref{eq:delta_g_var} and \eqref{eq:delta_xi_var} as
\begin{align*}
\mu_{k+1} &= \mathrm{Ad}_{\tau(\xi_{k,k+1})}^* \left[ \mu_{k} + h \sum_{j = 1}^s b_j \left(\mathrm{d}^L\tau^{-1}_{-\Xi_k^j}\right)^* \mathrm{N}_k^j\right],\\
\widetilde{\mathrm{M}}_k^i &= - h \left[ \mathrm{N}_k^i + \left( \mathrm{dd}^{L}\tau_{\Xi_k^i} \right)^* \left( \mathrm{H}_k^i, \Pi_k^i \right)\right],
\end{align*}
where in the first one we have also used the fact that $\mathrm{d}^{L}\tau_{\Xi}^{-1} \mathrm{Ad}^{-1}_{\tau(\Xi)} = \mathrm{d}^{L}\tau_{-\Xi}^{-1}$. This first one is in fact one of the equations we were after.

Inserting both of these in eq.\eqref{eq:delta_eta_var}, and leaving only terms with $\Pi_k$ on the left-hand side we finally obtain:
\begin{equation*}
\mathrm{M}_k^i = \mathrm{Ad}_{\tau(\xi_{k,k+1})}^* \left[ \mu_{k} + h \sum_{j = 1}^s b_j \left( \mathrm{d}^L\tau^{-1}_{-\Xi_k^j} - \frac{a_{j i}}{b_i} \mathrm{d}^{L}\tau^{-1}_{-\xi_{k,k+1}}\right)^* \mathrm{N}_k^j\right]
\end{equation*}
which is the remaining equation we wanted to obtain.
\end{proof}

\begin{remark}
It is worth noting that perhaps eq.\eqref{eq:inner_mu_equations} is not the most geometric way to express such a relation. That form has been chosen for notational economy and mnemotechnic reasons.

In order to give a more geometrically sound version of this equation we should identify the different elements that appear in it. First, let us consider a point $(G_k^i, V_k^i) \in TG$ and the section of $T^*TG$ induced by $\mathrm{d}L$ on that point, i.e.,
\begin{equation*}
\left(G_k^i, V_k^i, D_1 L(G_k^i, V_k^i), D_2 L(G_k^i, V_k^i) \right).
\end{equation*}

One may rush to the conclusion that if we represent such an element in $T^*T\mathfrak{g}$ we should get:
\begin{align*}
&\left(\Xi_k^i, \mathrm{H}_k^i, \left(\mathrm{d}^L \tau_{\Xi_k^i}\right)^* L_{L_{g_k} \tau\left(\Xi_k^i\right)}^* D_1 L(L_{g_k} \tau\left(\Xi_k^i\right),\left(L_{g_k}\right)_* \mathrm{d}^L \tau_{\Xi_k^i} \mathrm{H}_k^i), \right.\\
&\hspace{3cm}\left. \left(\mathrm{d}^L \tau_{\Xi_k^i}\right)^* L_{L_{g_k} \tau\left(\Xi_k^i\right)}^* D_2 L(L_{g_k} \tau\left(\Xi_k^i\right),\left(L_{g_k}\right)_* \mathrm{d}^L \tau_{\Xi_k^i} \mathrm{H}_k^i) \right),
\end{align*}
which, using the invariance of the Lagrangian, reduces to
\begin{align*}
&\left(\Xi_k^i, \mathrm{H}_k^i, \left(\mathrm{d}^L \tau_{\Xi_k^i}\right)^* L_{L_{g_k} \tau\left(\Xi_k^i\right)}^* D_1 \ell(L_{g_k} \tau\left(\Xi_k^i\right), \mathrm{d}^L \tau_{\Xi_k^i} \mathrm{H}_k^i), \right.\\
&\hspace{3cm}\left. \left(\mathrm{d}^L \tau_{\Xi_k^i}\right)^* D_2 \ell(L_{g_k} \tau\left(\Xi_k^i\right), \mathrm{d}^L \tau_{\Xi_k^i} \mathrm{H}_k^i) \right),
\end{align*}
but this is not correct in this instance. The reason for this is that this expression is not compatible with the restriction of the natural inner product $\left\langle\cdot, \cdot\right\rangle: TTG \times T^*TG \to \mathbb{R}$ to $T^{(2)}G$. This compatibility is required to obtain the correct invariance when considering Tulczyjew's triple, which allows us to interpret the third component as the velocity associated to the canonical momenta. In particular, $\alpha_G^{-1} : T^*TG \to TT^*G$, $(g,v,P_q,P_v) \mapsto (g,p = P_v,V_g = v,V_p = P_q)$.

Thus the correct representation in $T^*T\mathfrak{g}$ must be:
\begin{align*}
&\left(\Xi_k^i, \mathrm{H}_k^i, \left(\mathrm{d}^L \tau_{\Xi_k^i}\right)^* L_{L_{g_k} \tau\left(\Xi_k^i\right)}^* D_1 \ell(L_{g_k} \tau\left(\Xi_k^i\right), \mathrm{d}^L \tau_{\Xi_k^i} \mathrm{H}_k^i) \right.\\
&\hspace{1.5cm}+ \left(\mathrm{dd}^L \tau_{\Xi_k^i}\right)^*\left(\Xi_k^i, \left(\mathrm{d}^L \tau_{\Xi_k^i}\right)^* D_2 \ell(L_{g_k} \tau\left(\Xi_k^i\right),\mathrm{d}^L \tau_{\Xi_k^i} \mathrm{H}_k^i)\right),\\
&\hspace{3cm}\left. \left(\mathrm{d}^L \tau_{\Xi_k^i}\right)^* D_2 \ell(L_{g_k} \tau\left(\Xi_k^i\right), \mathrm{d}^L \tau_{\Xi_k^i} \mathrm{H}_k^i) \right).
\end{align*}

Clearly $\widetilde{\mathrm{M}}_k^i$ is proportional to this third component. If we transport this element from $\Xi_k^i$ to $0 \in \mathfrak{g}$ we are left with:
\begin{align*}
&\left(0,  \mathrm{d}^L \tau_{\Xi_k^i} \mathrm{H}_k^i, L_{L_{g_k} \tau\left(\Xi_k^i\right)}^* D_1 \ell(L_{g_k} \tau\left(\Xi_k^i\right), \mathrm{d}^L \tau_{\Xi_k^i} \mathrm{H}_k^i), D_2 \ell(L_{g_k} \tau\left(\Xi_k^i\right), \mathrm{d}^L \tau_{\Xi_k^i} \mathrm{H}_k^i) \right).
\end{align*}
which shows that this third component becomes $\left(\mathrm{d}^L \tau_{\Xi_k^i}^{-1}\right)^* \mathrm{N}_k^i$.

In order to simplify the final expression, let us write these as $\left(\xi_{k}^i,{\eta_{\xi}}_{k}^i, {\nu_{\xi}}_{k}^i, {\mu_{\xi}}_{k}^i\right)$ and $\left(0, {\eta_{0}}_{k}^i, {\nu_{0}}_{k}^i, {\mu_{0}}_{k}^i\right)$ respectively. Taking this into account we can finally rewrite eq.\eqref{eq:inner_mu_equations} as
\begin{align*}
&\left(\mathrm{d}^L \tau_{\xi_{k,k+1}}^{-1}\right)^* {\mu_{\xi}}_k^i =\\
&\mathrm{Ad}_{\tau(\xi_{k,k+1})}^* \left\lbrace\mu_k + h \sum_{j = 1}^s b_j \left[ \left(\mathrm{Ad}_{\tau(\xi_{k}^j)}^{-1}\right)^* {\nu_0}_{k}^j - \frac{a_{j i}}{b_i} \left(\mathrm{Ad}_{\tau(\xi_{k,k+1})}^{-1}\right)^* \left(\mathrm{d}^L \tau_{\xi_{k,k+1}}^{-1}\right)^* {\nu_{\xi}}_{k}^j \right] \right\rbrace,
\end{align*}
and if we use the notation $\left(\mathrm{d}^L \tau_{\xi_{k,k+1}}^{-1}\right)^* {\zeta_{\xi}}_k^i = \left.\hat{\zeta}_{\xi}\right._k^i$, this reduces to
\begin{align*}
&\left.\hat{\mu}_{\xi}\right._k^i = \mathrm{Ad}_{\tau(\xi_{k,k+1})}^* \left\lbrace\mu_k + h \sum_{j = 1}^s b_j \left[ \left(\mathrm{Ad}_{\tau(\xi_{k}^j)}^{-1}\right)^* {\nu_0}_{k}^j - \frac{a_{j i}}{b_i} \left(\mathrm{Ad}_{\tau(\xi_{k,k+1})}^{-1}\right)^* \left.\hat{\nu}_{\xi}\right._{k}^j \right] \right\rbrace.
\end{align*}
\end{remark}

\subsection{Holonomic constraints}
In order to consider nonholonomic constraints we will first check the holonomic case. Assume now that our system is subjected to a set of holonomic constraints locally spanned by a function $\Phi: G \to \mathbb{R}^m$. The inclusion of these constraints amounts to the addition of a new set of terms to the discrete Hamilton-Pontryagin action, eq.\eqref{eq:Lie_discrete_Hamilton-Pontriagyn_action}:
\begin{equation*}
\left(\mathcal{\widetilde{J}_{HP}}\right)_d = \left(\mathcal{J_{HP}}\right)_d + \sum_{k = 0}^{N-1} \sum_{i = 1}^s h b_i \left\langle \Lambda_k^i, \Phi(g_k \tau(\Xi_k^i))\right\rangle
\end{equation*}

Once more, we restrict to methods satisfying hypotheses \ref{itm:H1}, \ref{itm:H2} and \ref{itm:H3}. Variation of these new terms sheds the following:
\begin{align*}
\delta g : \quad &\sum_{k = 0}^{N-1} \sum_{i = 1}^s h b_i \left\langle \Lambda_k^i, \left\langle D\Phi(g_k \tau(\Xi_k^i)), \delta g_k \tau(\Xi_k^i)\right\rangle \right\rangle\\
&= \sum_{k = 0}^{N-1} \sum_{i = 1}^s h b_i\left\langle \left\langle \Lambda_k^i, \left(\mathrm{d}^L \tau_{-\Xi_k^i}^{-1}\right)^* \left(\mathrm{d}^L \tau_{\Xi_k^i}\right)^* L_{g_k \tau(\Xi_k^i)}^* D\Phi(g_k \tau(\Xi_k^i))\right\rangle, \zeta_k\right\rangle,\\
\delta \Xi : \quad &\sum_{k = 0}^{N-1} \sum_{i = 1}^s h b_i \left\langle \Lambda_k^i, \left\langle D\Phi(g_k \tau(\Xi_k^i)), g_k D\tau(\Xi_k^i) \delta \Xi_k^i\right\rangle\right\rangle\\
&= \sum_{k = 0}^{N-1} \sum_{i = 1}^s h b_i \left\langle \left\langle \Lambda_k^i, \left(\mathrm{d}^L \tau_{\Xi_k^i}\right)^* L_{g_k \tau(\Xi_k^i)}^* D\Phi(g_k \tau(\Xi_k^i))\right\rangle, \delta \Xi_k^i\right\rangle,\\
\delta \Lambda : \quad &\sum_{k = 0}^{N-1} \sum_{i = 1}^s h b_i \left\langle \delta \Lambda_k^i, \Phi(g_k \tau(\Xi_k^i))\right\rangle.
\end{align*}

The first two manifest in a modification of eqs.\eqref{eq:inner_mu_equations} and \eqref{eq:outer_mu_equations} with $\mathrm{N}_k^i \mapsto \mathrm{N}_k^i + \mathrm{T}_k^i$, where:
\begin{equation*}
\mathrm{T}_k^i = \left\langle \Lambda_k^i, \left(\mathrm{d}^L \tau_{\Xi_k^i}\right)^* L_{g_k \tau(\Xi_k^i)}^* D\Phi(g_k \tau(\Xi_k^i))\right\rangle
\end{equation*}

Of course, the variations in $\Lambda$ are nothing more than the constraint equations themselves, which must be added to the rest of the equations.

As in the vector space case, we will still need to add the tangency condition to these equations to generate a well-defined Hamiltonian map $\tilde{\mathcal{F}}_{L_d}: (g_{k}, \mu_{k}) \mapsto (g_{k+1}, \mu_{k+1})$. This final equation must read:
\begin{equation*}
\left\langle L_{g_{k+1}}^*D \Phi(g_{k+1}), D_2 \mathscr{h}\left(g_{k+1}, \mu_{k+1}\right)\right\rangle = 0
\end{equation*}
where $\mathscr{h}: G \times \mathfrak{g}^* \to \mathbb{R}$ is the corresponding reduced Hamiltonian function.

\subsection{Nonholonomic constraints}
This time, assume that our system is subjected to a set of nonholonomic constraints locally spanned by a function $\Phi : TG \to \mathbb{R}^m$ and that $\Phi \circ \mathbb{F}L^{-1} = \Psi : T^*G \to \mathbb{R}^m$.

%The continuous equations of motion for a nonholonomic system on a Lie group are:
%\begin{equation*}
%\left\lbrace
%\begin{array}{rl}
%\frac{\mathrm{d}}{\mathrm{d}t} \left(\frac{\partial \ell}{\partial \eta}\right) - \mathrm{ad}_{\eta(t)}^* \left(\frac{\partial \ell}{\partial \eta}\right) & = L_{g(t)}^{*}\left[ \frac{\partial \ell}{\partial g} + \left\langle \lambda, \frac{\partial \Phi}{\partial \dot{g}}\right\rangle\right],\\
%\Phi(g,\dot{g}) &= 0,\\
%\dot{g} &= \left(L_{g}\right)_{*} \eta.
%\end{array}
%\right.
%\end{equation*}

Applying the same reasoning as in the vector space case it is clear that we need to apply the substitution $\mathrm{N}_k^i \mapsto \mathrm{N}_k^i + \left(\mathrm{T}_\mathrm{nh}\right)_k^i$, where:
\begin{equation*}
\left(\mathrm{T}_\mathrm{nh}\right)_k^i = \left\langle \Lambda_k^i, \left(\mathrm{d}^L \tau_{\Xi_k^i}\right)^* L_{g_k \tau(\Xi_k^i)}^* D_2\Phi\left(g_k \tau(\Xi_k^i), g_k \tau(\Xi_k^i) \mathrm{d}^L \tau_{\Xi_k^i} \mathrm{H}_k^i \right) \right\rangle
\end{equation*}

Aside from that, we need to introduce the equations
\begin{equation*}
\mu_{k}^i = \mathrm{Ad}_{\tau(\Xi_{k}^i)}^* \left[ \mu_{k} + h \sum_{j = 1}^s a_{i j} \left(\mathrm{d}^L\tau^{-1}_{-\Xi_k^j}\right)^* \mathrm{N}_k^j\right],
\end{equation*}
together with the constraint equations
\begin{equation*}
\Psi\left(g_k \tau(\Xi_k^i), L_{\left(g_k \tau(\Xi_k^i)\right)^{-1}}^* \mu_{k}^i\right) = 0.
\end{equation*}

If the constraint functions can be (left) trivialized so that we can write $\phi : G \times \mathfrak{g} \to \mathbb{R}^m$ and $\psi : G \times \mathfrak{g}^* \to \mathbb{R}^m$, then
\begin{equation*}
\left(\mathrm{T}_\mathrm{nh}\right)_k^i = \left\langle \Lambda_k^i, \left(\mathrm{d}^L\tau_{\Xi_k^i}\right)^* D_2\phi\left(g_k \tau(\Xi_k^i), \mathrm{d}^L \tau_{\Xi_k^i} \mathrm{H}_k^i \right) \right\rangle,
\end{equation*}
and the constraint equations that must be imposed become
\begin{equation*}
\psi\left(g_k \tau(\Xi_k^i), \mu_k^i\right) = 0.
\end{equation*}

For the resulting nonholonomic integrators the results of theorems \ref{thm:local_convergence} and \ref{thm:global_convergence} still hold thanks to the way in which we have handled the discretisation. Thus the order of these integrators matches the expected order one would obtain in the holonomic case.

\section{Numerical experiments}\label{sec:examples}
In this section we study several nonholonomic systems using our methods. These will allow us to compare our theoretical results with actual numerical simulations and shown some of the properties of our integrators.
\subsection{Nonholonomic particle in an harmonic potential}
In this case we have $Q = \mathbb{R}^3$ and its corresponding Lagrangian and constraint functions can be written as
\begin{gather*}
L(x, y, z, v_x, v_y, v_z) = \frac{1}{2}(v_x^2 + v_y^2 + v_z^2) - \frac{1}{2}(x^2 + y^2),\\
\Phi(x, y, z, v_x, v_y, v_z) = v_z - y v_x.
\end{gather*}

This is a classic nonholonomic system frequently used as an academic example. As it can be seen in fig. \ref{fig:Order_plots}, the numerical order obtained coincides with the expected one.

\begin{figure}[h!]
\centering
\begin{subfigure}{.5\textwidth}
  \centering
  \includegraphics[width=6.2cm, clip=true, trim=40mm 90mm 40mm 90mm]{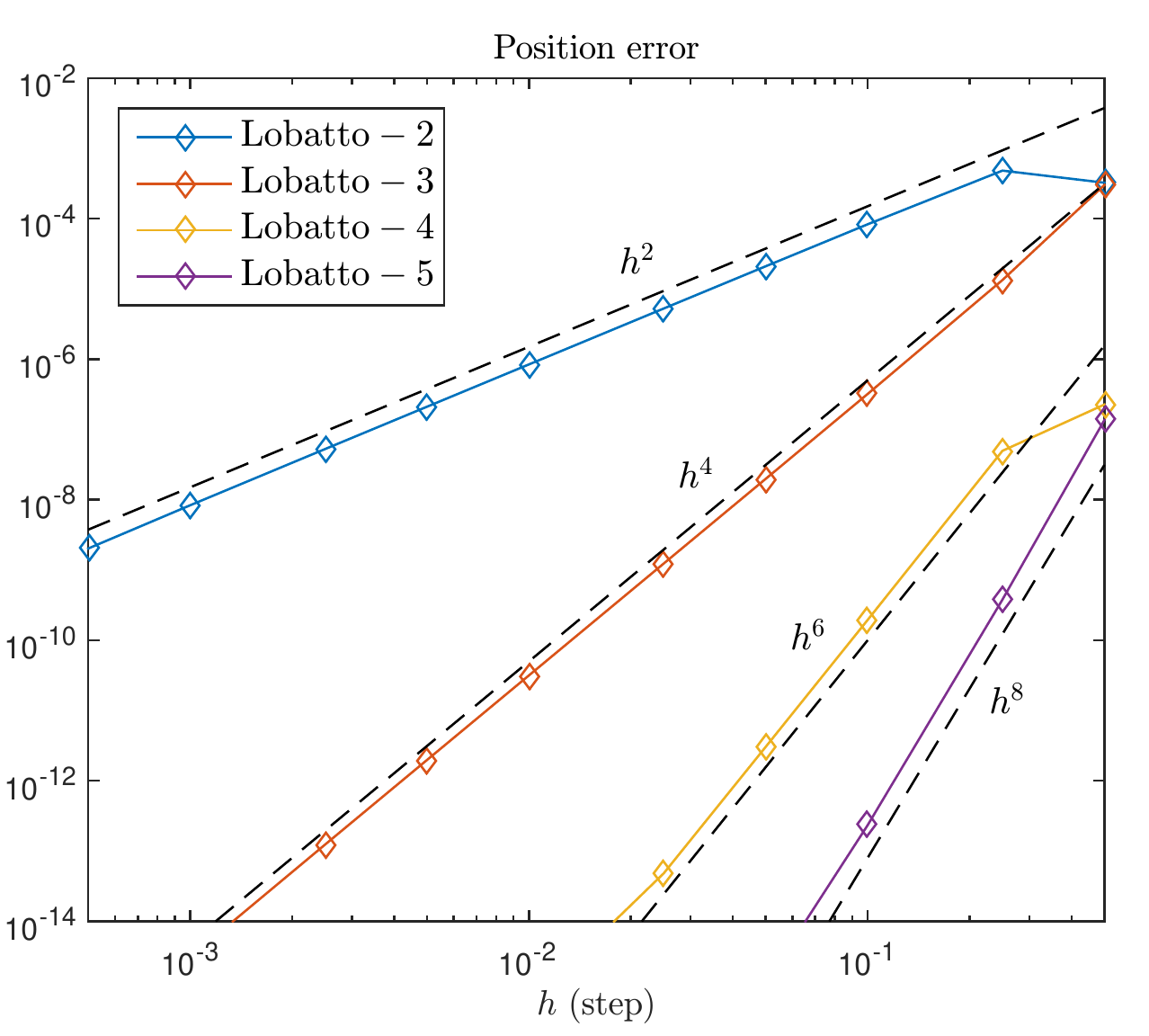}
  \includegraphics[width=6.2cm, clip=true, trim=40mm 90mm 40mm 90mm]{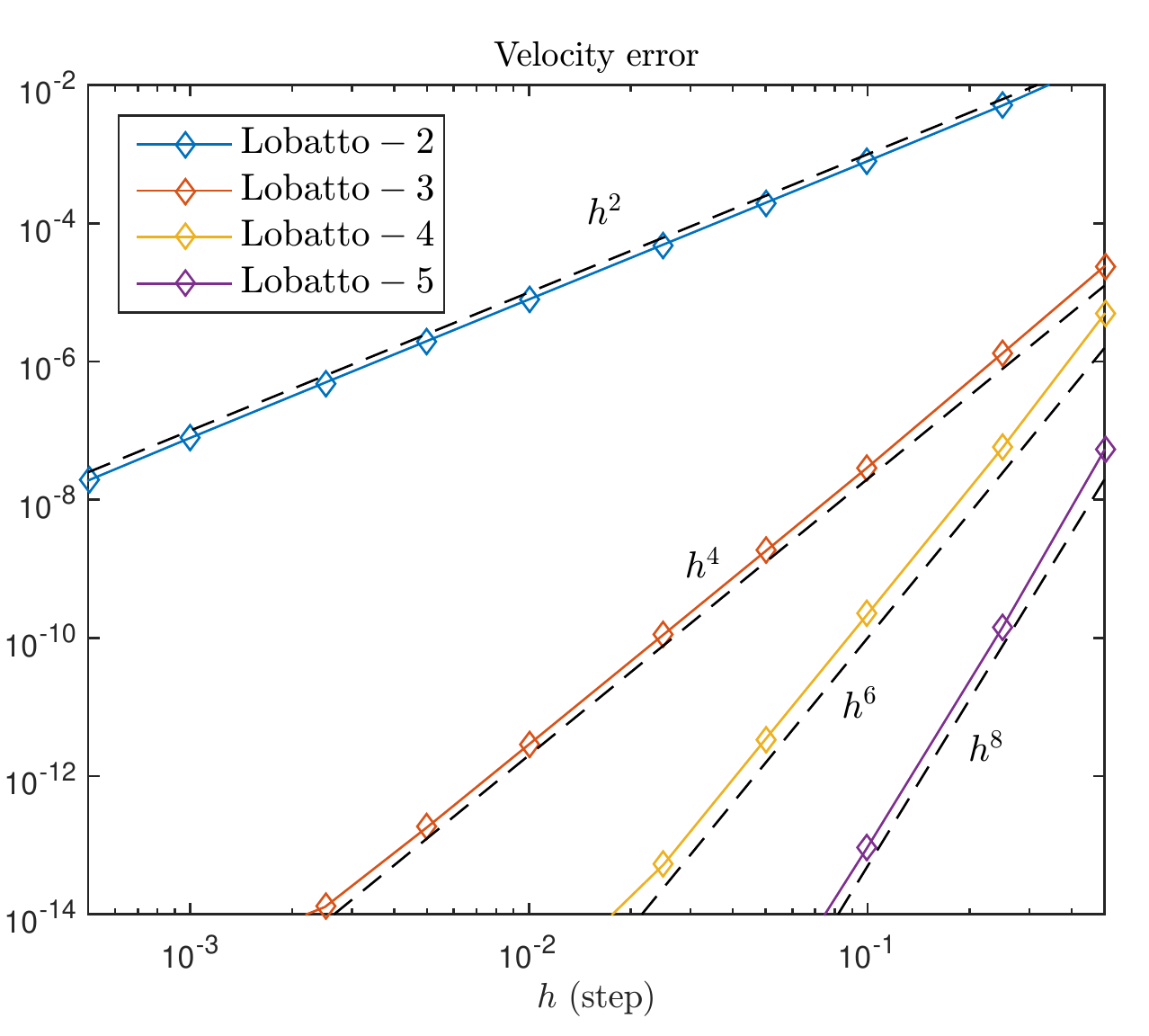}
\end{subfigure}%
\begin{subfigure}{.5\textwidth}
  \centering
  \includegraphics[width=6.2cm, clip=true, trim=40mm 90mm 40mm 90mm]{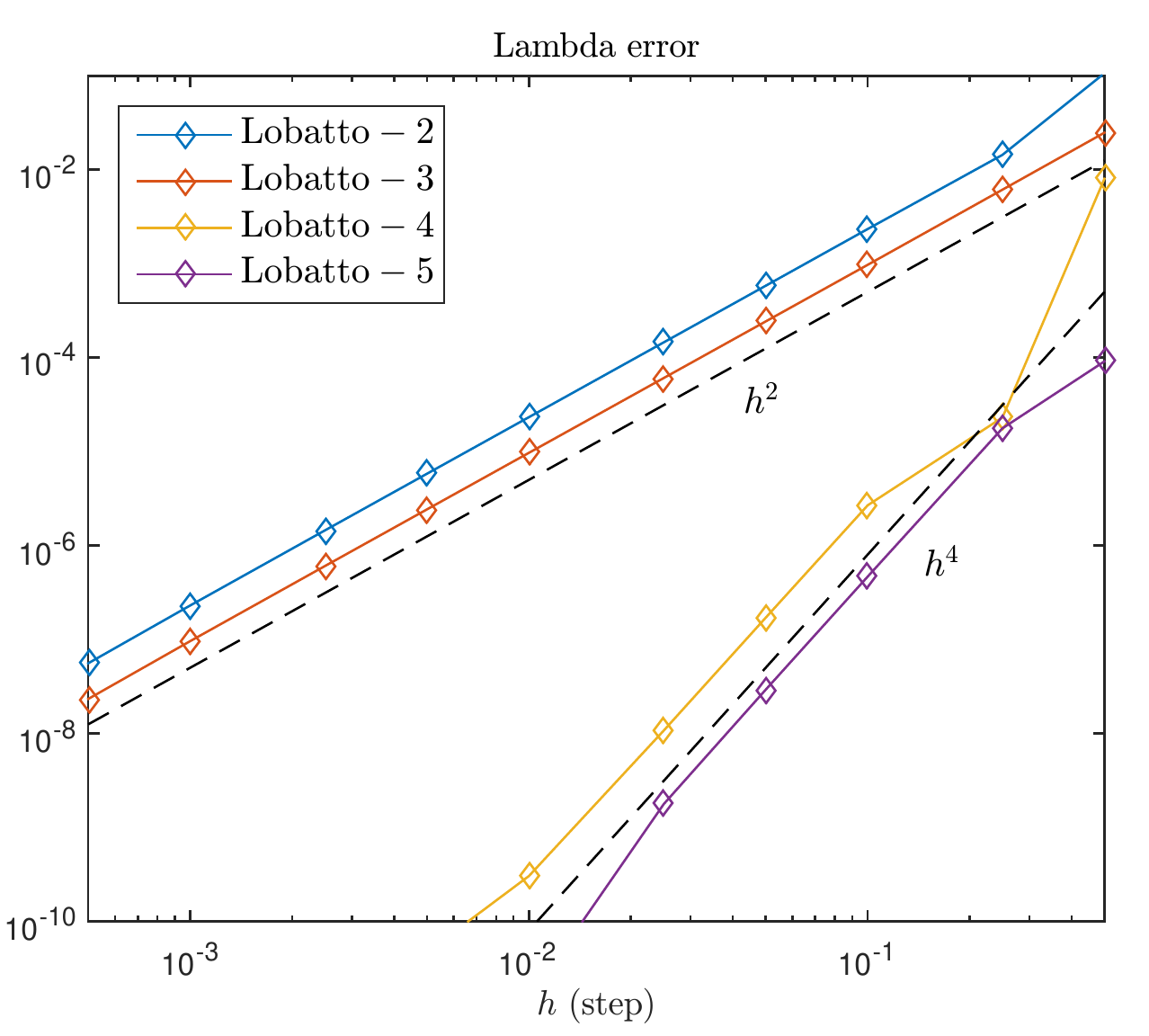}
  \includegraphics[width=6.2cm, clip=true, trim=40mm 90mm 40mm 90mm]{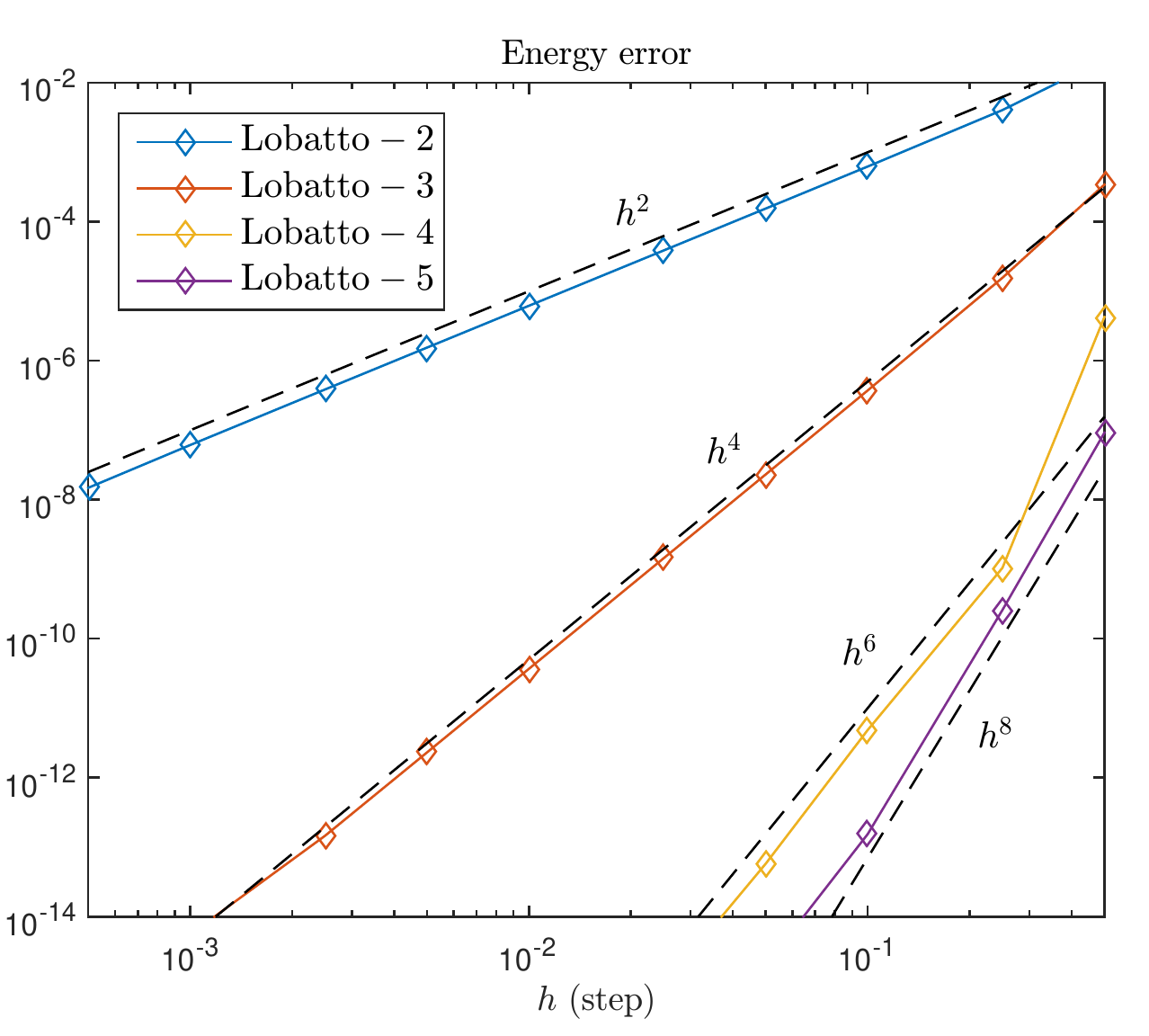}
\end{subfigure}
\caption{\scriptsize{Relative error w.r.t. reference values obtained for $h = \texttt{1e-4}$ for integrators of various orders. As can be seen, the behaviour of the Lagrange multipliers differs from the other variables, as predicted.}}
\label{fig:Order_plots}
\end{figure}

\subsection{Pendulum-driven continuous variable transmission (CVT)}
For the sake of simplicity let us consider $Q = \mathbb{R}^3$ as the configuration manifold for this system. Its corresponding Lagrangian and constraint functions are
\begin{gather*}
L(x, y, z, v_x, v_y, v_z) = \frac{1}{2}(v_x^2 + v_y^2 + v_z^2) - \frac{1}{2}\left(x^2 + z^2 - 2 \cos(y) + \epsilon \sin(2 y) \right),\\
\Phi(x, y, z, v_x, v_y, v_z) = v_z + \sin(y) v_x.
\end{gather*}
with $\epsilon \geq 0$. This system was featured in a recent preprint, \cite{modin-verdier}, where it was used as a benchmark for the behaviour of different numerical integrators. In particular those authors wanted to draw attention towards the behaviour of the energy of the $(x,z,v_x,v_z)$, \emph{passanger}, and $(y, v_y)$, \emph{driver}, subsystems when $\epsilon = 0, \frac{1}{2}$. This is done for two sets of initial conditions, one corresponding to low energy where the driver subsystem is restricted to its oscillatory regime, and one corresponding to high energy where the driver subsystems rotates.

The corresponding initial conditions are
\begin{equation*}
\vec{\mathbf{q}}_0 = \left(1, 0, 1\right), \quad \vec{\mathbf{v}}_0 = \left(0, \frac{3 \sqrt{10}}{5}, 0\right)
\end{equation*}
for the low energy case, with total energy $E_T = \frac{9}{5}$ ($E_d = \frac{4}{5}$, $E_p = 1$), and
\begin{equation*}
\vec{\mathbf{q}}_0 = \left(1, 0, 1\right), \quad \vec{\mathbf{v}}_0 = \left(0, \sqrt{8}, 0\right)
\end{equation*}
for the high energy case, with total $E_T = 4$ ($E_d = 3$, $E_p = 1$).

It is interesting to note that for the time step chosen in that paper, namely $h = \pi/10$, our integrator exhibits a rather erratic behaviour which suggests that the step might be too big. If a more sensible value, such as $h = 1/10$ is chosen, the behaviour of our integrator displays an excellent energy behaviour, as can be seen in figures \ref{fig:LowE} and \ref{fig:HighE}. This is true both for each subsystem and for the complete system. %One cannot help but wonder whether some of the problematic numerical results of \verb|\{ModinVerdier}| are in fact due to a poor choice of time step.

\begin{figure}[h!]
\centering
\begin{subfigure}{.5\textwidth}
  \centering
  \includegraphics[width=6.25cm, clip=true, trim=33mm 95mm 40mm 90mm]{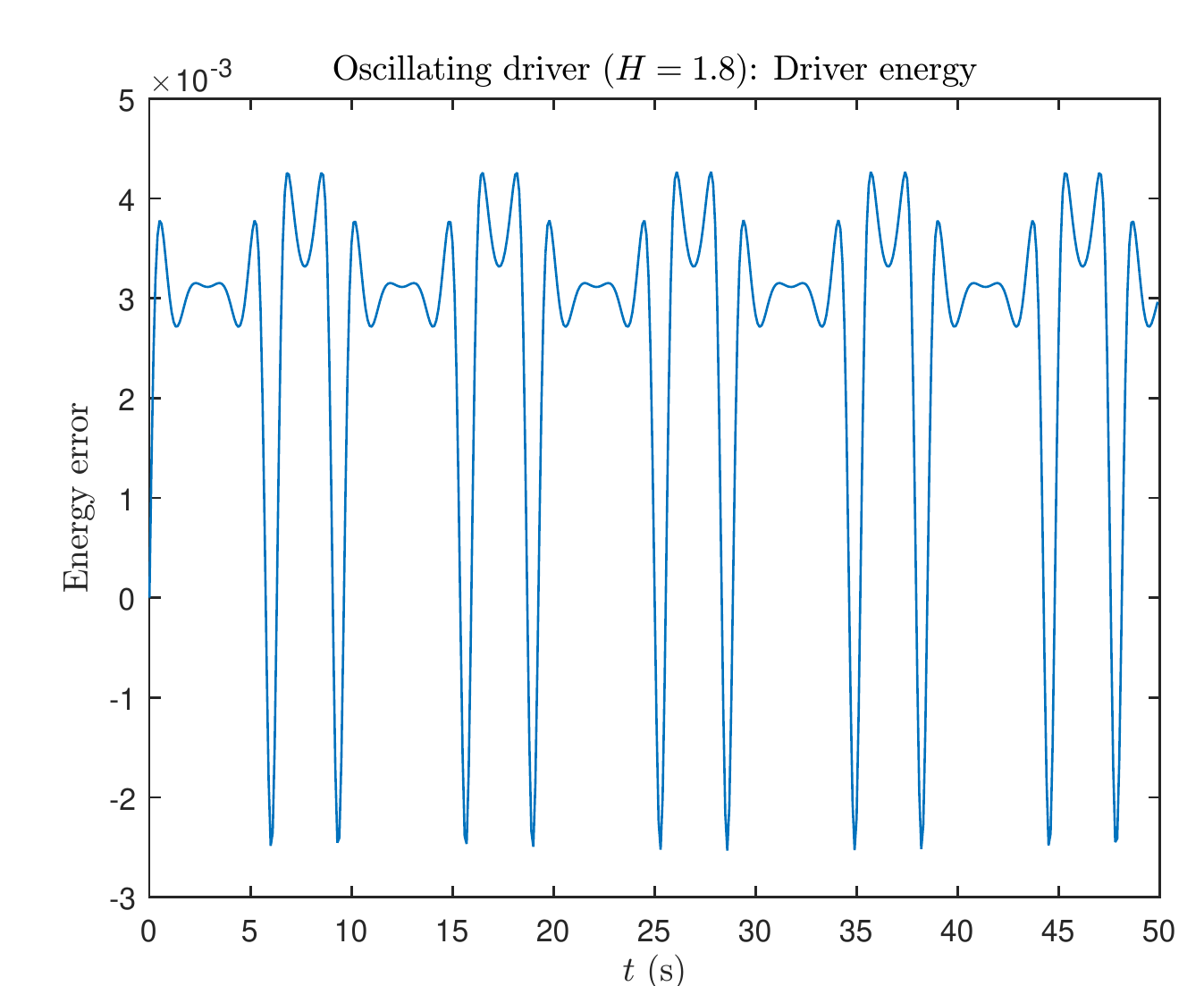}
  \caption{\scriptsize{Short term $E_d$ behaviour}}
  \label{fig:LowE_driver_short}
  \includegraphics[width=6.25cm, clip=true, trim=33mm 95mm 40mm 90mm]{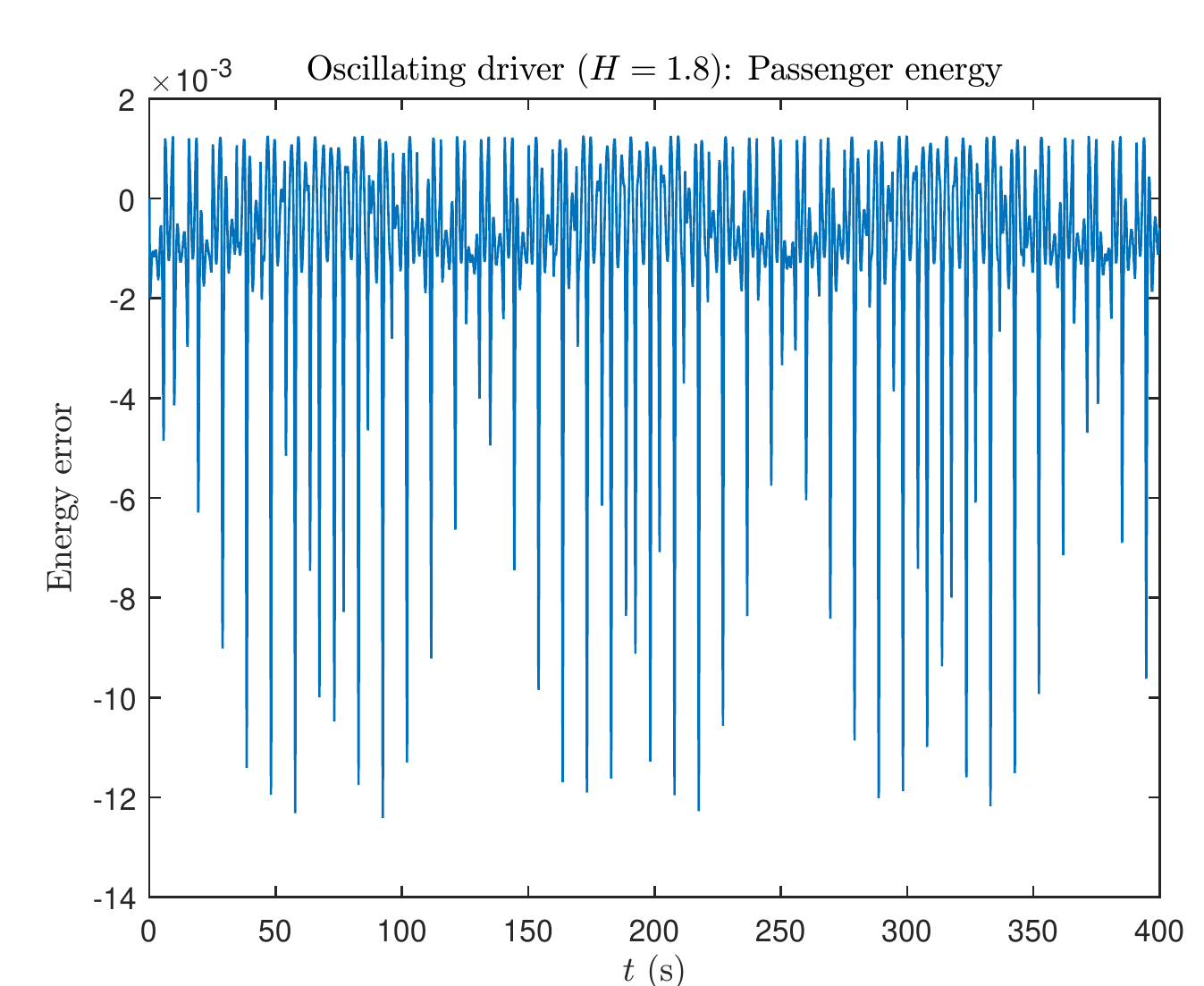}
  \caption{\scriptsize{Short term $E_p$ behaviour}}
  \label{fig:LowE_passenger_short}
\end{subfigure}%
\begin{subfigure}{.5\textwidth}
  \centering
  \includegraphics[width=6.25cm, clip=true, trim=33mm 95mm 40mm 90mm]{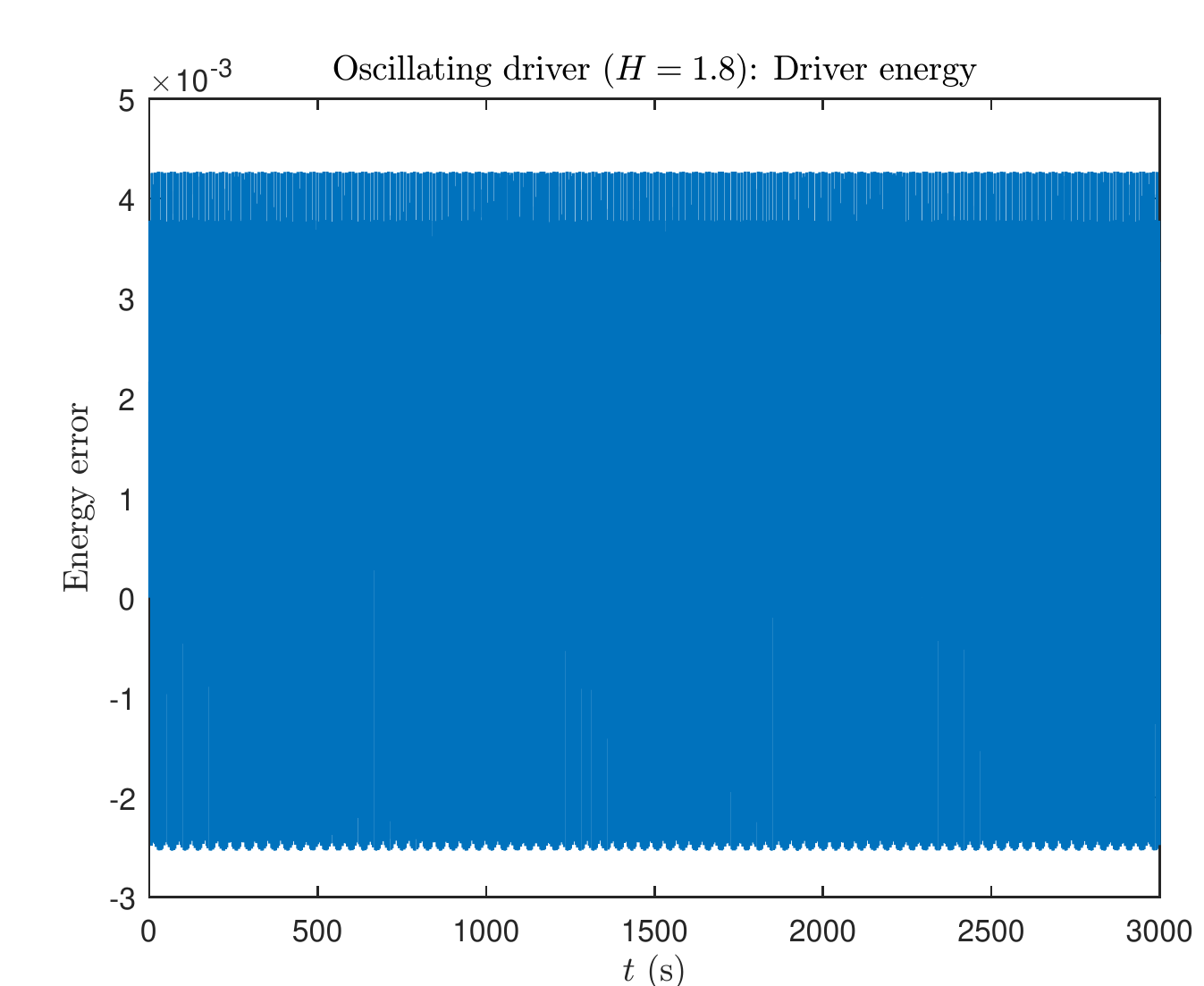}
  \caption{\scriptsize{Long term $E_d$ behaviour}}
  \label{fig:LowE_driver_long}
   \includegraphics[width=6.25cm, clip=true, trim=33mm 95mm 40mm 90mm]{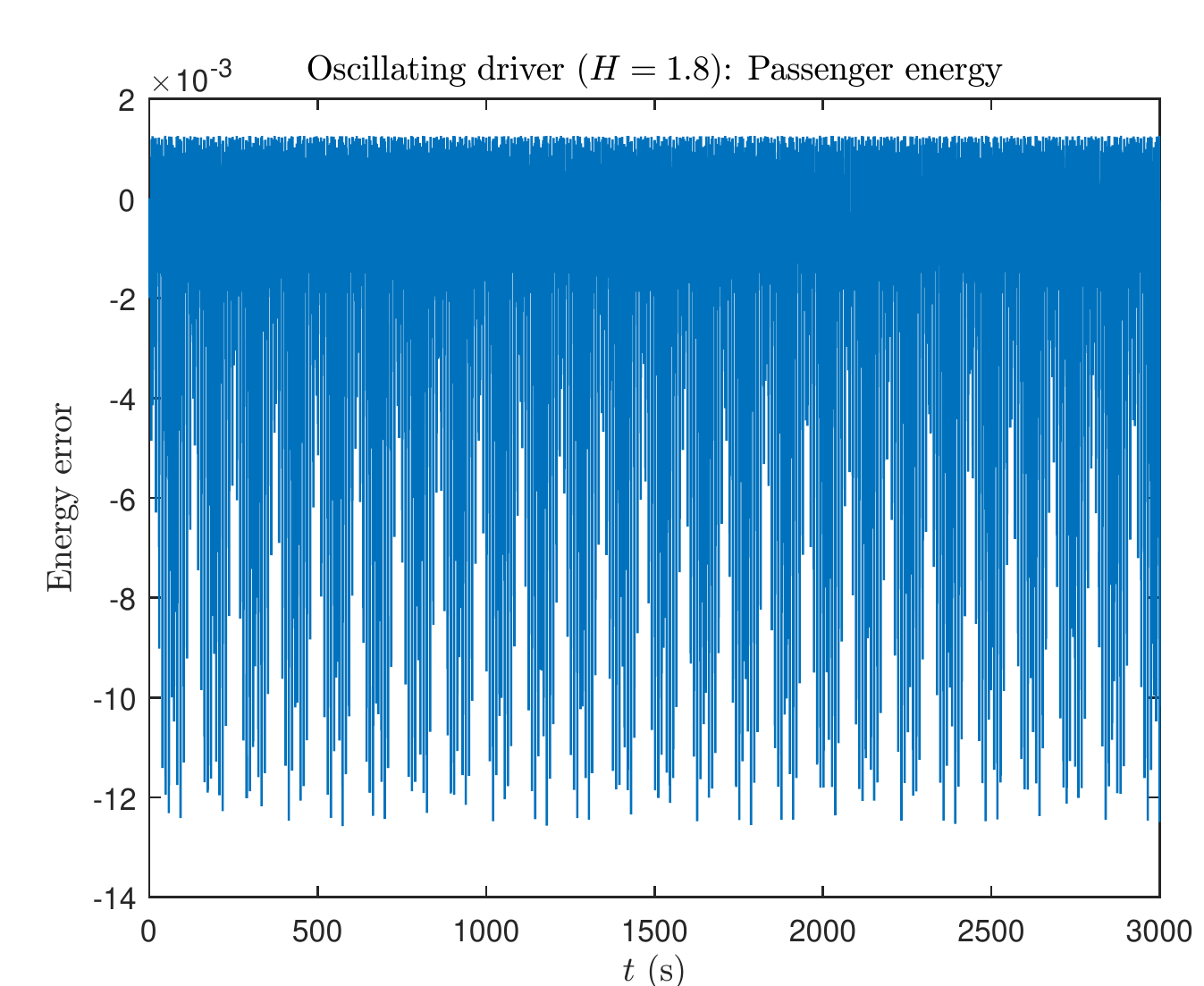}
  \caption{\scriptsize{Long term $E_p$ behaviour}}
  \label{fig:LowE_passenger_long}
\end{subfigure}
\caption{\scriptsize{Energy behaviour of the different subsystems for the oscillating regime ($E_T = 9/5$) with $\epsilon = 1/2$ for the Lobatto-2 method.}}
\label{fig:LowE}
\end{figure}

\begin{figure}[h!]
\centering
\begin{subfigure}{.5\textwidth}
  \centering
  \includegraphics[width=6.25cm, clip=true, trim=33mm 95mm 40mm 90mm]{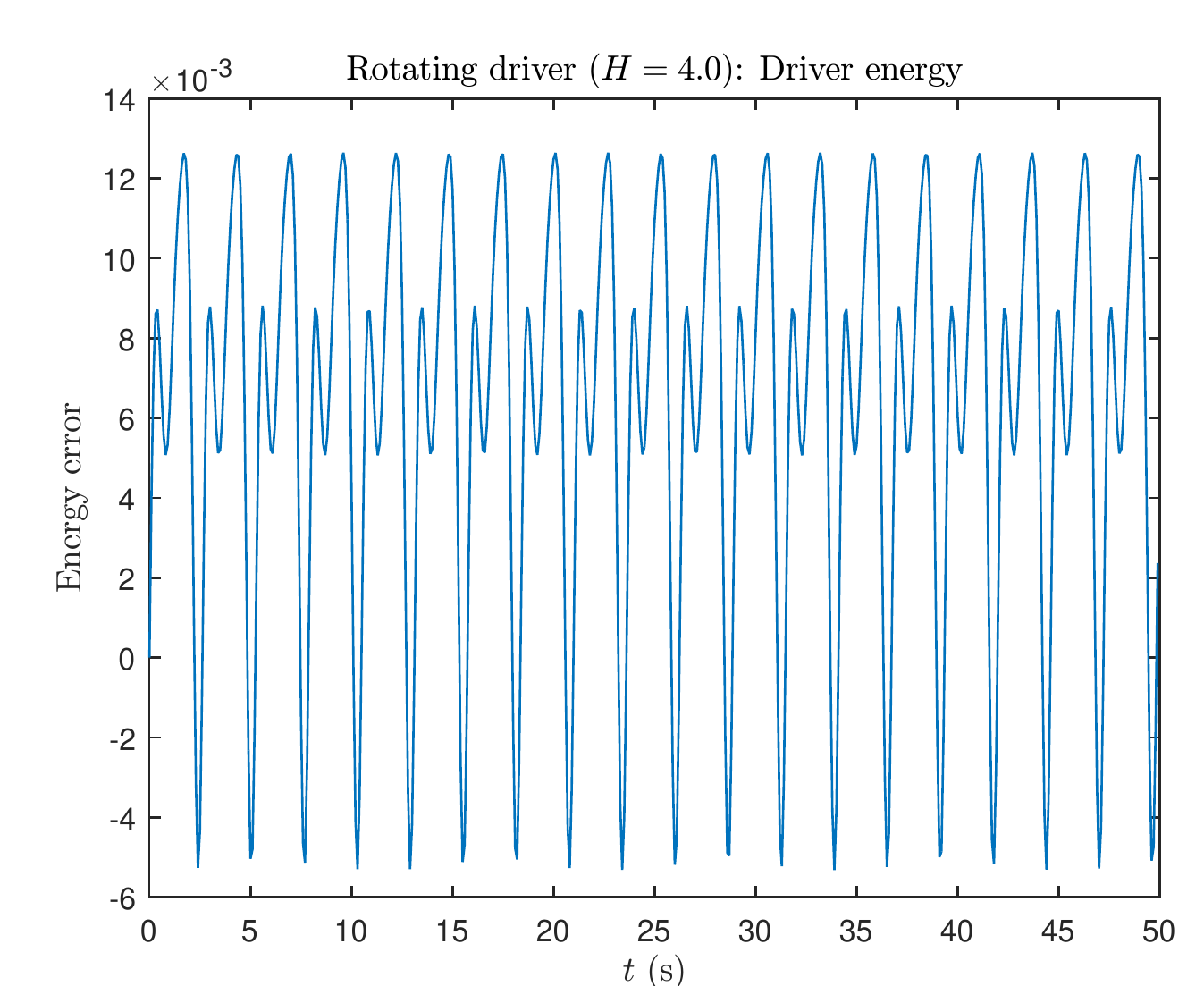}
  \caption{\scriptsize{Short term $E_d$ behaviour}}
  \label{fig:HighE_driver_short}
  \includegraphics[width=6.25cm, clip=true, trim=33mm 95mm 40mm 90mm]{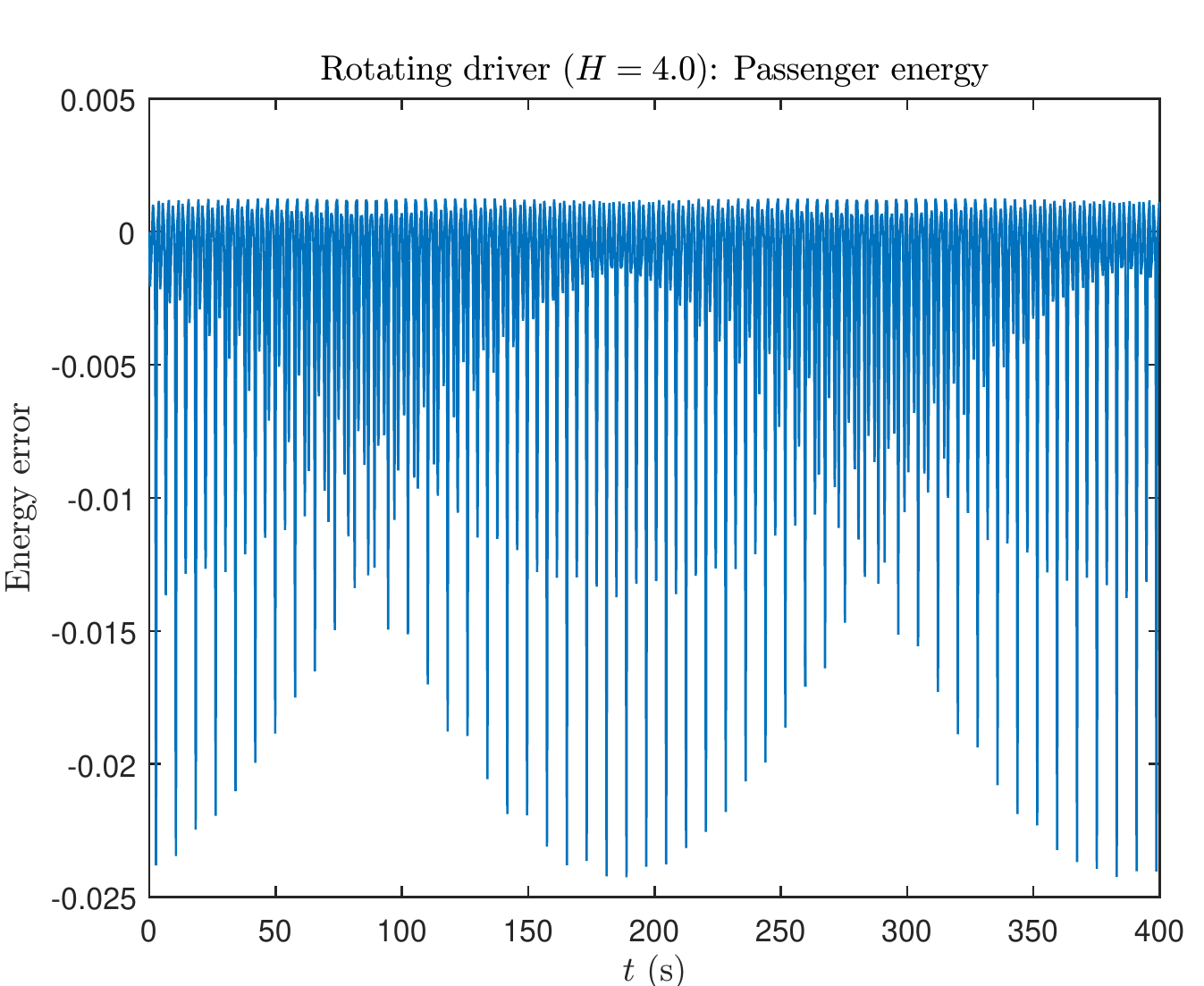}
  \caption{\scriptsize{Short term $E_p$ behaviour}}
  \label{fig:HighE_passenger_short}
\end{subfigure}%
\begin{subfigure}{.5\textwidth}
  \centering
  \includegraphics[width=6.25cm, clip=true, trim=33mm 95mm 40mm 90mm]{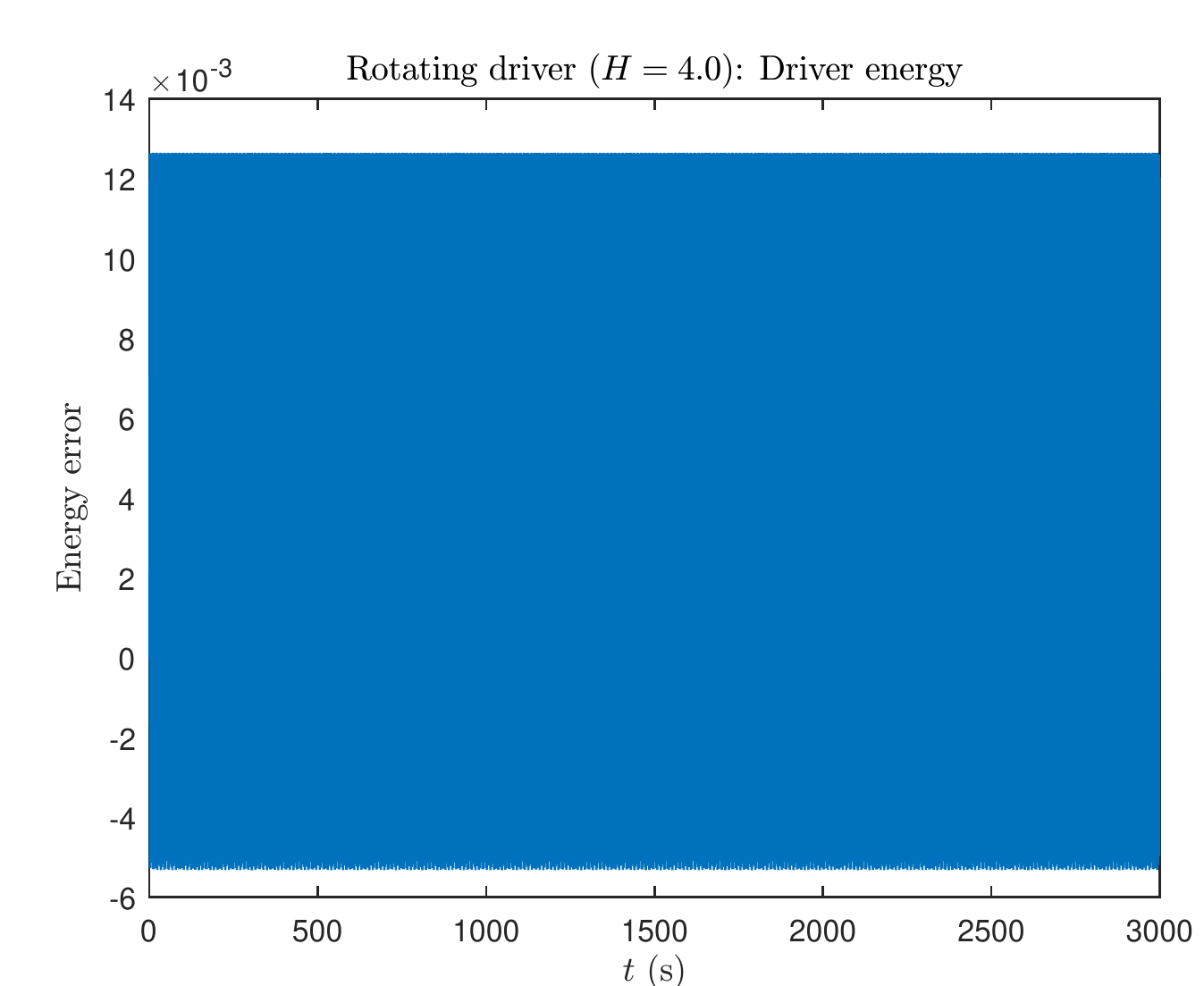}
  \caption{\scriptsize{Long term $E_d$ behaviour}}
  \label{fig:HighE_driver_long}
   \includegraphics[width=6.25cm, clip=true, trim=33mm 95mm 40mm 90mm]{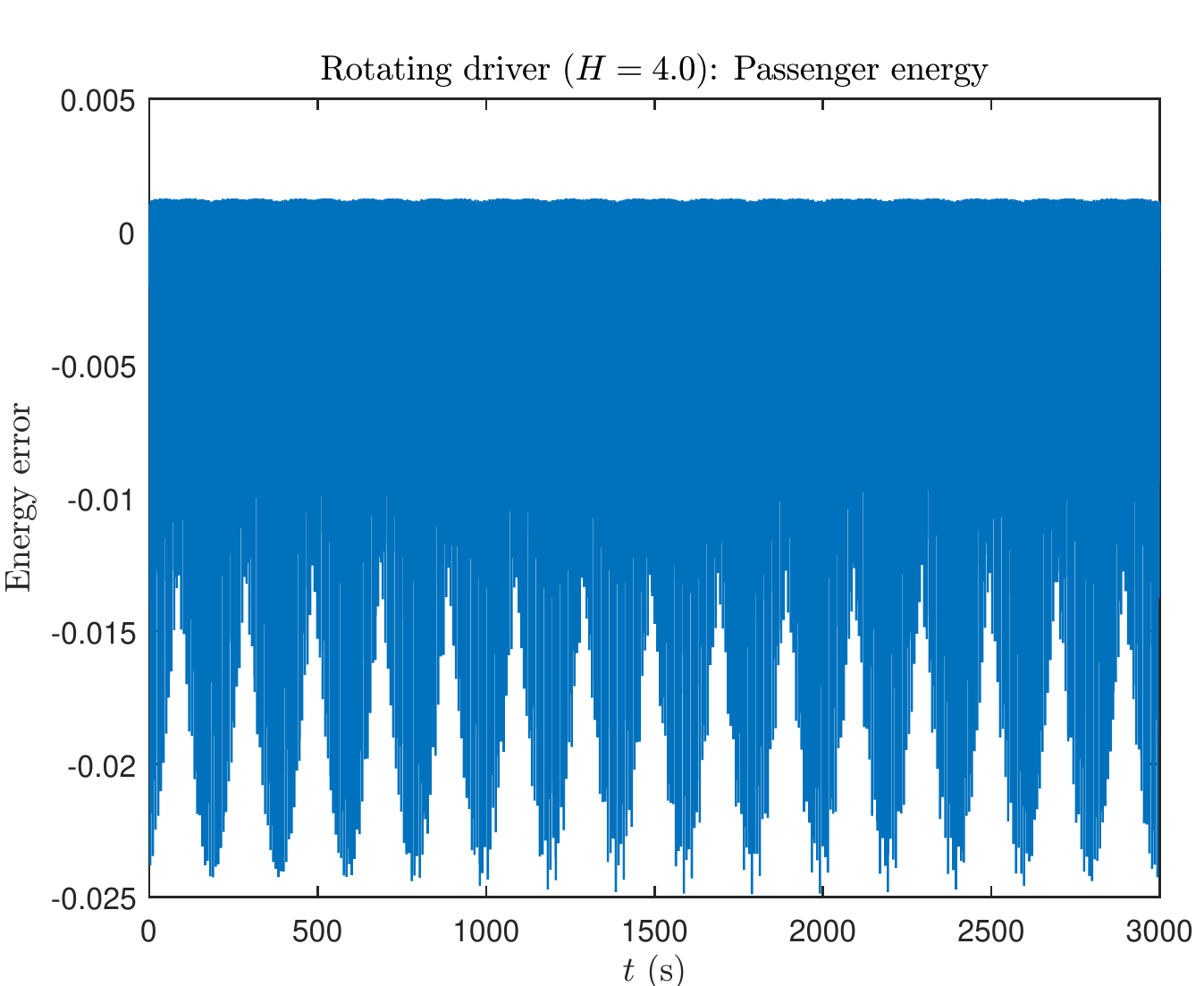}
  \caption{\scriptsize{Long term $E_p$ behaviour}}
  \label{fig:HighE_passenger_long}
\end{subfigure}
\caption{\scriptsize{Energy behaviour of the different subsystems for the rotating regime ($E_T = 4$) with $\epsilon = 1/2$ for the Lobatto-2 method.}}
\label{fig:HighE}
\end{figure}

\subsection{A Fully Chaotic nonholonomic System}
Our configuration manifold in this case is $Q = \mathbb{R}^n$, with $n = 2 m + 1$ and $m \geq 2$. The corresponding Lagrangian and constraint functions for this system are (see \cite{perlmutter06})
\begin{gather*}
L(\vec{\mathbf{q}},\vec{\mathbf{v}}) = \frac{1}{2}\left\Vert \vec{\mathbf{v}}\right\Vert_{2}^2 - \frac{1}{2}\left(\left\Vert \vec{\mathbf{q}}\right\Vert_{2}^2 + q^2_{m+2}q_{m+3}^2 + \sum_{i = 1}^m q_{1 + i}^2 q_{m + 1 + i}^2 \right),\\
\Phi(\vec{\mathbf{q}},\vec{\mathbf{v}}) = v_1 + \sum_{i = m + 2}^n q_i v_i.
\end{gather*}

This is a chaotic system displaying some strange behaviour. As $\Phi$ is linear in the velocities, the continuous system must preserve energy and one would expect the discrete system to neatly oscillate around that energy. Numerical results show otherwise, where the energy seems to perform a random walk and its mean squared error for ensembles of initial conditions on the same energy sheet appears to grow with time.

We performed numerical tests following those of \cite{Jay2009}, where $m = 3$ ($n = 7$) and ensembles of initial conditions with $E_0 = 3.06$,
\begin{equation*}
\vec{\mathbf{q}}_{0}(j,J) = \left( \alpha(j,J), 0.6, 0.4, 0.2, 1, 1, 1 \right), \quad \vec{\mathbf{v}}_{0}(j,J) = \left( 0, \beta(j,J), 0, 0, 0, 0, 0 \right),
\end{equation*}
where $\alpha(j,J) = \cos(j \pi/ (2J))$, $\beta(j,J) = \sin(j \pi/ (2J))$ and $j = 0, ..., J$.

The mean squared error of the energy at the $k$-th step is defined as:
\begin{equation}
\mu(E,k) = \frac{1}{J + 1} \sum_{j = 0}^J \left(E_{j k} - E_0\right)^2
\end{equation}
where $E_{j k}$ is the energy of the particle corresponding to the $j$-th initial condition measured at time step $k$.

\begin{figure}[h!]
\centering
  \includegraphics[width=\textwidth, clip=true, trim=25mm 0mm 25mm 0mm]{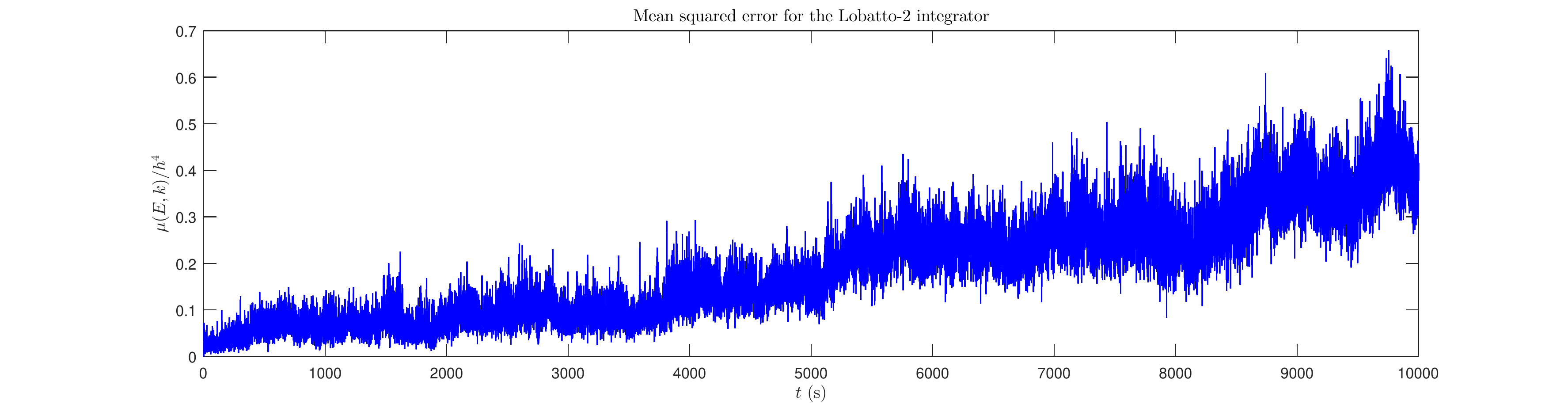}
  \includegraphics[width=\textwidth, clip=true, trim=25mm 0mm 25mm 0mm]{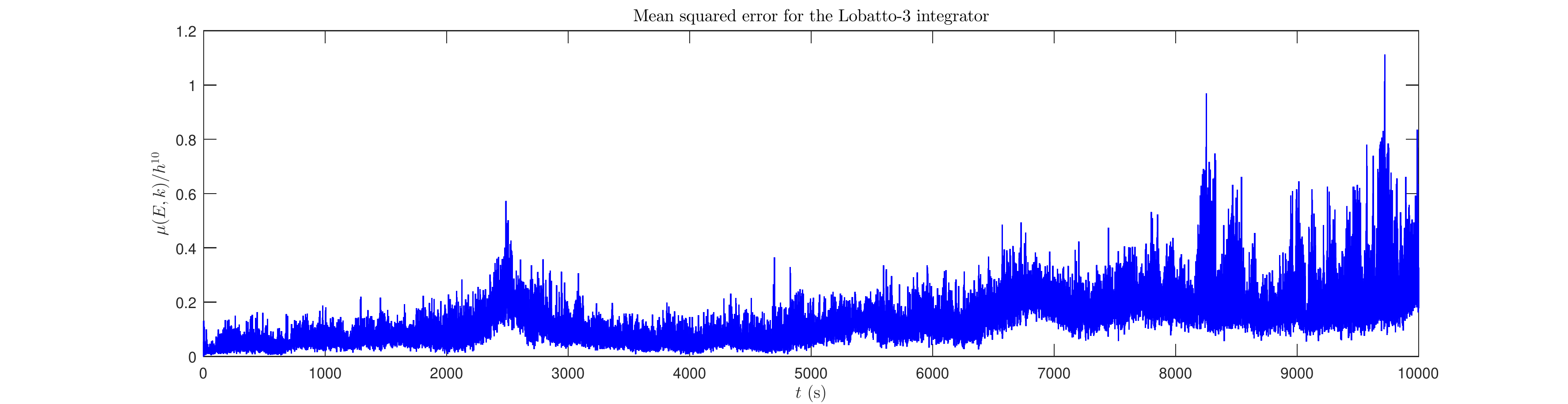}
\caption{\scriptsize{$\mu(E,k)/h^4$ and $\mu(E,k)/h^{10}$ behaviour with time for the Lobatto-2 and Lobatto-3 methods respectively.}}
\label{fig:perlmutter_mclachlan}
\end{figure}

Our integrator matches the behaviour of other non-energy preserving integrators such as SPARK or DLA with no apparent gain over any of these, but no loss either.

\subsection{Nonholonomic vertical disc (unicycle) and elastic spring}
As a first example of our integrators in the Lie group setting we consider the simple example of a vertical disc subjected to a harmonic potential which can be thought of as an elastic spring binding it to the origin. In this case $Q = SE(2)$ and the Lagrangian and constraint functions are:
\begin{gather*}
L(x, y, \theta, v_x, v_y, v_{\theta}) = \frac{1}{2}\left[m (v_x^2 + v_y^2) + I_z v_{\theta}^2\right] - \frac{1}{2}(x^2 + y^2),\\
\Phi(x, y, \theta, v_x, v_y, v_{\theta}) = v_y \cos \theta - v_x \sin \theta.
\end{gather*}

These can be left-trivialized so that our velocity phase space becomes $SE(2) \times \mathfrak{se}(2)$:
\begin{gather*}
\ell(x, y, \theta, v_1, v_2, \omega) = \frac{1}{2}\left[m (v_1^2 + v_2^2) + I_z \omega^2\right] - \frac{1}{2}(x^2 + y^2),\\
\phi(x, y, \theta, v_1, v_2, \omega) = v_2.
\end{gather*}

For the discretization, the $\mathrm{cay}$ map was used. As it can be seen in fig. \ref{fig:uni_Order_plots}, the numerical order obtained coincides with the expected one.
\begin{figure}[!h]
\centering
\includegraphics[width=6.2cm, clip=true, trim=40mm 90mm 40mm 90mm]{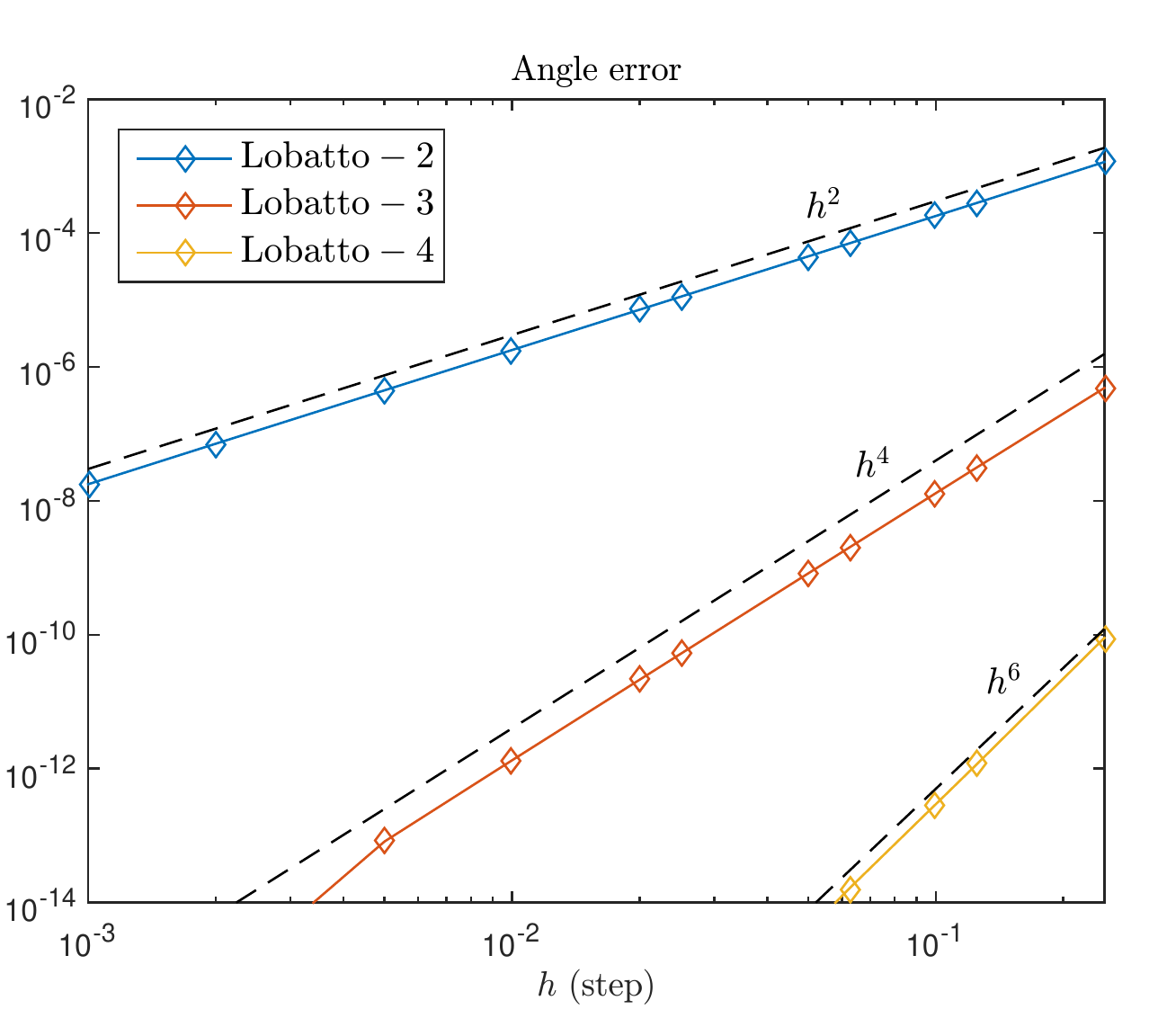}
\begin{subfigure}{.5\textwidth}
  \centering
  \includegraphics[width=6.2cm, clip=true, trim=40mm 90mm 40mm 90mm]{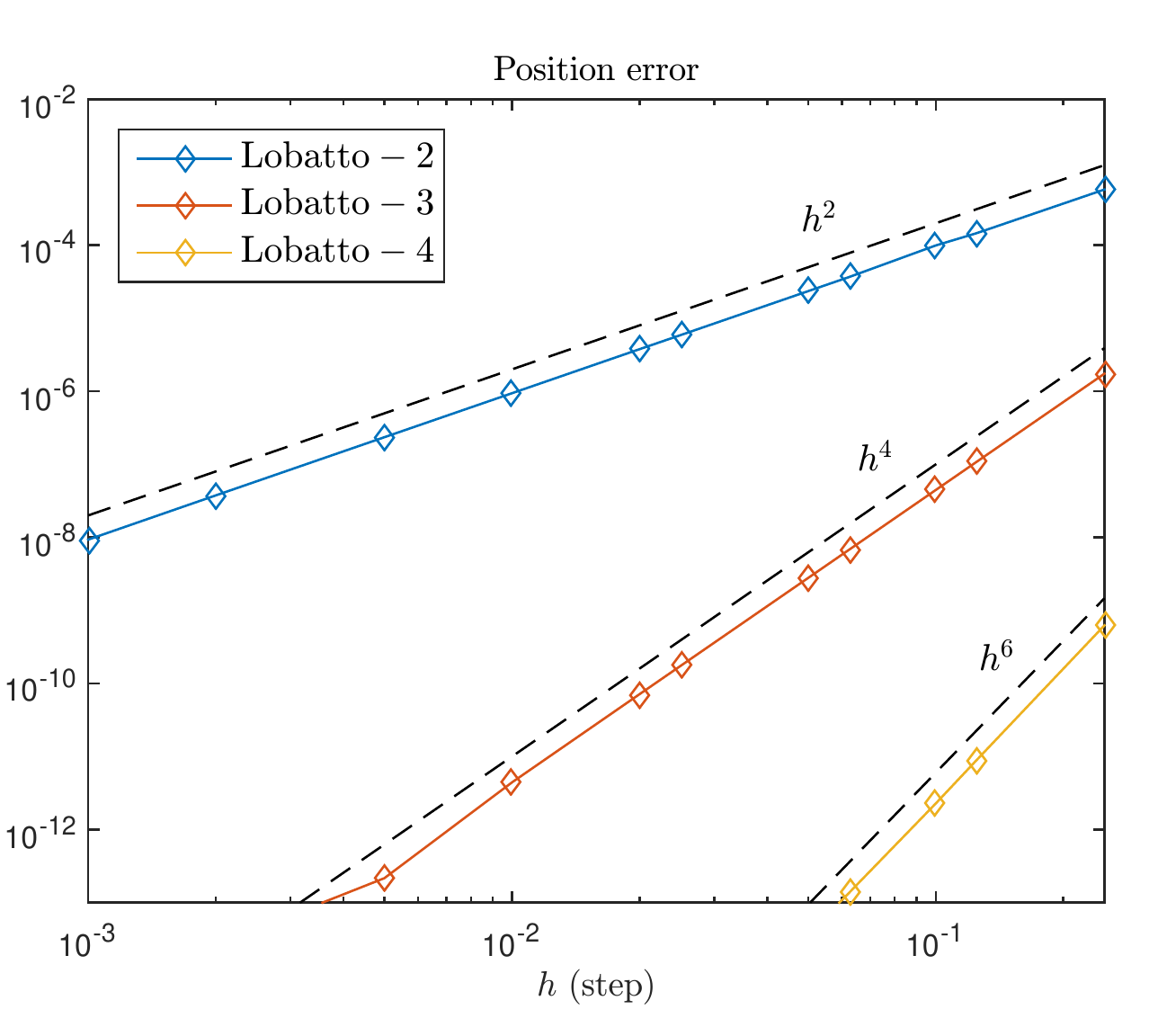}
  \includegraphics[width=6.2cm, clip=true, trim=40mm 90mm 40mm 90mm]{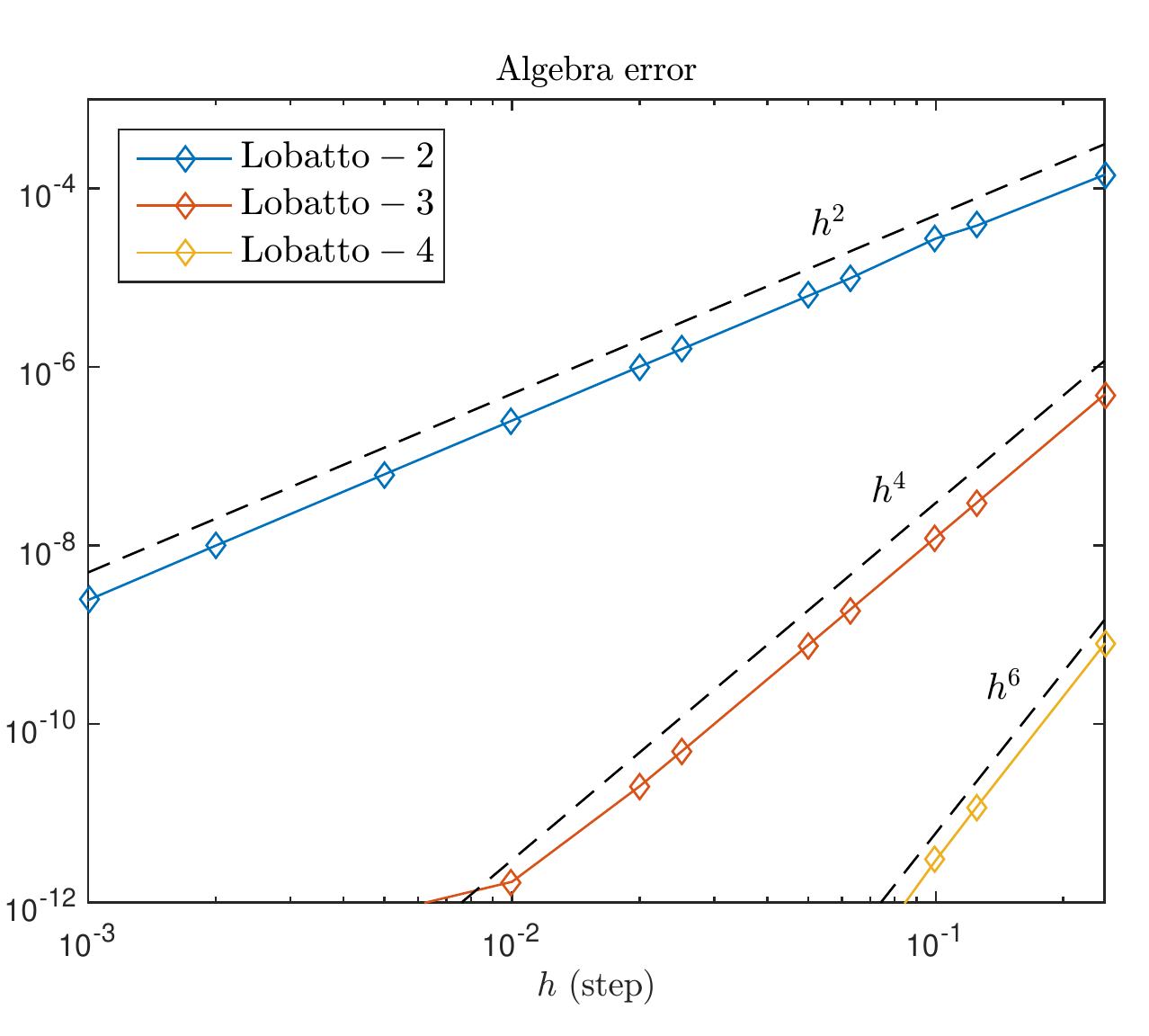}
\end{subfigure}%
\begin{subfigure}{.5\textwidth}
  \centering
  \includegraphics[width=6.2cm, clip=true, trim=40mm 90mm 40mm 90mm]{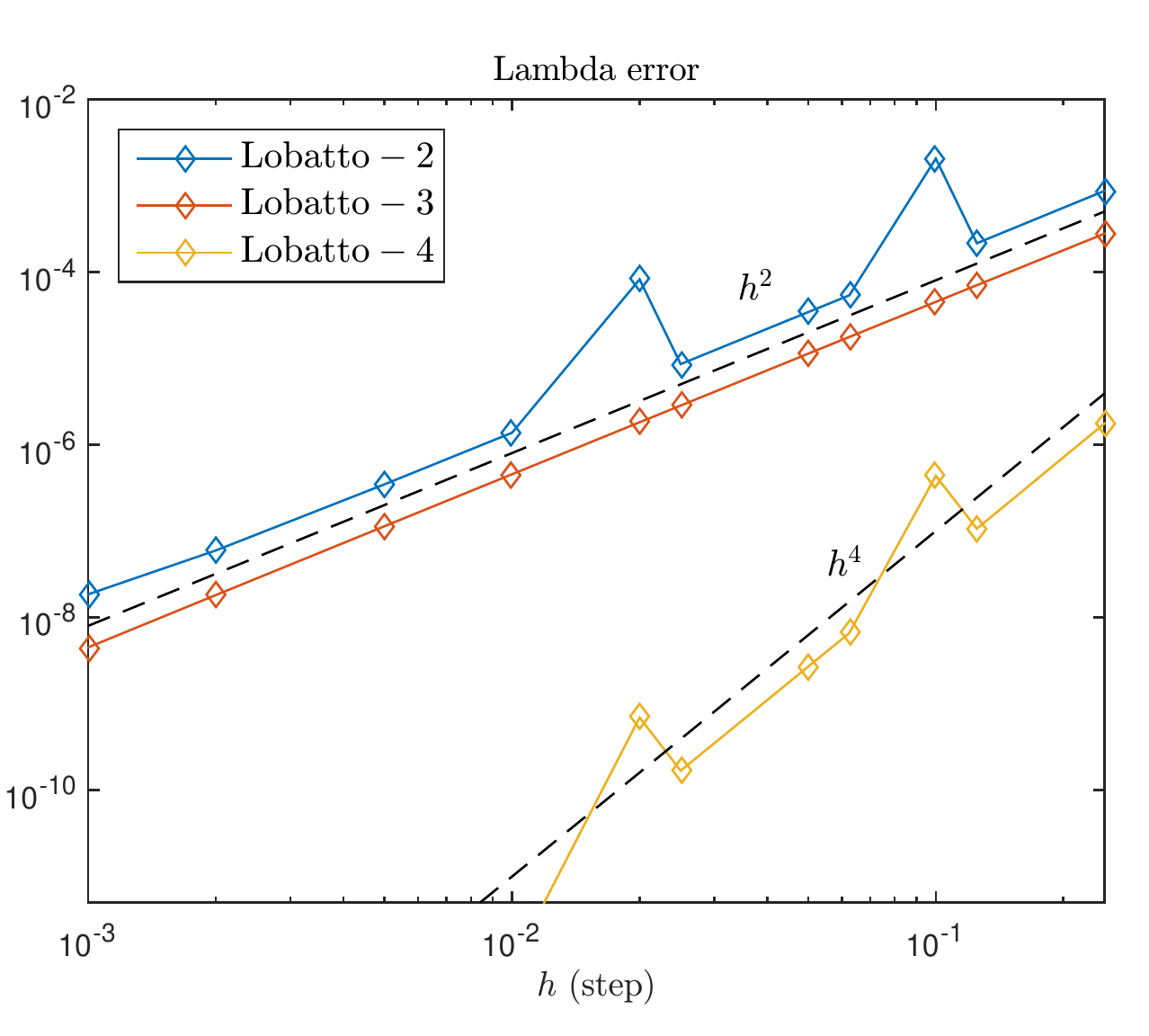}
  \includegraphics[width=6.2cm, clip=true, trim=40mm 90mm 40mm 90mm]{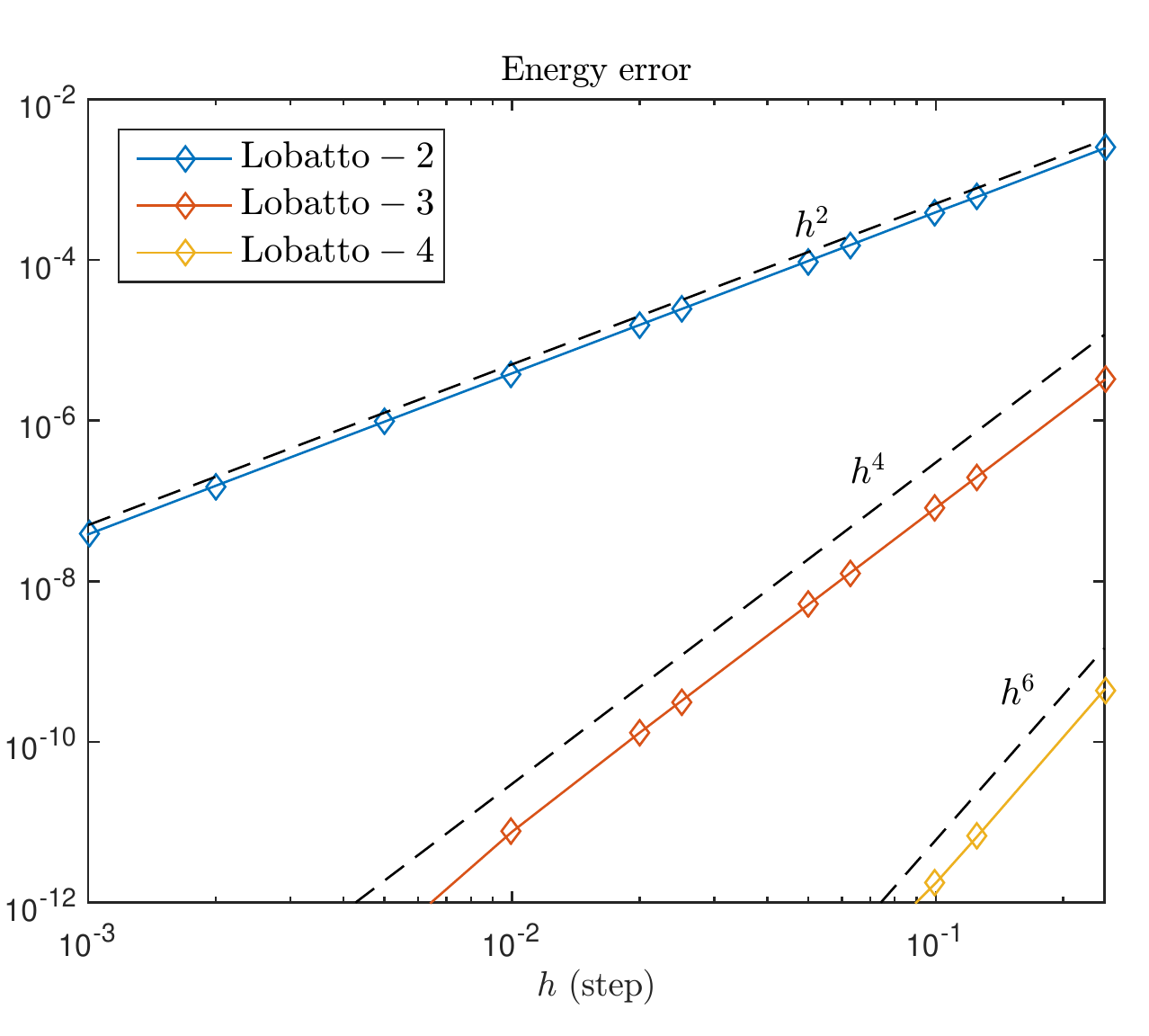}
\end{subfigure}
\caption{\scriptsize{Relative error w.r.t. reference values obtained for $h = \texttt{1e-4}$ for integrators of various orders. As can be seen, the behaviour of the Lagrange multipliers differs from the other variables, as predicted.}}
\label{fig:uni_Order_plots}
\end{figure}

\subsection{Nonholonomic ball on a turntable}
As a second and final example in the Lie group setting we consider the classic example of a ball rolling without slipping on a turntable that rotates with constant angular velocity. In this case $Q = SO(3) \times \mathbb{R}^2$ and the left-trivialized Lagrangian and constraint functions are:
\begin{gather*}
\ell(\phi, \theta, \psi, x, y, \omega_{\xi}, \omega_{\eta}, \omega_{\zeta}, v_x, v_y) = \frac{1}{2} \left(v_x^2 + v_y^2\right) + \frac{r^2}{2} \left(a \omega_{\xi}^2 + b \omega_{\eta}^2 + c \omega_{\zeta}^2\right),\\
\phi_1(\phi, \theta, \psi, x, y, \omega_{\xi}, \omega_{\eta}, \omega_{\zeta}, v_x, v_y) = v_x + \Omega y - r \omega_{\eta},\\
\phi_2(\phi, \theta, \psi, x, y, \omega_{\xi}, \omega_{\eta}, \omega_{\zeta}, v_x, v_y) = v_y - \Omega x + r \omega_{\xi},
\end{gather*}
where $\xi, \eta, \zeta$ are the principal axes of the ball and $a, b$ and $c$ are rescaled moments of inertia.
\begin{figure}[!h]
\centering
\begin{subfigure}{.5\textwidth}
  \centering
  \includegraphics[width=6.25cm, clip=true, trim=38mm 90mm 40mm 90mm]{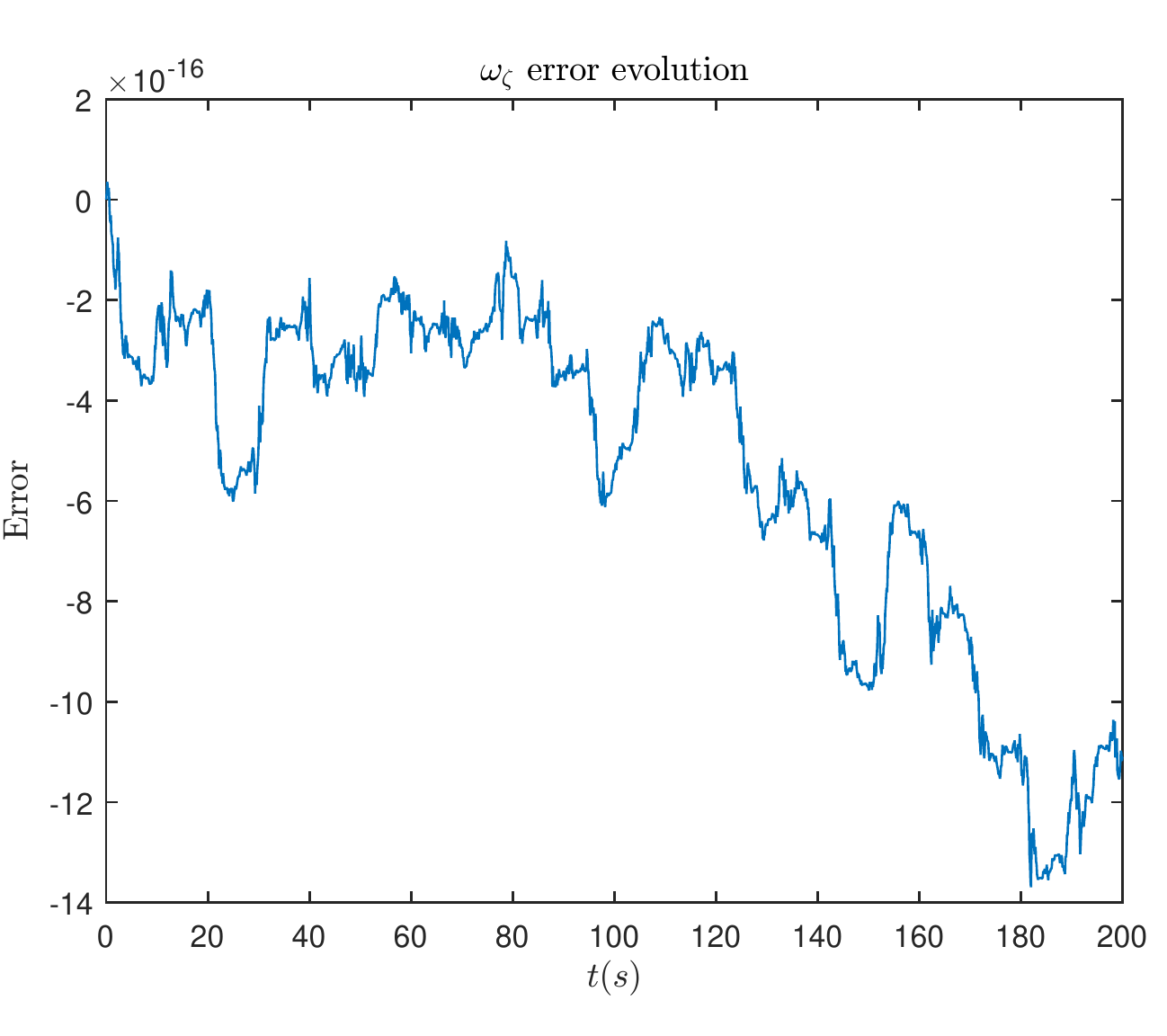}
  \includegraphics[width=6.25cm, clip=true, trim=38mm 90mm 40mm 90mm]{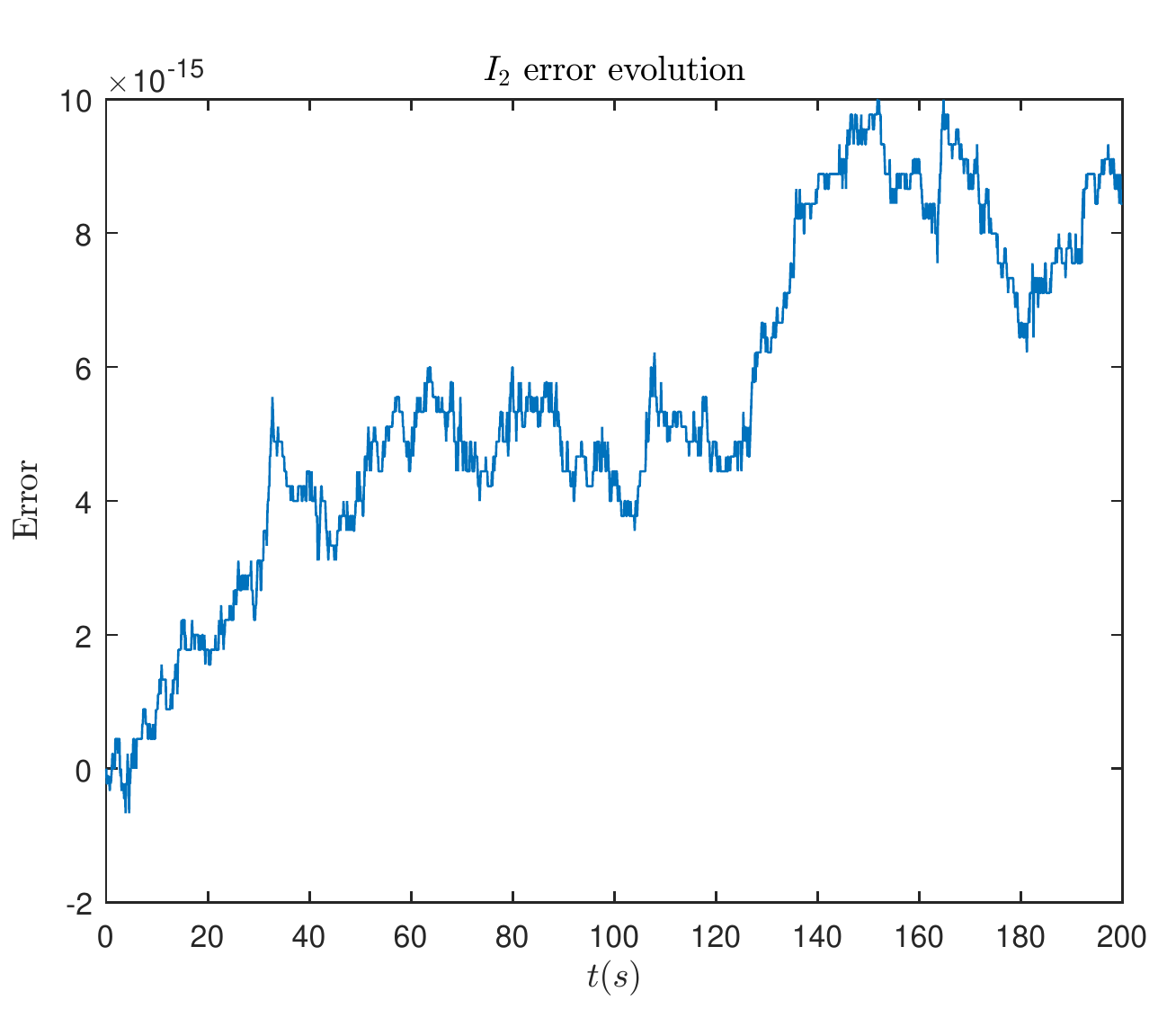}
\end{subfigure}%
\begin{subfigure}{.5\textwidth}
  \centering
  \includegraphics[width=6.25cm, clip=true, trim=38mm 90mm 40mm 90mm]{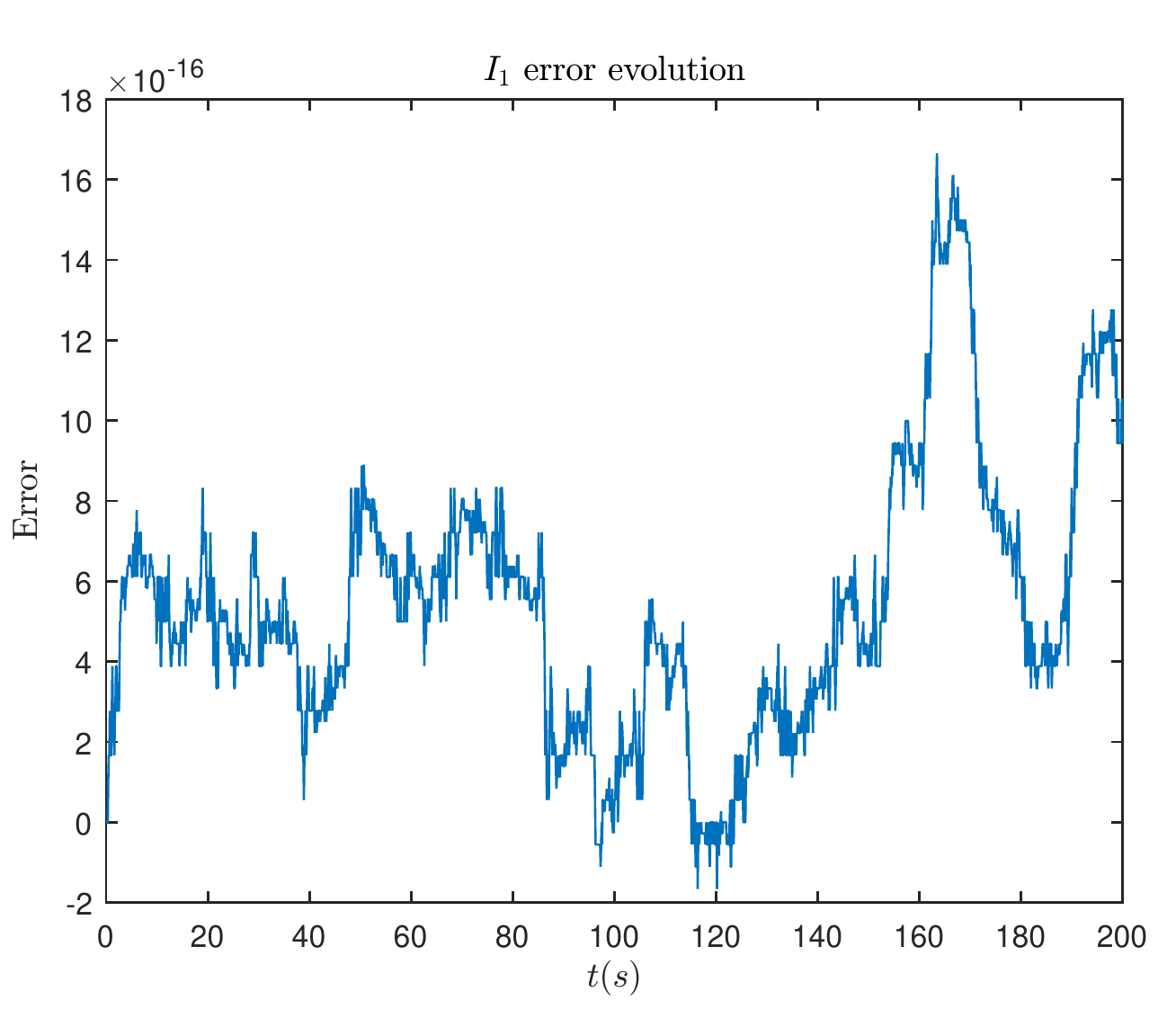}
  \includegraphics[width=6.25cm, clip=true, trim=38mm 90mm 40mm 90mm]{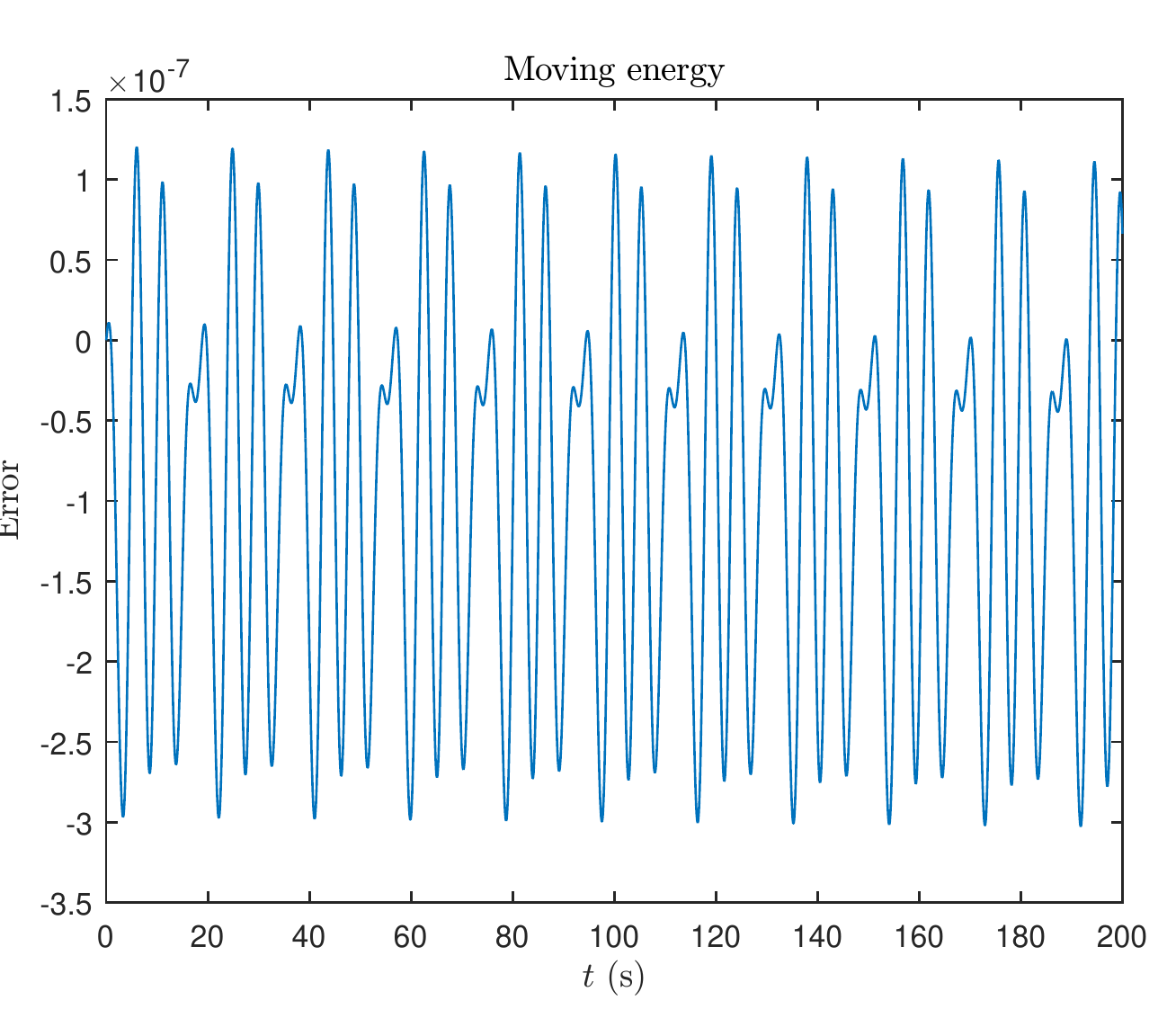}
\end{subfigure}
\caption{\scriptsize{Error evolution of the four first integrals. As noted, for the first three the error is introduced by the solver and is in fact lower than the tolerance set for the chosen one (\texttt{fsolve} in MATLAB with \texttt{TolX = 1e-12}). The behaviour of the moving energy is similar to that of the energy of a regular holonomic system. The simulation corresponds to a Lobatto-3 implementation with $\tau = \mathrm{cay}$.}}
\label{fig:ball_integrals}
\end{figure}

The homogeneous case, where $a = b = c$, is of special interest as it displays periodic motion, implying the existence of $\dim Q - 1 = 4$ first integrals \cite{fasso1,MR3784987}. Following three,
\begin{equation*}
\omega_{\zeta}, \quad r \omega_{\xi} - \frac{\Omega}{1 + a} x, \quad r \omega_{\eta} - \frac{\Omega}{1 + a} y,
\end{equation*}
are preserved by the integrator, but unless carefully implemented the numerical solver will introduce error (see fig. \ref{fig:ball_integrals}). The fourth, dubbed the \emph{moving energy} of the ball,
\begin{equation*}
\frac{1}{2} \left(v_x^2 + v_y^2\right) + \frac{a r^2}{2} \left(\omega_{\xi}^2 + \omega_{\eta}^2 + \omega_{\zeta}^2\right) + r \Omega \left( x \omega_{\xi} + y \omega_{\eta} \right) - \Omega^2\left(x^2 + y^2\right)
\end{equation*}
displays the sort of behaviour one expects from the energy of a holonomic system.
\section{Conclusions and future work}
In this paper we have introduced a new class of high-order geometric numerical integrators for nonholonomic systems defined on a vector bundle and on a Lie group. Moreover, we have tested its performance in some interesting and relevant examples. 

In future papers, we will study its relation to the nonholonomic Hamilton-Jacobi equation and we will explore its extension to other interesting problems related with optimal control theory, nonholonomic systems subjected to additional external forces and also the discretisation of thermodynamical systems.

 \section*{Aknowledgements}
The authors have been partially supported by Ministerio de Econom\'ia, Industria y Competitividad (MINEICO, Spain) under grants MTM 2013-42870-P, MTM 2015-64166-C2-2P, MTM2016-76702-P and ``Severo Ochoa Programme for Centres of Excellence'' in R\&D (SEV-2015-0554).

 \bibliographystyle{unsrt}
 \bibliography{References}

\appendix
\section{Appendix. Brief introduction to the numerical solution of systems of ODEs}\label{sec:app_num_sol}
Assume we are interested in solving the following generic initial value problem (IVP) numerically:
\begin{equation}
\left\lbrace
\begin{array}{rl}
\dot{y}(t) &= f(t,y(t))\\
y(t_0) &= y_0
\end{array}
\right.
\label{eq:ivproblem}
\end{equation}
where $y(t) \in \mathbb{R}^n$.

The exact (or analytic) solution of problem \eqref{eq:ivproblem} is a mapping (a flow) $\Phi: I \times \mathbb{R}^n \to \mathbb{R}^n$, where $I \subseteq \mathbb{R}$, such that $y(t) = \Phi(t) y_0, \forall t \in I$.

A numerical solution of problem \eqref{eq:ivproblem} is a mapping $\Phi_h: y_0 \mapsto y_1$, where $y_1 \approx y(t_0 + h)$. We say the order of approximation of our numerical solution is $p$ if, as $h \to 0$, it satisfies:
\begin{equation*}
y_1 - y(t_0 + h) = \mathcal{O}(h^{p+1})
\end{equation*}

We are going to take an interest in methods based on polynomial interpolation rules and more specifically, in collocation type methods. These methods consist on finding a polynomial whose derivative at certain interpolation nodes (\emph{collocation points}) coincides with the vector field $f$ of the problem (see \cite{MR0280008, MR1387137, HaNoWa93}).

\subsection{Continuous collocation}
An $s$-stage continuous collocation polynomial $u(t)$ (of degree $s$) for problem \eqref{eq:ivproblem} must satisfy (see \cite{hairer}):
\begin{equation}
\left\lbrace
\begin{array}{rl}
u(t_0) &= y_0\\
\dot{u}(t_0 + c_i h) &= f(t_0 + c_i h, u(t_0 + c_i h)), \quad i = 1, ..., s
\end{array}
\right.
\label{eq:continuous_collocation}
\end{equation}
where $c_i$ are distinct real numbers. $s$ is called the \emph{number of steps} of the collocation polynomial. Using Lagrange interpolation the polynomial must be such that:
\begin{equation}
\dot{u}(t) = \sum_{j = 1}^{s} k_j \ell_j(t) = \sum_{j = 1}^{s} f(t_0 + c_j h, u(t_0 + c_j h)) \ell_j(t)
\label{eq:continuous_collocation_polynomial}
\end{equation}
where $\ell_j (t)$ is the $j$-th element of the Lagrange basis of dimension $s$. Thus each element of this basis is a polynomial of degree $s-1$.

The most well-known and widely used methods of this kind are Gau{\ss}, Radau and Lobatto methods. We will be focusing on the latter for reasons that will become clear later.

A Lobatto continuous collocation polynomial has the highest order possible subject to the condition $c_1 = 0, c_s = 1$, i.e. they must include the endpoints as interpolation nodes. Said interpolation nodes are the zeros of the polynomial $x (x - 1) \mathcal{J}^{(1,1)}_{s-2}(2 x - 1)$, where $\mathcal{J}^{(\alpha,\beta)}_{n}(x)$ is a Jacobi polynomial. They are also symmetric, meaning $\Phi_h^{-1} = \Phi_{-h}$, which warrants that their order is even. Its quadrature order is $p = 2 s - 2$, and its lowest order member is the implicit trapezoidal rule.

One of the most interesting features of continuous collocation methods is that they provide us with a continuous approximation of the solution between $t_0$ and $t_1 = t_0 + h$, namely the interpolation polynomial $u(t)$, instead of just a discrete set of points. This polynomial is an approximation of order $s$ to the exact solution (\cite{hairer}, Lemma 1.6, p.33), i.e.:
\begin{equation}
\left\Vert u(t) - y(t)\right\Vert \leq C h^{s+1} \quad \forall t \in [t_0, t_0 + h]
\label{eq:continuous_degree_of_approximation}
\end{equation}
and for sufficiently small $h$.

Moreover, the approximation at quadrature points is of order $p$ (immediate consequence of \cite{hairer}, Theorem 1.5, p.32).

\subsection{Discontinuous collocation}
An $s$-stage discontinuous collocation polynomial $u(t)$ for problem \eqref{eq:ivproblem} is a polynomial of degree $s - 2$ satisfying (see \cite{hairer}):
\begin{equation}
\left\lbrace
\begin{array}{rl}
u(t_0) &= y_0 - h b_1(\dot{u}(t_0) - f(t_0,u(t_0)))\\
\dot{u}(t_0 + c_i h) &= f(t_0 + c_i h, u(t_0 + c_i h)), \quad i = 2, ..., s-1\\
y_1 &= u(t_1) - h b_s(\dot{u}(t_1) - f(t_0,u(t_1)))
\end{array}
\right.
\label{eq:discontinuous_collocation}
\end{equation}
where $t_1 = t_0 + h$, $b_1, b_s$ and $c_i$ are distinct real numbers. $s$ is called the \emph{number of steps} of the collocation polynomial. Using Lagrange interpolation the polynomial must be such that:
\begin{equation}
\dot{u}(t) = \sum_{j = 2}^{s-1} k_j \ell_{j-1}(t) = \sum_{j = 2}^{s-1} f(t_0 + c_j h, u(t_0 + c_j h)) \ell_{j-1}(t)
\label{eq:discontinuous_collocation_polynomial}
\end{equation}
where $\ell_j (t)$ is the $j$-th element of the Lagrange basis of dimension $s-2$. Thus each element of this basis is a polynomial of degree $s-3$.

Contrary to continuous methods the generated interpolation polynomial provides a poor continuous approximation of the solution (\cite{hairer}, Lemma 1.10, p.38), i.e.:
\begin{equation}
\left\Vert u(t) - y(t)\right\Vert \leq C h^{s-1} \quad \forall t \in [t_0, t_0 + h],
\label{eq:discontinuous_degree_of_approximation}
\end{equation}
and is better seen as providing a scaffolding from which to build an approximation of $y_1$. Again we consider Lobatto collocation polynomials, again subject to $c_1 = 0, c_s = 1$. These methods still provide an approximation of order $p = 2 s - 2$ for $y_1$ (\cite{hairer}, Theorem 1.9, p.37).

\subsection{Runge-Kutta methods}
The collocation methods discussed above can be seen separately as a particular instance of a Runge-Kutta method, completely defined by a set of coefficients $(a_{i j}, b_i, c_i)$, where $\sum_{j = 1}^s a_{i j} = c_i$. In particular, for continuous collocation methods (\cite{MR0280008,MR1387137}):
\begin{equation}
a_{i j} = \int_{0}^{c_i} \ell_j (\tau) \mathrm{d}\tau \quad b_{i} = \int_{0}^{1} \ell_i (\tau) \mathrm{d}\tau.
\label{eq:RK_coefficients}
\end{equation}

A numerical solution of \eqref{eq:ivproblem} can be found using an $s$-stage Runge-Kutta method with coefficients $(a_{i j}, b_{j})$ leading to:
\begin{equation}
\begin{array}{rl}
y_1 &= y_0 + h \sum_{j = 1}^s b_{j} k_j\\
Y_i &= y_0 + h \sum_{j = 1}^s a_{i j} k_j\\
k_i &= f(Y_i,Z_i)
\end{array}
\label{eq:RK_system}
\end{equation}

In order to analyse the properties of a given Runge-Kutta scheme it is useful to stablish a series of \emph{simplifying assumptions} that it satisfies:
\begin{align}\label{simpli}
&B(p): \sum_{i = 1}^s b_i c_i^{k-1} = \frac{1}{k} \quad \text{for } k = 1,..., p\\
&C(q): \sum_{j = 1}^s a_{i j} c_j^{k-1} = \frac{c_i^k}{k} \quad \text{for } i = 1, ..., s,\; k = 1,..., q\\
&D(r): \sum_{i = 1}^s b_i c_i^{k-1} a_{i j} = \frac{b_j (1 - c_j^k)}{k} \quad \text{for } j = 1, ..., s,\; k = 1,..., r
\end{align}
When referring to these assumptions for a Runge-Kutta method $(\hat{a}_{i j}, \hat{b}_i)$ we will write them as $\widehat{X}(\hat{y})$. 

Lastly there is a function associated to a Runge-Kutta method that we need to define. Consider the linear problem $\dot{y} = \lambda y$, and apply one step of the given method for an initial value $y_0$. The function $\mathcal{R}(z)$ defined by $y_1 = \mathcal{R}(h \lambda) y_0$ is the so-called stability function of the method.

For an arbitrary Runge-Kutta method we have that
\begin{equation*}
\mathcal{R}(z) = 1 + z b (\mathrm{Id} - z A)^{-1} \mathbbm{1},
\end{equation*}
where $A = \left(a_{i j}\right)$, $b = (b_1,...,b_s)$ and $\mathbbm{1} = (1,...,1)^T$. In the particular case of a method satisfying hypothesis \ref{itm:H3} this can be reduced to:
\begin{equation*}
\mathcal{R}(z) = e_s (\mathrm{Id} - z A)^{-1} \mathbbm{1},
\end{equation*}
where $e_i$ denotes an $s$ dimensional row vector whose entries are all zero except for its $i$-th entry which is 1.

\subsubsection{Partitioned Runge-Kutta methods}
Apart from the usual Runge-Kutta methods there exists a slightly more general class of methods called partitioned Runge-Kutta methods. These methods are of special relevance when the ODE system of problem \eqref{eq:ivproblem} can be partitioned, i.e. it has a natural partition of the form:
\begin{equation}
\left\lbrace
\begin{array}{rl}
\dot{y}(t) &= f(y(t),z(t))\\
\dot{z}(t) &= g(y(t),z(t))
\end{array}
\right.
\label{eq:partitioned_problem}
\end{equation}
Such is the case of problems derived from classical mechanics, where the phase space is usually $T^*Q$ which can be naturally partitioned at each point as $Q \times \mathbb{R}^n$.

These methods consist on applying different Runge-Kutta schemes to each part in order to take advantage of the structure of the problem. As such, a partitioned Runge-Kutta method is defined by a pair $(a_{i j}, b_i), (\hat{a}_{i j}, \hat{b}_i)$.

In the realm of mechanics a very important set of partitioned Runge-Kutta methods arises naturally from the application of the discrete Hamilton's principle. These are the so-called symplectic partitioned Runge-Kutta methods which manage to preserve the symplectic structure of the original problem. Such methods satisfy (\cite{hairer}, Theorem 4.6, p.193):
\begin{equation*}
\begin{array}{rlrl}
b_i \hat{a}_{i j} + \hat{b}_j a_{j i} &= b_i \hat{b}_j, & i, j &= 1,...,s\\
b_i &= \hat{b}_i, & i &= 1,...,s
\end{array}
\end{equation*}
Clearly each of the methods need not be symplectic in order for the partitioned method to be overall symplectic.

Note that if $(a_{i j}, b_i)$ and $(\hat{a}_{i j}, \hat{b}_i)$ are two symplectic conjugated methods, each satisfying the symplifying assumptions $B(p), C(q), D(r)$ and $\widehat{B}(\hat{p}), \widehat{C}(\hat{q}), \widehat{D}(\hat{r})$ then $\hat{p} = p$, $C(q)$ implies $\hat{r} = q$, and conversely $D(r)$ implies $\hat{q} = r$.

Apart from these, there are a few more simplifying assumptions that pairs of compatible methods satisfy (see \cite{Jay1996}):
\begin{align*}
&C\widehat{C}(Q): \sum_{j = 1}^s \sum_{l = 1}^s a_{i j} \hat{a}_{j l} c_l^{k-2} = \frac{c_i^k}{k (k - 1)} \quad \text{for } i = 1, ..., s,\; k = 2,..., Q\\
&D\widehat{D}(R): \sum_{i = 1}^s \sum_{j = 1}^s b_i c_i^{k-2} a_{i j} \hat{a}_{j l} = \frac{b_l}{k (k - 1)} \left[(k - 1) - (k c_l - c_l^k)\right]
&\tag*{for  $l = 1, ..., s$, $k = 2,..., R$\hspace{1.2cm}}
\end{align*}
\begin{align*}
&\widehat{C} C(\hat{Q}): \sum_{j = 1}^s \sum_{l = 1}^s \hat{a}_{i j} a_{j l} c_l^{k-2} = \frac{c_i^k}{k (k - 1)} \quad \text{for } i = 1, ..., s,\; k = 2,..., \hat{Q}\\
&\widehat{D} D(\hat{R}): \sum_{i = 1}^s \sum_{j = 1}^s \hat{b}_i c_i^{k-2} \hat{a}_{i j} a_{j l} = \frac{\hat{b}_l}{k (k - 1)} \left[(k - 1) - (k c_l - c_l^k)\right]
&\tag*{for  $l = 1, ..., s$, $k = 2,..., \hat{R}$\hspace{1.2cm}}
\end{align*}

It can be shown that if both methods are symplectic conjugated then, $Q = R = p - r$ and $\hat{Q} = \hat{R} = p - q$. In particular, Lobatto methods, which will be very important for us, satisfy $B(2 s - 2)$, $C(s)$, $D(s-2)$, $\widehat{B}(2 s - 2)$, $\widehat{C}(s-2)$, $\widehat{D}(s)$, as well as $C\widehat{C}(s), D\widehat{D}(s), \widehat{C}C(s-2), \widehat{D}D(s-2)$.

\section{Appendix. Brief introduction to order conditions}\label{sec:app_ord_conds}
Order conditions for a Runge-Kutta type method are derived by comparing the Taylor series of the exact solution of \eqref{eq:ivproblem} with the solution obtained via our numerical method. This spans a very rich theory developed by Butcher and others during the second half of last century, using tools such as rooted trees, Hopf algebras and group theory, but we will not be delving so deeply. For the interested reader, refer to books such us \cite{hairer} for an in-depth review.

Focusing on the autonomous case, the idea is to consider the exact solution $y(t), t \in [t_0, t_0 + h]$ of:
\begin{equation}
\left\lbrace
\begin{array}{rl}
\dot{y}(t) &= f(y(t))\\
y(t_0) &= y_0
\end{array}
\right.
\label{eq:autonomous_ivproblem}
\end{equation}

Our main goal is to compute the Taylor expansion of $y(t_0 + h)$ in powers of $h$ and compare same order terms. We can compute higher derivatives by inserting our solution in $f$ and recursively using the chain rule as:
\begin{align*}
\dot{y} &= f(y)\\
\ddot{y} &= f'(y)\dot{y}\\
{y}^{(3)} &= f''(y)(\dot{y},\dot{y}) + f'(y)\ddot{y}\\
&...
\end{align*}
after which we eliminate all derivatives from the right-hand side, starting from the top, by inserting the preceding formulas:
\begin{align*}
\dot{y} &= f\\
\ddot{y} &= f' f\\
{y}^{(3)} &= f''(f,f) + f' f' f\\
&...\nonumber
\end{align*}

Each term on the right-hand side has a rooted-tree representation, denoted by $F(\tau)$ (see Hairer, Lubich, Wanner), and each comes multiplied by certain combinatorial coefficients $\alpha(\tau)$ (all of them 1 in the examples shown). This leads to a neat and compact expression for all derivatives as:
\begin{equation}
y^{(q)}(t_0) = \sum_{\vert \tau \vert = q} \alpha(\tau) F(\tau) (y_0)
\label{eq:exact_solution_expansion}
\end{equation}

We may also do this sort of expansion with our numerical method:
\begin{equation*}
g_i = h f(u_i)
\end{equation*}
\begin{equation*}
u_i = y_0 + \sum_j a_{i j} g_j, \quad y_1 = y_0 + \sum_j b_j g_j.
\end{equation*}
where we would obtain, using $\left.g_i^{(q)}\right\vert_{h = 0} = q \cdot (f(u_i))^{q-1}$:
\begin{align*}
\dot{g}_i &= 1 \cdot (f(y_0))\\
\ddot{g}_i &= 2 \cdot (f'(y_0)\dot{u}_i)\\
{g}^{(3)}_i &= 3 \cdot (f''(y_0)(\dot{u}_i,\dot{u}_i) + f'(y)\ddot{u}_i)\\
&...
\end{align*}
Using the derivatives of $u_i^{(q)} = \sum_{j} a_{i j} g_j^{(q)}$ and successively substituting derivatives in the right-hand side again, we are left with expressions:
\begin{align*}
\dot{g}_i &= 1 \cdot (f)\\
\ddot{g}_i &= 2 \cdot \left(\sum_j a_{i j} f' f \right)\\
{g}^{(3)}_i &= 3 \cdot \left( \sum_{j k} a_{i j} a_{i k} f''(f, f) + 2 \cdot \sum_{j k} a_{i j} a_{j k} f' f' f\right)\\
&...
\end{align*}

Again this leads to compact expressions:
\begin{align}
\left.u_i^{(q)}(t_0)\right\vert_{h = 0} &= \sum_{\vert \tau \vert = q} \gamma(\tau) \cdot \mathbf{u}_i(\tau) \cdot \alpha(\tau) F(\tau) (y_0)\label{eq:inner_numerical_solution_expansion}\\
\left.g_i^{(q)}(t_0)\right\vert_{h = 0} &= \sum_{\vert \tau \vert = q} \gamma(\tau) \cdot \mathbf{g}_i(\tau) \cdot \alpha(\tau) F(\tau) (y_0)\nonumber
\end{align}
where $\mathbf{u}_i, \mathbf{g}_i$ are the factors containing terms in $a_{i j}$, and $\gamma(\tau)$ are the integer coefficients appearing in each term.

Finally, using the notation:
\begin{equation*}
\phi(\tau) = \sum_i b_i \mathbf{g}_i(\tau)
\end{equation*}
the derivatives of the numerical solution become:
\begin{equation}
\left.y_1^{(q)}\right\vert_{h = 0} = \sum_{\vert \tau \vert = q} \gamma(\tau) \cdot \phi(\tau) \cdot \alpha(\tau) F(\tau) (y_0)
\label{eq:numerical_solution_expansion}
\end{equation}

Clearly, (\cite{hairer}, Theorem III.1.5, p.56) the method will be of order $p$ iff:
\begin{equation}
\phi(\tau) = \frac{1}{\gamma(\tau)} \quad \text{for }\vert \tau\vert \leq p.
\label{eq:order_condition}
\end{equation}

Something similar can be done for inner points of a method. It is easy to see that if we consider $y(t) = y(t_0 + c h), c \in [0,1]$ then \eqref{eq:exact_solution_expansion} still holds with slight modifications:
\begin{equation}
y_{[i]}^{(q)}(t_0) = c_i^q \cdot \left( \sum_{\vert \tau \vert = q} \alpha(\tau) F(\tau) (y_0) \right)
\label{eq:inner_exact_solution_expansion}
\end{equation}

\begin{proposition}
The order of approximation of the $i$-th inner node, $y_{[i]}$, of a method will be $k$ iff:
\begin{equation}
\mathbf{u}_i(\tau) = \frac{c_i^k}{\gamma(\tau)}, \quad \text{for }\vert \tau\vert \leq k;
\label{eq:inner_order_condition}
\end{equation}
with $\mathbf{u}_i(\tau) = \sum_j a_{i j} \mathbf{g}_j(\tau)$
\end{proposition}

\begin{proof}
Direct comparison of each order in \eqref{eq:inner_exact_solution_expansion} and \eqref{eq:inner_numerical_solution_expansion} shows sufficiency. As in (\cite{hairer}, Theorem III.1.5, p.56), necessity comes from independence of each $F(\tau)$.
\end{proof}

These results (\eqref{eq:inner_order_condition} and \eqref{eq:order_condition}) are directly related to the so-called \emph{simplifying assumptions} satisfied by collocation methods:
\begin{equation}\label{simplifying}
\begin{array}{rl}
B(p): & \sum_{i = 1}^s b_i c_i^{k-1} = \dfrac{1}{k}, \quad k = 1, ..., p; \;\\
C(q): & \sum_{j = 1}^s a_{i j} c_j^{k-1} = \dfrac{c^k_i}{k}, \quad \forall i; \; k = 1, ..., q.
\end{array}
\end{equation}

In particular, for Lobatto methods, the so called Lobatto IIIA (continuous) satisfies $B(2 s - 2), C(s)$, whereas, Lobatto IIIB (discontinuous) satisfies $B(2 s - 2), C(s - 2)$. This means that the inner values calculated by the former method are an approximation of order $s$ to the exact values, while they are of order $s - 2$ for the latter. Still, the order of approximation of $y_1$ is $2 s - 2$ for both, i.e., their quadrature order.

\begin{proposition}\label{prop:continuous_approximation_mixed}
Let $(a_{i j}, b_j)$ and $(\hat{a}_{i j}, \hat{b}_j)$ be the coefficients of the Lobatto IIIA-B pair. Let us solve \eqref{eq:ivproblem}, with $f$ Lipschitz, and denote the resulting interpolation polynomial by $u(t)$. Then for sufficiently small $h$ the polynomial:
\begin{equation*}
v(t) = y_0 + h \sum_{j = 1}^{s} f(t_0 + c_j h, u(t_0 + c_j h)) \int_{t_0}^{t} \ell_j(\tau) \mathrm{d}\tau
\end{equation*}
is an approximation of order $s-1$ of the solution $y(t)$ in the interval $t \in [t_0, t0 + h]$, i.e.,:
\begin{equation*}
\left\Vert v(t) - y(t)\right\Vert \leq C h^{s} \quad \forall t \in [t_0, t_0 + h]
\end{equation*}
Moreover, the derivatives of $u(t)$ satisfy:
\begin{equation*}
\left\Vert v^{(k)}(t) - y^{(k)}(t)\right\Vert \leq C h^{s-k} \quad \forall t \in [t_0, t_0 + h]
\end{equation*}
\end{proposition}

\begin{proof}
Following by example as in \cite{hairer}, Lemma 1.6, and using the same notation, we may express the exact solution as:
\begin{equation*}
\dot{y}(t_0 + \tau h) = y_0 + h \sum_{j = 1}^{s} f(t_0 + c_j h, y(t_0 + c_j h)) \ell_j(\tau) + h^s E(\tau, h),
\end{equation*}

where the interpolation $E(\tau, h)$ is bounded by a constant $M$.

By integration of the difference $\dot{y}(t_0 + \tau h) - \dot{v}(t_0 + \tau h)$ we obtain:
\begin{equation}
y(t_0 + \tau h) - v(t_0 + \tau h) = h \sum_{j = 1}^{s} \delta f_j \int_{0}^{\tau} \ell_j(\sigma) \mathrm{d}\sigma + h^{s+1} \int_{0}^{\tau} E(\sigma, h) \mathrm{d}\sigma
\label{eq:lemma_main}
\end{equation}
where $\delta f_j = f(t_0 + c_j h, y(t_0 + c_j h)) - f(t_0 + c_j h, u(t_0 + c_j h))$.

Now, invoking the result of \cite{hairer}, Lemma 1.10:
\begin{equation*}
\left\Vert u(t) - y(t)\right\Vert \leq C h^{s - 1} \quad \forall t \in [t_0, t_0 + h]
\end{equation*}
we finally get that:
\begin{equation*}
\left\Vert y(t) - v(t)\right\Vert \leq h C \left\Vert y(t) - u(t)\right\Vert + h^{s+1} M \leq C L h^{s} + h^{s+1} M
\end{equation*}

Derivation of \eqref{eq:lemma_main} and further application of the same lemma proves the second statement.
\end{proof}

\section{Appendix. Continuous collocation on Lie groups}\label{sec:app_lie_group_collocation}
Let us begin with a general ODE on the Lie group $G$.
\begin{equation*}
\dot{g}(t) = \Phi(t, g(t))
\end{equation*}
This can be recast into a form which will be easier to work with:
\begin{equation*}
\dot{g}(t) = T_{e} L_{g(t)}\phi(t, g(t)) \equiv g(t) \phi(t, g(t))
\end{equation*}
where $\phi: I \times G \to \mathfrak{g}$, with $I = [t_0, t_0 + h]$. We wish to numerically solve an IVP for this ODE with $g(t_0) = g_0$ using an $s$-stage (continuous) collocation method. This means that we need to find an approximation of $g(t)$ which we will call $u(t)$, such that:
\begin{equation*}
\left\lbrace
\begin{array}{rcl}
u(t_0) & = & g_0\\
\dot{u}(t_0 + c_i h) & = & u(t_0 + c_i h) \phi(t_0 + c_i h, u(t_0 + c_i h)), \quad \forall i = 1, ..., s
\end{array}
\right.
\end{equation*}
where $c_i \in \mathbb{R}$ are the collocation coefficients of the method.

Similar to what we would do in the vector space case we begin with an ansatz of the form:
\begin{equation*}
u(t_0 + \lambda h) = g_0 \tau\left( h \sum_{j}^{s} \eta^j \int_0^\lambda \ell_j(\sigma) \mathrm{d}\sigma\right)
\end{equation*}
with $\ell_i(\sigma)$ the $i$-th element of the $s$-dimensional Lagrange basis associated with the $c_i$ coefficients and $\eta^i \in \mathfrak{g}$. This clearly satisfies the first condition.

Now we need to impose the second condition to find the $\eta$'s. Using $L_j(\lambda) = \int_0^\lambda \ell_j(\sigma) \mathrm{d}\sigma$, we get:
\begin{align*}
\dot{u}(t_0 + \lambda h) &= g_0 D \tau\left( h \sum_{j}^{s} \eta^j L_j(\lambda)\right) \left( \sum_{j}^{s} \eta^j \ell_j(\lambda)\right)\\
&= g_0 \tau\left( h \sum_{j}^{s} \eta^j L_j(\lambda)\right) \mathrm{d}^L \tau_{ h \sum_{j}^{s} \eta^j L_j(\lambda)} \left( \sum_{j}^{s} \eta^j \ell_j(\lambda)\right)\\
&= u(t_0 + \lambda h) \mathrm{d}^L \tau_{ h \sum_{j}^{s} \eta^j L_j(\lambda)} \left( \sum_{j}^{s} \eta^j \ell_j(\lambda)\right),
\end{align*}
thus:
\begin{equation}
\mathrm{d}^L \tau_{ h \sum_{j}^{s} \eta^j L_j(c_i)} \left( \sum_{j}^{s} \eta^j \ell_j(c_i)\right) = \phi\left(t_0 + c_i h, g_0 \tau\left( h \sum_{j}^{s} \eta^j L_j(c_i)\right)\right).
\label{eq:collocation_Lie_group}
\end{equation}

From Guillou \& Soul\'e we have the following relations between Runge-Kutta coefficients and collocation polynomials:
\begin{equation*}
a_{i j} = L_j(c_i) \quad \quad \quad b_{j} = L_j(1)
\end{equation*}
We also know that $\ell_j(c_i) = \delta_{i j}$, so  we can finally express eq.\eqref{eq:collocation_Lie_group} as:
\begin{equation}
\mathrm{d}^L \tau_{ h \sum_{j}^{s} a_{i j} \eta^j} \eta^i = \phi\left(t_0 + c_i h, g_0 \tau\left( h \sum_{j}^{s} a_{i j} \eta^j\right)\right).
\label{eq:Runge-Kutta_Lie_group}
\end{equation}

For its application to variational integrators we are interested in the case where $\phi(t, g(t))$ is the (left-)trivialised velocity of the system. If we introduce some auxiliary variables $\xi^i$ to simplify the expressions, we finally obtain:
\begin{align*}
\xi^i &= h \sum_{j = 1}^{s} a_{i j} \eta^j,\\
G^i &= g_0 \tau(\xi^i),\\
\left(G^i\right)^{-1} V^i &= \mathrm{d}^L \tau_{\xi^i} \eta^i,\\
g_1 &= g_0 \tau\left( h \sum_{j = 1}^{s} b_{j} \eta^j \right).
\end{align*}

\section{Summary: Discrete nonholonomic Lagrangian equations. Vector space}
\begin{equation*}
(q_k, v_k, \lambda_k = \Lambda_k^1) \in \left. TQ \right\vert_{N} \times \Lambda \mapsto (q_{k+1}, v_{k+1}, \lambda_{k+1} = \Lambda_k^s) \in \left. TQ \right\vert_{N} \times \Lambda
\end{equation*}
\begin{center}
	\begin{alignat*}{2}
		q_{k+1} & = q_k + h \sum_{i = 1}^{s} b_{i} V_k^i, & p_{k+1} & = p_k + h \sum_{i = 1}^{s} \hat{b}_{i} W_k^i,\\
		Q_k^i & = q_k + h \sum_{j = 1}^{s} a_{i j} V_k^j, & P_k^i & = p_k + h \sum_{j = 1}^{s} \hat{a}_{i j} W_k^j,\\
		W_k^i & = D_1 L(Q_k^i, V_k^i) + \left\langle\Lambda_k^i, D_2 \Phi(Q_k^i, V_k^i)\right\rangle \vphantom{\sum_{j = 1}^{s}}, & P_k^i & = D_2 L(Q_k^i, V_k^i) \vphantom{\sum_{j = 1}^{s}},\\
		q_k^i & = Q_k^i \vphantom{\sum_{j = 1}^{s}}, & p_k^i & = p_k + h \sum_{j = 1}^{s} a_{i j} W_k^j \vphantom{\sum_{j = 1}^{s}},\\
		p_k^i & = D_2 L(q_k^i, v_k^i), & p_k & = D_2 L(q_k, v_k) \vphantom{\sum_{j = 1}^{s}}
	\end{alignat*}
	\begin{equation*}
		\Phi(q_k^i, v_k^i) = 0 \vphantom{\sum_{j = 1}^{s}}.
	\end{equation*}
\end{center}

\section{Summary: Discrete nonholonomic Hamiltonian equations. Vector space}
\begin{equation*}
(g_k, p_k, \lambda_k = \Lambda_k^1) \in \left. T^*Q \right\vert_{M} \times \Lambda \mapsto (g_{k+1}, p_{k+1}, \lambda_{k+1} = \Lambda_k^s) \in \left. T^*Q \right\vert_{M} \times \Lambda
\end{equation*}
\begin{center}
	\begin{alignat*}{2}
		q_{k+1} & = q_k + h \sum_{i = 1}^{s} b_{i} V_k^i, & p_{k+1} & = p_k - h \sum_{i = 1}^{s} \hat{b}_{i} W_k^i,\\
		Q_k^i & = q_k + h \sum_{j = 1}^{s} a_{i j} V_k^j, & P_k^i & = p_k - h \sum_{j = 1}^{s} \hat{a}_{i j} W_k^j,\\
		V_k^i & = D_2 H(Q_k^i, P_k^i), & W_k^i & = D_1 H(Q_k^i, P_k^i) - \left\langle \Lambda_k^i, \flat_H\left( D_2 \Psi\right)(Q_k^i, P_k^i)\right\rangle , \vphantom{\sum_{j = 1}^{s}}\\
		q_k^i & = Q_k^i \vphantom{\sum_{j = 1}^{s}}, & p_k^i & = p_k - h \sum_{j = 1}^{s} a_{i j} W_k^j \vphantom{\sum_{j = 1}^{s}},
	\end{alignat*}
	\begin{equation*}
		\Psi(q_k^i, p_k^i) = 0 \vphantom{\sum_{j = 1}^{s}}.
	\end{equation*}
\end{center}

\section{Summary: Discrete nonholonomic Lagrangian equations. Lie group}
\begin{equation*}
(g_k, \eta_k, \lambda_k = \Lambda_k^1) \in \left. G \times \mathfrak{g}\right\vert_{N} \times \Lambda \mapsto (g_{k+1}, \eta_{k+1}, \lambda_{k+1} = \Lambda_k^s) \in \left. G \times \mathfrak{g}\right\vert_{N} \times \Lambda
\end{equation*}
\begin{align*}
\Xi_k^i &= \tau^{-1}\left(g_k^{-1} G_k^i\right) = h \sum_{j = 1}^s a_{i j} \mathrm{H}_k^j = \xi_{k}^i,\\
\xi_{k,k+1} &= \tau^{-1}\left(g_k^{-1} g_{k+1}\right) = h \sum_{j = 1}^s b_{j} \mathrm{H}_k^j,\\
\mathrm{M}_k^i &= \mathrm{Ad}_{\tau(\xi_{k,k+1})}^* \left[ \mu_{k} + h \sum_{j = 1}^s b_j \left( \mathrm{d}^L\tau^{-1}_{-\Xi_k^j} - \frac{a_{j i}}{b_i} \mathrm{d}^{L}\tau^{-1}_{-\xi_{k,k+1}}\right)^* \mathrm{N}_k^j\right],\\
\mu_{k}^i &= \mathrm{Ad}_{\tau(\Xi_{k}^i)}^* \left[ \mu_{k} + h \sum_{j = 1}^s a_{i j} \left(\mathrm{d}^L\tau^{-1}_{-\Xi_k^j}\right)^* \mathrm{N}_k^j\right],\\
\mu_{k+1} &= \mathrm{Ad}_{\tau(\xi_{k,k+1})}^* \left[ \mu_{k} + h \sum_{j = 1}^s b_j \left(\mathrm{d}^L\tau^{-1}_{-\Xi_k^j}\right)^* \mathrm{N}_k^j\right],\\
0 &= \phi\left(g_k \tau(\xi_k^i), \eta_{k}^i\right);
\end{align*}
where
\begin{align*}
\mathrm{N}_k^i &= \left(\mathrm{d}^L\tau_{\Xi_k^i}\right)^* \left[L_{g_k \tau(\Xi_k^i)}^* D_1 \ell\left(g_k \tau(\Xi_k^i), \mathrm{d}^{L}\tau_{\Xi_k^i} \mathrm{H}_k^i\right)\right.\\
& + \left.\left\langle \Lambda_k^i, D_2\phi\left(g_k \tau(\Xi_k^i), \mathrm{d}^L \tau_{\Xi_k^i} \mathrm{H}_k^i \right) \right\rangle\right],\\
\mathrm{M}_k^i &= \left(\mathrm{d}^{L}\tau^{-1}_{\xi_{k,k+1}}\right)^*\left[\mathrm{\Pi}_k^i + h \sum_{j = 1}^s \frac{b_j a_{j i}}{b_i} \left( \mathrm{dd}^{L}\tau_{\Xi_k^j} \right)^* \left( \mathrm{H}_k^j, \mathrm{\Pi}_k^j\right)\right],\\
\mathrm{\Pi}_k^i &= \left(\mathrm{d}^{L}\tau_{\Xi_k^i}\right)^* D_2 \ell\left(g_k \tau(\Xi_k^i), \mathrm{d}^{L}\tau_{\Xi_k^i} \mathrm{H}_k^i\right),\\
\mu_k^i &= D_2 \ell\left(g_k \tau(\xi_{k}^i), \eta_{k}^i\right),\\
\mu_k &= D_2 \ell\left(g_k, \eta_{k}\right).
\end{align*}
\newpage

\section{Summary: Discrete nonholonomic Hamiltonian equations. Lie group}
\begin{equation*}
(g_k, \mu_k, \lambda_k = \Lambda_k^1) \in \left. G \times \mathfrak{g}^*\right\vert_{M} \times \Lambda \mapsto (g_{k+1}, \mu_{k+1}, \lambda_{k+1} = \Lambda_k^s) \in \left. G \times \mathfrak{g}^*\right\vert_{M} \times \Lambda
\end{equation*}
\begin{align*}
\Xi_k^i &= \tau^{-1}\left(g_k^{-1} G_k^i\right) = h \sum_{j = 1}^s a_{i j} \mathrm{H}_k^j = \xi_{k}^i,
\\
\xi_{k,k+1} &= \tau^{-1}\left(g_k^{-1} g_{k+1}\right) = h \sum_{j = 1}^s b_{j} \mathrm{H}_k^j,
\\
\mathrm{M}_k^i &= \mathrm{Ad}_{\tau(\xi_{k,k+1})}^* \left[ \mu_{k} - h \sum_{j = 1}^s b_j \left( \mathrm{d}^L\tau^{-1}_{-\Xi_k^j} - \frac{a_{j i}}{b_i} \mathrm{d}^{L}\tau^{-1}_{-\xi_{k,k+1}}\right)^* \mathrm{N}_k^j\right],
\\
\mu_{k}^i &= \mathrm{Ad}_{\tau(\Xi_{k}^i)}^* \left[ \mu_{k} - h \sum_{j = 1}^s a_{i j} \left(\mathrm{d}^L\tau^{-1}_{-\Xi_k^j}\right)^* \mathrm{N}_k^j\right],\\
\mu_{k+1} &= \mathrm{Ad}_{\tau(\xi_{k,k+1})}^* \left[ \mu_{k} - h \sum_{j = 1}^s b_j \left(\mathrm{d}^L\tau^{-1}_{-\Xi_k^j}\right)^* \mathrm{N}_k^j\right],\\
0 &= \psi\left(g_k \tau(\xi_k^i), \mu_k^i\right);
\end{align*}
where
\begin{align*}
\mathrm{N}_k^i &= \left(\mathrm{d}^L\tau_{\Xi_k^i}\right)^* \left[ L_{g_k \tau(\Xi_k^i)}^* D_1 \mathscr{h}\left(g_k \tau(\Xi_k^i), \left(\mathrm{d}^{L}\tau_{\Xi_k^i}^{-1}\right)^* \mathrm{\Pi}_k^i\right)\right.\\
& - \left.\left\langle \Lambda_k^i, \flat_{\mathscr{h}}\left( D_2 \psi\right)\left(g_k \tau(\Xi_k^i), \left(\mathrm{d}^{L}\tau_{\Xi_k^i}^{-1}\right)^* \mathrm{\Pi}_k^i\right)\right\rangle\right],\\
\mathrm{M}_k^i &= \left(\mathrm{d}^{L}\tau^{-1}_{\xi_{k,k+1}}\right)^*\left[\mathrm{\Pi}_k^i + h \sum_{j = 1}^s \frac{b_j a_{j i}}{b_i} \left( \mathrm{dd}^{L}\tau_{\Xi_k^j} \right)^* \left( \mathrm{H}_k^j, \mathrm{\Pi}_k^j\right)\right],\\
\mathrm{H}_k^i &= \left(\mathrm{d}^{L}\tau_{\Xi_k^i}\right)^* D_2 \mathscr{h}\left(g_k \tau(\Xi_k^i), \left(\mathrm{d}^{L}\tau_{\Xi_k^i}^{-1}\right)^* \mathrm{\Pi}_k^i\right).
\end{align*}

\end{document}